\documentclass[11pt,onecolumn, preprint]{IEEEtran}

\usepackage[usenames]{color}
 \usepackage{verbatim}

\usepackage{latexsym}
\usepackage[cmex10]{amsmath}
\RequirePackage{amsthm,amsmath}
\RequirePackage[numbers]{natbib}
\usepackage{enumerate}

\usepackage{amsmath,amsfonts,amssymb}
\usepackage{color} 
\usepackage[cmex10]{amsmath}
\usepackage{amstext,amssymb, amsfonts, amsthm}
\usepackage{wrapfig}

\usepackage{epsfig}

\usepackage{graphicx}
\usepackage{epstopdf}

 \def\epsy{\varepsilon}

\def\Prob{{\sf P}}

\def\Expect{{\sf E}}

\def\limsup{\mathop{\rm lim{\,}sup}}
\def\liminf{\mathop{\rm lim{\,}inf}}

\def\sq{\hbox{\rlap{$\sqcap$}$\sqcup$}}
\def\qed{\ifmmode\sq\else{\unskip\nobreak\hfil
\penalty50\hskip1em\null\nobreak\hfil\sq
\parfillskip=0pt\finalhyphendemerits=0\endgraf}\fi\medskip}

 \def\FRAC#1#2#3{\genfrac{}{}{}{#1}{#2}{#3}}

\def\ddtp{{\mathchoice{\FRAC{1}{d^{\hbox to 2pt{\rm\tiny +\hss}}}{dt}}%
{\FRAC{1}{d^{\hbox to 2pt{\rm\tiny +\hss}}}{dt}}%
{\FRAC{3}{d^{\hbox to 2pt{\rm\tiny +\hss}}}{dt}}%
{\FRAC{3}{d^{\hbox to 2pt{\rm\tiny +\hss}}}{dt}}}}

\def\half{{\mathchoice{\FRAC{1}{1}{2}}%
{\FRAC{1}{1}{2}}%
{\FRAC{3}{1}{2}}%
{\FRAC{3}{1}{2}}}}

\def\eqdef{\mathbin{:=}}

\newtheorem{theorem}{Theorem}%[theorem0]

\newtheorem{proposition}{Proposition}
\newtheorem{lemma}{Lemma}%[theorem]

\newtheorem{assumption}{Assumption}

\def\Lemma#1{Lemma~\ref{t:#1}}
\def\Proposition#1{Proposition~\ref{t:#1}}
\def\Theorem#1{Theorem~\ref{t:#1}}

\def\Assumption#1{Assumption~\ref{as:#1}}

\def\Section#1{Section~\ref{#1}}

\def\Figure#1{Figure~\ref{f:#1}}

\def\Appendix#1{Appendix~\ref{apx:#1}}

\def\eq#1/{(\ref{e:#1})}

\newcommand{\field}[1]{\mathbb{#1}}

\def\ind{\field{I}}

\def\Re{\field{R}}

\newcounter{rmnum}

\newcounter{anum}

\newcounter{cmnum}
\newenvironment{compactnumerate}{\begin{list}{{\upshape \arabic{cmnum}.}}{\usecounter{cmnum}
\setlength{\leftmargin}{0pt} \setlength{\rightmargin}{0pt}
\setlength{\itemindent}{27pt} }}{\end{list}}

\def\til={{\widetilde =}}

\def\clP{{\cal P}}

\def\bfmq{{\mbox{\protect\boldmath$q$}}}

\def\bfmz{{\mbox{\protect\boldmath$z$}}}

\def\bfmY{{\mbox{\protect\boldmath$Y$}}}

\def\bfmZ{{\mbox{\protect\boldmath$Z$}}}

\newlength{\noteWidth}
 \long\def\notes#1{\ifinner
             {\tiny #1}
             \else
             \marginpar{\parbox[t]{\noteWidth}{\raggedright\tiny #1}}
             \fi}

\def\notes#1{}

\def\Var{\hbox{\sf Var}\,}

\def\pori{\pi}
\def\pip{p}
\def\mu{q}

\def\Ho{H0}
\def\Ha{H1}

\def\Kset{U_m}

\def\bmu{\bar{\mu}}
\def\sameorder{\asymp}

\def\mP{{\mathcal{Q}_n}}
\def\bmP{{\clP_m}}

\def\cJ{I}
\def\Figure#1{Fig.~\ref{#1}}
\def\eqdef{: =}

\def\obeta{\overline{\beta}}
\def\mub{\gamma}
\def\bphi{\mathbf{\phi}}

\def\kstat{S_n^{*}}
\def\tstat{S_n^{\sf W}}
\def\pearson{S_n^{\sf P}}
\def\pstata{S_n^{\sf P0}}
\def\kstata{S_n^{*+}}
\def\kstattilde{\tilde{S}_n^{*}}

\def\kstatatilde{\tilde{S}_n^{*+}}
\def\ptest{\phi^{{\sf  P}}}
\def\ttest{\phi^{{\sf  W}}}
\def\ktest{\phi^{{*}}}  
\def\ktesta{\phi^{{*+}}}
\def\ptesta{\phi^{{\sf  P0}}}

\def\lambdab{\lambda_0}
\def\Rec{\operatorname{Re}}
\renewcommand\Im{\operatorname{Im}}
\def\taul{\bar{\kappa}}
\def\bmPmub{\mathcal{P}^b_m}
\def\supseta{\mathcal{W}_\eta}
\def\prt{\mathbf{A}}
\def\ctau{\check{\tau}}
\def\kinf{\kappa}
\def\sq{\hbox{\rlap{$\sqcap$}$\sqcup$}}
\def\qed{\ifmmode\sq\else{\unskip\nobreak\hfil
\penalty50\hskip1em\null\nobreak\hfil\sq
\parfillskip=0pt\finalhyphendemerits=0\endgraf}\fi\medskip}
\def\smu{\bfmq}

\def\dtv{d_{TV}}

\begin{document}

\title{Generalized Error Exponents \\For Small Sample Universal Hypothesis Testing}
\author{Dayu Huang,~\IEEEmembership{Member,~IEEE} and  Sean
Meyn,~\IEEEmembership{Fellow,~IEEE}\thanks{Dayu Huang conducted this research while a PhD candidate with the Department of Electrical and
Computer Engineering and the Coordinated Science Laboratory,
University of Illinois at Urbana-Champaign, Urbana, IL. Email: dayuhuang@gmail.com.

Sean Meyn is with the Department of Electrical  and Computer Engineering, University of Florida,  Gainesville, FL 32611. Email: meyn@ece.ufl.edu.

Portions of the results presented here were published in abridged form in \cite{huamey12p3261} and \cite{huamey12p344}.
 
}

}
\maketitle
%\runtitle{Error Exponents for Goodness of Fit}

% indicate corresponding author with \corref{}
% \author{\fnms{John} \snm{Smith}\corref{}\ead[label=e1]{smith@foo.com}\thanksref{t1}}
% \thankstext{t1}{Thanks to somebody} 
% \address{line 1\\ line 2\\ printead{e1}}
% \affiliation{Some University}
%\begin{aug}

%\author{\fnms{Dayu} \snm{Huang}\corref{}\thanksref{t1, m1}\ead[label=e1]{dhuang8@illinois.edu}}
%\address{Department of Electrical\\ and Computer Engineering \\ and Coordinated Science Laboratory \\
%University of Illinois at Urbana-Champaign\\
%Urbana, IL 61801\\\printead{e1}}
%\and
%\author{\fnms{Sean} \snm{Meyn}\thanksref{t1, m2}\ead[label=e2]{meyn@ece.ufl.edu}}
%\address{Department of Electrical\\  and Computer Engineering  \\ University of Florida\\ Gainsville, FL 32611\\ \printead{e2}}
%\affiliation{University of Illinois at Urbana-Champaign\thanksmark{m1} and University of Florida \thanksmark{m2}} 
%\runauthor{D. Huang and S. Meyn}

\begin{abstract}
The small sample universal hypothesis testing problem
is investigated in this paper, 
in which   the number of samples $n$ is smaller than the number of possible outcomes $m$. The goal of this work is to find an appropriate criterion to analyze statistical tests in this setting. A suitable model for analysis is the high-dimensional model in which both $n$ and $m$ increase to infinity,
and $n=o(m)$. A new performance criterion based on large deviations analysis is proposed and it generalizes the classical error exponent applicable for large sample problems (in which $m=O(n)$). This generalized error exponent criterion provides insights that are not available from asymptotic
consistency or central limit theorem analysis. The following results are established for the uniform null distribution:

(i) The best achievable probability of error $P_e$ decays as $P_e=\exp\{-(n^2/m) J (1+o(1))\}$ for some $J>0$.  

(ii) A class of tests based on separable statistics, including the coincidence-based test, attains the optimal generalized error exponents. 

(iii) Pearson's chi-square test has a zero generalized error exponent and thus  its probability of error is asymptotically larger than the optimal test.

\end{abstract}

%\kwd{Hypothesis testing}
%\kwd{goodness of fit}
%\kwd{large deviations}
%\kwd{large alphabets}
%\kwd{chi-square test}
%\kwd{error exponents}
%\kwd{separable statistic}
%\kwd{birthday problem}
%\kwd{small sample}
%\kwd{multinomial distribution}
%\end{keyword}

%\end{frontmatter}

% AOS,AOAS: If there are supplements please fill:
%\begin{supplement}[id=suppA]
%  \sname{Supplement A}
%  \stitle{Title}
%  \slink[url]{http://lib.stat.cmu.edu/aoas/???/???}
%  \sdescription{Some text}
%\end{supplement}

\section{Introduction}\label{sec:intro}

%The goodness-of-fit problem is to test whether a sequence of observations comes from a specified distribution. It is one of the fundamental problems in statistical inference, and finds applications in many fields such as signal processing, biomedical research, economics and social studies. 
    
%Throughout the twentieth century, this problem has been studied in numerous papers in the literature, beginning with Pearson's famous chi-square test that was published at the start of the last century \citep{pea00p157}. 

The goal of this paper is to better understand hypothesis testing problems with large but finite observation alphabet.  A motivating example is the following hypothesis testing problem on a \textit{continuous} state space.

Consider a hypothesis testing problem in which an i.i.d.\ sequence $\bfmY_1^n=\{Y_1, \ldots, Y_n\}$ is observed,  with $Y_i \in [0,1]$. There are two hypotheses: Under the null hypothesis $\Ho$, the probability measure induced by $Y_i$ is denoted by $P$. Under the alternative hypothesis $\Ha$, 
it is only known that the probability measure $Q$ induced by $Y_i$ satisfies $Q \in \mathcal{Q}$. 
All of these probability measures are assumed to be  absolutely continuous with respect to the Lebesgue measure on $[0,1]$, and each $Q$ is absolutely continuous with respect to $P$. 
The goal is to design a test $\bphi: [0,1]^n \rightarrow \{0,1\}$ with small probabilities of false alarm and missed detection:
\[P_F\eqdef \Prob_P\{\phi_n(\bfmY_1^n)=1\}, P_M\eqdef \sup_{Q \in \mathcal{Q}}\Prob_Q\{\phi_n(\bfmY_1^n)=0\}.\]

%Consider the following hypothesis testing problem: Consider a measurable space $(\mathbf{Y}, \mathcal{B})$. An i.i.d.\ sequence $\bfmY_1^n=\{Y_1, \ldots, Y_n \in \mathbf{Y}\}$ is observed. Under the null hypothesis, the probability measure $P$ induced by $Y_i$ is known. Under the alternative hypothesis, it is only known that the probability measure $Q$ induced by $Y_i$ satisfies $Q \in \mathcal{Q}$. A test maps the sequence of observations $\bfmY_1^n$ into a binary decision on whether the null hypothesis is true. The goal is to design a test with small probability of false alarm (false positive) $P_F$ and missed detection (false negative) $P_M$, where the probability of false alarm is the probability of favoring the the alternative hypothesis while the null hypothesis is true, and the probability of missed detection is the \emph{worst-case} (over set $\mathcal{Q}$) probability of favoring the null hypothesis while the alternative hypothesis is true. 

We consider a universal hypothesis testing problem, also called goodness of fit, in which the set  $\mathcal{Q}$ takes the following form,
\[
\mathcal{Q}=\{Q: d(Q, P) \geq \epsy\}
\]
where $d$ is a distance function that could change with $n$, and $\epsy>0$. As discussed in \cite{bar89p107}, if the distance function is the total variation distance or any distance function dominating the total variation distance, then there is no test that is asymptotically consistent: i.e. $P_F \rightarrow 0$ and $P_M \rightarrow 0$ as $n \rightarrow \infty$. On the other hand, there is a consistent test if the distance function is the total variation distance defined on a \emph{finite partition} of $[0,1]$: 
 Let 
\begin{equation}\label{e:defprt}
\prt=\{\mathcal{A}_1, \ldots, \mathcal{A}_m\}
\end{equation}
be a partition of $[0,1]$. The total variation distance defined on this partition is given by 
\begin{equation}\label{e:partitiondistance}
d_\prt(Q,P)=\sup_{\mathcal{A} \subset \prt}\{|Q(A)-P(A)|\}.
\end{equation}

As the number of observations $n$ increases, it is desirable for a test to not only have a decreasing probability of error, but also be effective against an increasingly larger alternative set $\mathcal{Q}$. Therefore, we consider a sequence of distance functions defined with increasingly finer partitions. We restrict ourselves to partitions of which the cells have equal probabilities under $P$:
\begin{equation} \label{e:equalPcell}
P(A_j)=1/m \textrm{ for } 1\leq j \leq m.
\end{equation}
%Note that $m=|\prt|$ since we assume in this section that $P$ is positive almost everywhere. 
One reason to consider uniform cells, as argued in \cite{manwal42p306}, is that the total-variation distance based on this partition gives the best possible distinguishability with respect to the Kolmogorov-Smirnov distance: Consider the maximum Kolmogorov-Smirnov distance between the null distribution and any alternative distribution that has zero partition-based total variation distance to the null distribution. Then among any partitions with the same number of cells, the maximum Kolmogorov-Smirnov distance is minimized by the partition with uniform cells. In other words, this partition minimizes $\sup_{Q: d_\prt(Q,P)=0}d_{KS}(Q,P)$ where $d_{KS}$ is the Kolmogorov-Smirnov distance.

The dependence between $n$ and $m$ plays a significant role on test analysis and synthesis: the \emph{small sample} case in which $n/m\rightarrow 0$ has a different nature than the \emph{large sample} case in which $n/m \rightarrow \infty$. In the large sample case, the number of samples per cell increases to infinity, and thus eventually the underlying probability that $Y_i$ falls in each cell of $\prt$ can be estimated. This does not hold for the small sample case in which $m$ increases faster than $n$. The goal of this paper is to find an appropriate analysis criterion for the small sample problem.

\subsection{Related work}
This brief literature review focuses on modes of analysis in prior work, and the asymptotic settings considered. Many of the papers cited address models more general than \eqref{e:equalPcell}.

Examples of tests that can be applied to the problem considered in this paper include Pearson's chi-square test, Generalized Likelihood Ratio Test (GLRT) and the coincidence-based test proposed in \citep{pan08p4750}. The procedure to apply these tests is described in \Section{sec:application2continous}.

Existing results differ in the asymptotic setting considered, which can be roughly classified into three cases: 1) $m$ is fixed; 2) $m$ is increasing and $m=O(n)$; 3) $n=o(m)$ and $m=o(n^2)$. There is no need to consider the case $n=O(\sqrt{m})$ because the converse result (lower-bounds on probability of error) established in  \citep{pan08p4750} indicates that no asymptotically consistent test exists if $n=O(\sqrt{m})$. 

There are three predominant types of analysis: 
\begin{enumerate}
\item Asymptotic consistency / sample complexity analysis: This type of analysis characterizes how fast $m$ can increase with $n$, while still ensuring that $\limsup_{n \rightarrow \infty} P_F<\delta, \limsup_{n \rightarrow \infty} P_M < \delta$ for any small $\delta \geq 0$. 

Finer results on $P_F$ and $P_M$ are obtained in  Central Limit Theorem (CLT) and large deviations analysis.

\item CLT analysis: CLTs are applied to obtain asymptotic approximations of the distributions of the test statistic under both hypotheses. It is usually assumed that $\epsy \rightarrow 0$ as a function of $n$, i.e., the set of alternative distributions becomes closer to the null distribution as $n$ increases. This ensures that the decision boundary of the test is close to both the null distribution and the alternative distributions, so that the probabilities of false alarm and missed detection can be analyzed using CLTs. Under this choice of $\epsy$,  $P_F$ and $P_M$ usually converge to nonzero values.  The results characterize how the limits of $P_F$ and $P_M$ differ for different tests.

\item Large deviations analysis: The normalized limits (or asymptotic expansions) of $\log(P_F(\phi))$ and $\log(P_M(\phi))$ are studied. The distance $\epsy>0$ is held to be a constant. The proper normalization of $\log(P_F(\phi))$ and $\log(P_M(\phi))$ must first be identified, and then the normalized limits are calculated.
\end{enumerate}

The outcomes of CLT and large deviations analysis discussed above are asymptotic limits of probability of error given a specified increasing sequence of number of samples. The performance of two tests can also be compared using the number of samples required to achieve certain probability of error for a pair of null distribution and alternative distribution. Different requirement on the asymptotic behavior of $P_F$ and $P_{M}$ or varying the alternative distribution leads to  different measures of efficiency proposed in Pitman \cite{pit49}, Chernoff \cite{che52p493}, Hodges and Lehmann \cite{hodleh56p324}, and Bahadur \cite{bah60p276}. Methods for calculating the Pitman efficiency using CLT analysis  and calculating Chernoff,  Bahadur, Hodges and Lehmann efficiency using large deviations analysis are summarized in \cite{das08,ser80}. For the large sample case where $m=O(n)$, the connection between the error exponent and Bahadur efficiency is studied in \cite{quirob85p727, harvaj08p321}. For the small sample case $n=o(m)$, the generalized error exponent proposed has a similar connection, which is discussed in \Section{sec:bahadur}.

Consider the case where $m$ is fixed. 
\begin{enumerate}[a)]
%\item It is easy to see that both Pearson's chi-square test and GLRT are asymptotically consistent. 
\item Pearson's chi-square and GLRT statistics are asymptotically distributed as a chi-square distribution whose degree of freedom is $m-1$. These results and their extensions can be found in \citep{wil38p60, wal43p426, che54p573, bil61, hal83p1028, clabar90p453}. 
\item The performance of Pearson's chi-square test and GLRT is analyzed in \citep{hoe65p369} using the large deviations analysis. The following \emph{error exponent} criterion is used to evaluate a test $\bphi$:
\begin{equation}\label{e:defexponentchern}
\begin{aligned}
\cJ_F(\bphi) & \eqdef -\limsup_{n \rightarrow \infty} \frac{1}{n}\log(P_F(\phi_n)),\\
 \cJ_M(\bphi) &\eqdef -\limsup_{n \rightarrow \infty} \frac{1}{n}\log(P_M(\phi_n)).
\end{aligned}
\end{equation}
The GLRT is shown to have \emph{optimal} error exponents while Pearson's chi-square test does not. Our use of the term error exponent follows \citep{csilon71p181}. 

%The large deviations analysis was first used by Chernoff in \citep{che52p493} to study the performance of the GLRT for the binary \emph{simple versus simple} hypothesis testing problem, in which both the null and alternative distribution are fully specified. % Note that the hardness result in this regime is the Chernoff-Stein Lemma (See \citep{covtho06}).
\end{enumerate}

Next consider the case $m=O(n)$. 
\begin{enumerate}[a)]
\item Pearson's chi-square test and GLRT are both asymptotically consistent (For example, see \citep{erm98p589}).
\item Pearson's chi-square statistic and the GLRT statistic both have asymptotically normal distributions. These results and their extensions can be found in  \citep{tum56p117, ste57p237, hol72p137, mor75p165, quirob84pp794, oos85p115,kru01p69}. 

%One important application of this type of result is the problem of optimal choice of cells (or partitions) in Pearson's chi-square test when the observation alphabet is not finite \citep{manwal42p306, oos85p115}. %In this problem, the value of the observed sequence is orginally real-valued. To apply Pearson's chi-square test, the real line is partitioned into several intervals called cells, and the empirical distribution is obtained by counting the number of samples in each cell.  
\item A lower-bound on the best achievable probability of error in CLT analysis is given in \citep{erm98p589}: Under the condition $0 \!<\! \liminf_{n \rightarrow \infty} \!\frac{\epsy}{\sqrt{m}} \!\leq\! \limsup_{n \rightarrow \infty} \!\frac{\epsy}{\sqrt{m}} \!< \!\infty$, Pearson's chi-square test is asymptotically optimal. That is, for any test whose limit of $P_F$ is no larger than that of Pearson's chi-square test, the limit of its $P_M$ is asymptotically no smaller than that of Pearson's chi-square test. This result applies to the range of $m$ satisfying $m=o(n^2)$.
\item An achievability result (a lower-bound on the error exponent) and a complementing converse result (an upper-bound on the error exponent) in the large deviations analysis have been obtained in \citep{bar89p107}: There exists a test for which $P_F$ and $P_M$ both decay \emph{exponentially} fast with respect to $n$, i.e., $I_F$ and $I_M$ defined in \eqref{e:defexponentchern} are both nonzero, if and only if $m=O(n)$. %The GLRT satisfies this exponential decay when $m=o(n)$, and can be modified to work in the case $m\asymp n$. \footnote{The notation $f(n)\asymp g(n)$ is equivalent to $f(n)=O(g(n))$ and $g(n)=O(f(n))$.}
Other large deviations and \emph{moderate deviations} analyses of GLRT and Pearson's chi-square test can be found in \citep{quirob85p727, tus77p385,kal85p1554,ron84p800, kol05p255, sirmirism89p645}
\end{enumerate}

Finally consider the small sample case where $n=o(m)$ and $m=o(n^2)$.
\begin{enumerate}[a)]
\item Pearson's chi-square test is known to be asymptotically consistent \citep{erm98p589}. Two others tests shown to be asymptotically consistent  are the test based on counting pairwise-collisions \citep{golron00ECCC} and the {coincidence-based test} \citep{pan08p4750}. An approach to extend tests designed for uniform cells \eqref{e:equalPcell} to non-uniform cells has been proposed in \cite{batfisforkumrubwhi01p442}. 
 %\footnote{The number of samples required for this collision test is shown to be $O(\sqrt{m}\sf{poly}(\epsy)^{-1}\log(1/\delta)$ in order to achieve probability of error less than $\delta$. This essentially implies that the probability of error is \emph{upper-bounded} by $\exp(\frac{n}{\sqrt{m}}\epsy^2)$.  It is not clear whether the normalization term $\frac{n}{\sqrt{m}}$ is the proper normalization for this test. We also note that $\frac{n}{\sqrt{m}}$ is smaller than the normalization term established in this paper for the optimal tests.}
\item Results on the asymptotic distribution of Pearson's chi-square statistic and the GLRT statistic have been obtained in \citep{med77p1, med78p607}.% The asymptotic distribution could be normal, or the convolution of a normal distribution with a Poisson distribution, depending on assumptions on $\pip$ and $\mP$. %conditions on the distributions. 
\end{enumerate}
To the best of our knowledge, the proper normalization for the large deviations analysis  has not been identified before in the small sample case.\footnote{Combining the upper-bounds on probability of error given in \citep{pan08p4750, batfisforkumrubwhi01p442} with the Chernoff inequality gives a loose upper-bound on the asymptotic probability error and does not yield the proper normalization. } We note that the classical error exponent  is not suitable. 
\subsection{Our contributions}
In this paper, we consider the specific problem where the partition is chosen as \eqref{e:equalPcell} so that the induced null distribution over the cells is uniform. As discussed before, this choice of partition minimizes the radius of the pre-image of the induced null distribution, measured by the Kolmogorov-Smirnov distance. 

The new large deviations framework proposed here is motivated by and analogous to the classical error exponent \eqref{e:defexponentchern} in the large sample case. While the classical error exponent is defined with the normalization $n$, our main results imply that for the small sample problem, the following generalized error exponent is best for asymptotic analysis, defined with respect to the normalization $r(m,n)=n^2/m$:
\begin{equation}\label{e:defexponent}
\begin{aligned}
J_F(\bphi) &\eqdef -\limsup_{n \rightarrow \infty} \frac{1}{r(m,n)}\log(P_F(\phi_n)),\\
 \quad J_M(\bphi)& \eqdef -\limsup_{n \rightarrow \infty} \frac{1}{r(m,n)}\log(P_M(\phi_n)).
\end{aligned}
\end{equation}

The generalized error exponents give the following approximation to the probabilities of false alarm and missed detection:
\begin{equation}\label{e:exponentialdecayGEE}
P_F  \approxeq e^{-r(n,m) J_F}, \quad P_M  \approxeq e^{-r(n,m) J_m}.
\end{equation}
The generalized error exponent provides new insights  that are not available from asymptotic consistency, or CLT analysis. The following results are established:
\begin{enumerate}
\item %The proper scaling used in large deviation analysis for the small sample goodness of fit is identified to be $r(n,m)=n^2/m$. 
The minimum probability of error $P_e\!=\!\max\{P_F,\!P_M\}$, decays as $-\log(P_e)=r(n,m)J(1+o(1))$, where $r(n,m)=n^2/m$ and $J$ is the generalized error exponent for the probability of error. This is applicable not only for the case where the set of alternative distributions is defined by the total variation distance in \eqref{e:partitiondistance}, but also for a broad collection of distance / divergence functions.  

\item  
A class of tests based on \emph{separable
 statistics}, including the coincidence-based test $\ktest$, is shown to achieve the \emph{optimal} pair of generalized error exponents $J_F$ and $J_M$: 
\[J_M(\ktest)=\max\{J_M(\phi): J_F(\phi) \geq J_F(\ktest)\}.\] 
The {exact} formulae for these generalized error exponents are obtained. 

\item The performance of Pearson's chi-square test is worse than the coincidence-based test under the generalized error exponent criterion.

%These numerical results are given in \Section{sec:numerical}.
\end{enumerate}

\subsection{Overview of the approach}
In the large deviations analysis of the large sample problem, a main tool is the concentration of empirical distribution around the underlying distribution, e.g. Sanov's theorem and the method of types. For the small sample problem, the analysis in this paper is based on the concentration of \emph{profile} \cite{orlsanjun04p1469}, defined as the number of symbols that appear $l$ times for any fixed $l$. We focus on small $l$:  $l=0$, $l=1$ and $l=2$. The large deviations of the profile for these values of $l$ are obtained from asymptotic approximations to the log-moment generating function using the Poissonization technique following the literature of  {separable
 statistics}. This leads directly to the performance characterization of the coincidence-based test. 

The converse results in this paper are proved using bounds on likelihood ratio between null and alternative distributions on the decision region, a technique also used in \cite{bar89p107,pan08p4750}. To obtain tight bounds, we use a technique similar to the expurgating method in \cite{gal68}. The distributions used in proving the bounds are constructed using the mixing of indistinguishable distributions method (See e.g. \cite{pan08p4750, kelwagtulvis13p782}). 

%\subsection{Notations}\label{sec:notations}
%%There are two conventions we adopt throughout the paper:
%$o$, $O$ and $\asymp$ notations are used throughout the paper: For $y, x$ as functions of $n$, 
%
%\[y\!=\!o(x)\!\Leftrightarrow\!\! \lim_{n \to \!\infty}\!\frac{y}{x}\!\!=\!0,  y\!=\!O(x)\! \Leftrightarrow \!\limsup_{n \to \infty}\!\frac{|y|}{x}\!\!<\!\!\infty,  y\!\asymp \!x\!\Leftrightarrow\! y\!=\!O(x)\! \textrm{ and } x\!=\!O(y).\!\]
%When these notations appear in an asymptotic approximation that holds for a set of distributions, we mean that the approximation error is uniform, i.e., the convergences in the above definitions are uniform over this set of distributions. 
%
%Throughout the paper, $\nu$ denotes a generic distribution (which could be $\pip$), and $\mu$ a distribution in $\mP$, the set of alternative distributions given in \eqref{e:H1set}. 

% where the subscript signals that the convergence only depends on the parameter $c$:
%\[y=o(x) \Leftrightarrow |y|\leq f_{1,n,m}(c)x, \qquad y=O(x) \leftrightarrow |y| \leq f_2(c)x\]
%where $f_1$ is sequence of functions of $n$ satisfying $\lim_{n \rightarrow \infty}f_{1,n,m}(c)=0$, and $f_2$ is a function of $c$ satisfying $\limsup_{n \rightarrow \infty}f_2(c)<\infty$. We drop the subscript if $f_1$ or $f_2$ do not depend on any other parameters.  

%The notation $\asymp_c$ is defined as 

\subsection{Organization of the paper}
The remainder of the 
paper is organized as follows: The universal hypothesis testing problems and tests are presented in \Section{sec:model}. The main achievability and converse results on generalized error exponents are described in \Section{sec:mainresults}. Extensions of the coincidence-based test are given in \Section{sec:coincidencetest}. Performance characterization of Pearson's chi-square test is given in \Section{sec:pearson}. In \Section{sec:extension}, it is shown that the generalized error exponent criterion is also applicable when the set of alternative distributions is defined using many other distance  functions. Connections to asymptotic relative efficiency and  extensions to more general universal hypothesis testing problem are discussed in \Section{sec:discussion}. The paper is concluded in \Section{sec:conclusion}.

%A test for the sample sample regime can then be evaluated using either the error exponent of the average error given by 
%\[\min\{J_M(\phi,m),J_F(\phi,m)\},\]
%or the following asymptotic Neyman-Pearson criterion: A test $\phi$ is optimal if it solves the following optimization problem 
%\[\max\{J_M(\phi,m): J_F(\phi,m)\geq \eta\}.\]

\section{Models and Preliminaries}\label{sec:model}

Here we introduce a more general model based on a sequence of universal hypothesis testing problems, each with a finite number of outcomes (a finite alphabet). Consider an i.i.d.\ sequence of observations $\bfmZ_1^n\eqdef\{Z_1, \ldots, Z_n\}$ where $Z_i \in [m]\eqdef\{1,2,\ldots, m\}$. Let $\bmP$ denote the collection of probability mass functions (p.m.f.s) on $[m]$. We have two hypotheses: Under the null hypothesis $\Ho$, the p.m.f of $Z_i$ is given by $\pip$, the uniform distribution on $[m]$:
\begin{equation} \label{e:dispi}
\pip_j=1/m \textrm{ for } j \in [m].
\end{equation}
Under the alternative hypothesis $\Ha$, the p.m.f. of $Z_i$ belongs to a set $\mP$ given by
\begin{equation}\label{e:H1set}
\mP \eqdef \{\mu \in \bmP: d(\mu, \pip
) \geq  \epsy\}\end{equation}
where $d$ is taken to be the total variation distance $\dtv$ defined for any pair of p.m.f.s on $[m]$:
\[\dtv(\mu, \pip)=\sup_{B \subseteq [m]}\{|\mu(B)-\pip(B)|\}.\]

A test $\bphi=\{\phi_n\}_{n \geq 1}$ is given by a sequence of binary-valued functions $\phi_n: [m]^n \rightarrow \{0,1\}$. The test decides in favor of $\Ho$ if $\phi_n(\bfmZ_1^n)=0$. The test is required to be powerful against the set $\mP$ of alternative p.m.f.s, and thus its performance is evaluated using the probabilities of false alarm $P_F(\phi_n)$ and worst-case probability of missed detection $P_M(\phi_n)$: 
\[\begin{aligned}
P_F(\phi_n) & \eqdef \Prob_\pip\{\phi_n(\bfmZ_1^n)=1\}, \\P_{M,\mu}(\phi_n)&\eqdef\Prob_\mu\{\phi_n(\bfmZ_1^n)=0\},\\P_M(\phi_n)&\eqdef\sup_{\mu \in \mP}\Prob_\mu\{\phi_n(\bfmZ_1^n)=0\}.\end{aligned}\]

An important class of tests is based on separable statistics~\citep{med77p1}:
This is a test statistic of the form 
\begin{equation}\label{e:sepstats}
S_n=\sum_{j=1}^m f_j(n\Gamma^n_j),
\end{equation}
where  
\begin{equation}\label{e:gammadef}
\Gamma^n_j \eqdef \frac{1}{n}\sum_{i=1}^n \ind\{Z_i=j\}
\end{equation} is the empirical distribution, and $f_j$ is any function that does not depend on $n\Gamma^n$ except via its argument. General theorems on asymptotic distributions and asymptotic moments of separable statistics are available in \citep{med77p1}. Large deviations analysis for the case $m=O(n)$ is given in \citep{ron84p800,kol05p255}. We are not aware of previous general large deviations results for the small sample case where $n=o(m)$. 

In this paper, we focus on two tests based on separable statistics: Pearson's chi-square test~\citep{pea00p157} and the coincidence-based test introduced in \citep{pan08p4750}. 

After normalization, the test statistic of Pearson's chi-square test is given by
\begin{equation}\label{e:defpearson}
\pearson = \frac{n}{m}\sum_{j=1}^m \frac{(n\Gamma^n_j-n\pip_j)^2}{n\pip_j}.
\end{equation}
The test is given by $\ptest_n(\bfmZ_1^n)=\ind\{\pearson \geq \tau_n\}$. When the null distribution $\pip$ is uniform, the test statistic is the $\ell_2$ norm: $\pearson=\ell_2^2(n\Gamma^n, n\pip)$.

The  test statistic of the coincidence-based test is given by,
\begin{equation}\label{e:ktest}
\kstat=-\sum_{j=1}^m \ind\{n\Gamma^n_j=1\}.
\end{equation}
This test statistic $\kstat$ counts the number of symbols in $[m]$ that appear in the sequence exactly \emph{once}. The coincidence-based test is given by $\ktest_n(\bfmZ_1)=\ind\{\kstat \geq \Expect_\pip[\kstat]+\tau_n\}$. The coincidence-based test is applicable only when the null distribution is uniform.

An important difference between $\kstat$ and $\pearson$ is that $f_j$ is bounded in $\kstat$, while this is not true in $\pearson$. In \Section{sec:pearson}, we show that this difference has a significant impact on their  performance.

\subsection{Applications to hypothesis testing problems on a continuous state space}\label{sec:application2continous}
Tests designed for finite-valued observations can be applied to solve a hypothesis testing problem with continuous-valued observations by first partitioning the observation space.
Consider the hypothesis testing problem given in \Section{sec:intro} where the i.i.d.\ sequence of observations $\bfmY_1^n$ satisfies $Y_i \in \mathbf{Y}\eqdef[0,1]$. 
%Consider a measurable space $(\mathbf{Y}, \mathcal{B})$, and let $\bfmY_1^n=\{Y_1, \ldots, Y_n\}$ be an i.i.d.\ sequence of observations with $Y_i \in \mathbf{Y}$. We have two hypotheses:
%\begin{equation}\label{e:continuousH0H1Qn}
%H0: Y_i \sim P, \quad H1: Y_i \sim Q \in \mathcal{Q}
%\end{equation}
To apply a test designed for the finite-valued observations, we start with a partition $\prt$ as given in \eqref{e:defprt}.
%\[\prt=\{A_1, \ldots, A_{m}\}\]
%where $\cup_{1 \leq j \leq m} A_j=\mathbf{Y}$. 
The observation $Y_i$ is mapped to a finite-valued observation via $\mathcal{T}: \mathbf{Y} \rightarrow [m]$: $\mathcal{T}(Y_i)=j$ if $Y_i \in  A_j$. %Then a test defined for finite-valued observations can be applied towards $\{\mathcal{T}(Y_i)\}$.
Assume that the partition is chosen so that the marginal of ${Z_i}$ is uniform under the null hypothesis: $P(A_j)=\frac{1}{m}$. For a test $\bphi$ designed for a discrete uniform null distribution, the corresponding test $\bphi(\{\mathcal{T}(Y_1), \ldots, \mathcal{T}(Y_n)\}	)$ can be applied to the problem with continuous-valued observations. This partition-based approach gives tests that are optimal for the model introduced in Section~\ref{sec:intro}. 
Suppose that  the set of alternative distributions is defined as 
\[\mathcal{Q}=\{Q: d_{\prt}(Q,P) \geq \epsy\}
\] where $d_\prt$ is defined in \eqref{e:partitiondistance}.
Then in terms of the probability of false alarm and worst-case probability of missed detection, without loss of optimality we can restrict our attention to tests whose test statistics take constant value on each cell $A_j$ of the partition. This is exactly the collection of partition-based tests we have described. 

In the hypothesis testing problem given in \Section{sec:intro}, it is assumed that the alternative distribution $Q$ is absolutely continuous with respect to $P$. The partition-based tests are still applicable when the assumption $Q$ is absolutely continuous with respect to $P$ does not hold, provided that the tests for finite-valued observations are designed for a more general model where we allow $\pip$ not to have full support: Instead of \eqref{e:dispi}, let the null distribution $\pip$ be\[
\pip_j=1/k \textrm{ for }  1\leq j \leq k, \pip_j=0 \textrm{ for } k<j\leq m.
\]
The generalized error exponent analysis still applies except the normalization should be $n^2/k$ instead of $n^2/m$.

\section{Generalized Error Exponents}\label{sec:mainresults}

In this section, we describe the main results 
%on the proper normalization for large deviations analysis 
for the small sample universal hypothesis testing problem. The following assumption is imposed throughout:
\begin{assumption}\label{as:largealphabet}
$n=o(m)$ and $m=o(n^2)$.
\end{assumption}

To show that the proper normalization to be used in the definition of generalized error exponent is $n^2/m$, we need to establish:
\begin{enumerate}
\item There is a test for which both generalized error exponents are non-zero. Therefore for any smaller normalization, the generalized error exponent  is infinite for the best possible tests. 
 
\item For any test, at least one of the generalized error exponents is finite. Therefore for any larger normalization the generalized error exponent would be trivially zero for any test. 
\end{enumerate}

These are established in \Theorem{KPERFORMANCE} and \Theorem{converse}. Moreover, these two theorems give precise characterization of the achievable region of $(J_F, J_M)$. This is depicted in \Figure{f:errorregion}. The boundary of the achievable region is given by the following formulae: 
For $\tau \in [0,\taul(\epsy)-1]$, 
\begin{equation}\label{e:EEregion}
\begin{aligned}
&J^*_F(\tau) \eqdef  \sup_{\theta \geq 0} \{\theta\tau-\half\bigl(e^{2\theta}-(1+2\theta)\bigr)\},\\
&J^*_M(\tau) \eqdef \sup_{\theta \geq 0} \{\theta(\taul(\epsy) -1 - \tau)-\half\bigl(e^{-2\theta} -(1-2\theta)\bigr)\taul(\epsy)\},\end{aligned}
\end{equation}
where $\taul: \Re_{+} \rightarrow \Re_{+}$ is the $C^{1}$ function, 
\begin{equation}\label{e:equtaul}
\taul(\epsy)=\left\{\begin{array}{c c}
1+4\epsy^2, & \epsy< 0.5,\\
1+{\epsy}/{(1-\epsy)},& \epsy \geq 0.5.
\end{array}\right.
\end{equation}
%Note that we always have $J_F^*(\tau)< \infty$ and $J_M^*(\tau)<\infty$. 
%For $\tau \in (0, \taul(\epsy)-1)$,  we have $J_F^*(\tau)>0$ and $J_M^*(\tau)>0$. 

\begin{theorem}[Achievability]\label{t:KPERFORMANCE}
The coincidence-based test $\ktest$ achieves the generalized error exponents given in \eqref{e:EEregion}, i.e., for any $\tau \in [0, \taul(\epsy)-1]$, if the sequence of thresholds $\{\tau_n\}$ is chosen so that, 
\begin{equation}\label{e:eqdeftau}
\tau=\lim_{n \rightarrow \infty} m\tau_n/n^2,
\end{equation}
then the coincidence-based test has the generalized error exponents:
\begin{equation}\label{e:Kperformance}
J_F(\ktest)=J^*_F(\tau), \quad J_M(\ktest)=J^*_M(\tau).
\end{equation}
\end{theorem}

%The result holds over all possible sequences $\bdm$ for which \Assumption{largealphabet} holds. The error exponents achieved by the coincidence-based test $\ktest$ given in \Section{sec:achievability}.

%\begin{theorem}[Achievability]\label{t:achievable}
%There exists a test $\phi$ whose generalized error exponents can be lower-bounded as follows:
%\[J_F(\bphi)\geq 2\epsy^4/5, \quad J_M(\bphi)\geq \epsy^4/(1+4\epsy^2).\]
%The bound is uniform over all possible sequences $\bdm$ for which \Assumption{largealphabet} holds.\qed
%\end{theorem}
\begin{theorem}[Converse]\label{t:converse}
%Suppose  \Assumption{piuniform} and \Assumption{largealphabet} hold. 
Consider any $\tau \in [0,\kappa(\epsy)-1]$.
For any test $\phi$ satisfying
\begin{equation}
J_F(\phi)\geq J_F^*(\tau),\nonumber%\sup_{\theta \geq 0} \{\theta\tau-\half(e^{2\theta}-1-2\theta)\},\nonumber
\end{equation}
the following upper-bound on the generalized error exponent of missed detection holds:
\begin{equation}
J_M(\phi)\leq J_M^*(\tau).\nonumber%\sup_{\theta \geq 0} \{\theta(\kappa(\epsy)\!-\!\tau)-\half(e^{-2\theta}\!-1+2\theta)(1+\kappa(\epsy))\}.\nonumber
\end{equation}
\end{theorem}

Compare the results in \Theorem{KPERFORMANCE} and \Theorem{converse} with the asymptotic consistency result in \citep{pan08p4750}, where it is shown that $n=\sqrt{m}$ is the critical point that separates the cases whether a consistent test exists.
The achievability result in the asymptotic consistency an	alysis, which states that there is a consistent test whenever $n=\sqrt{m}$,  follows directly from  \Theorem{KPERFORMANCE}. The converse result in asymptotic consistency also follows  from an intermediate result in the proof of \Theorem{converse}. The fact that $n=\sqrt{m}$ is the critical point is connected to the birthday problem: The number of people needed to have a coincident birthday is approximately $\sqrt{365}$. Similarly, the number of samples needed to have repeated observations is $n=\sqrt{m}$. Without a repeated observation, it is impossible to distinguish between the null and alternative distribution. A refined large-deviation analysis of the coincidence is used in this paper to prove  \Theorem{KPERFORMANCE} and \Theorem{converse}.

\begin{figure}[h!]
\centering
\includegraphics[width=0.45\textwidth]{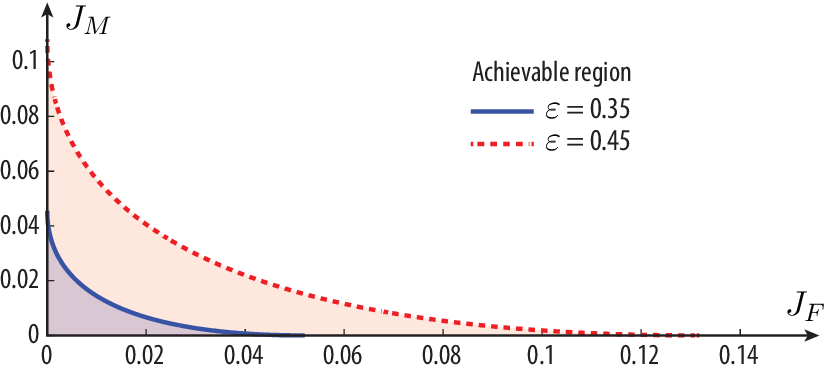}
\caption{Achivable region when $\epsy=0.35$ and $\epsy=0.45$ given by the lower-bound in \Theorem{KPERFORMANCE} and upper-bound in \Theorem{converse}. The lower and upper bound meet over the entire region.}\vspace{-0.1cm}
\label{f:errorregion}
\end{figure}

% \Theorem{KPERFORMANCE} and \Theorem{converse} indicate that the best possible probabilities of false alarm and missed detection decay as in \eqref{e:exponentialdecayGEE}. 
We now compare the approximation in \eqref{e:exponentialdecayGEE} given by the generalized error exponent analysis to the actual empirical performance of the coincidence-based test $\ktest$. The results are shown in \Figure{f:performKep035} for $\epsy=0.35$ and  \Figure{f:performKep045} for $\epsy=0.45$. We choose the threshold $\tau$ based on \eqref{e:Kperformance} so that $J_F$ and $J_M$ are the same. The generalized error exponents are estimates of the \emph{slope} of $\log(P_F)$ and $\log(P_M)$ with respect to $r(n,m)$. It can be observed that the {slope} from the theoretical approximation by generalized error exponents approximately matches the slope of the simulated value. The remaining difference between the theoretical and the empirical slope in \Figure{f:performKep045}  is mainly due to two reasons: First, the threshold chosen is based on the first order approximation. It can be observed from the figure that the slope for $P_M$ is slightly smaller than the predicted slope while the one for $P_F$ is larger. A slightly larger threshold might yield a slope that is  closer to the predicted. Second, the generalized error exponent is only the first term in the asymptotic expansion of $\log(P_F)$ and $\log(P_M)$. Higher order terms might capture the remaining difference. 

\begin{figure}[h!]
\centering
\includegraphics[width=0.48\textwidth]{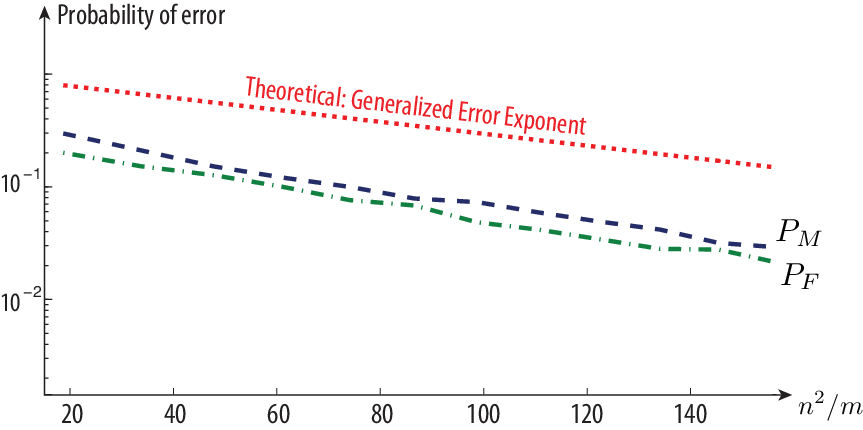}
\caption{ 
Performance of $\ktest$ with $\epsy=0.35$. %Y-axis: $P_M$ and $P_F$, averaged over $10^5$ runs; X-axis: $n^2/m$.
} 
\vspace{-2mm}
\label{f:performKep035}
\end{figure}

\begin{figure}[h!]
\centering
\includegraphics[width=0.45\textwidth]{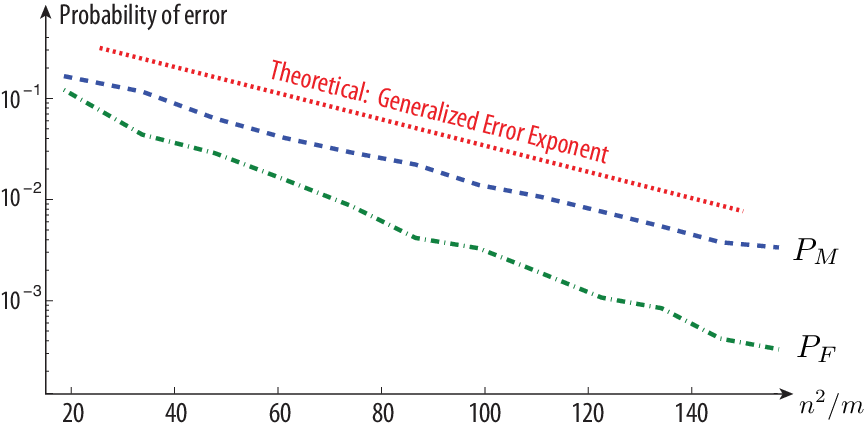}
\caption{ 
Performance of $\ktest$ with $\epsy=0.45$. %Y-axis: $P_M$ and $P_F$, averaged over $10^5$ runs; X-axis: $n^2/m$.
} 
\vspace{0mm}
\label{f:performKep045}
\end{figure}

\subsection{Rate function and worst-case distributions}
%The large deviations analysis can be applied to characterize the \emph{worst-case} alternative distributions for a given test, i.e., the sequence of alternative distributions for which the probability of missed detection associated with a test is asymptotically the largest. 

%Similar to a large deviations analysis for the large sample case, we can define a large deviations rate function for the small sample case. 

%\textbf{o.k. to delete this?  Saying "consider" twice was odd: Consider the coincidence-based test $\ktest$.
%Please read my new intro carefully, and see if it is o.k.}

In the analysis of a rate function for these hypothesis testing problems, we
consider the following restricted set of alternative distributions,
\begin{equation}\label{e:defbmPmub}
\bmPmub=\{\mu \in \bmP: \max_j \mu_j \leq \mub/m\},
\end{equation}
where $\mub$ is a constant satisfying $\mub \geq \max\{2/(1-\epsy), 4\epsy\}$.
This restricted set of distributions has bounded likelihood ratios with respect to the uniform distribution $\pip$. This bound simplifies treatment of the coincidence-based test $\ktest$.

In analogy with standard terminology from large deviations theory, the following limit
will be called the \textit{rate function} associated with the test $\ktest$, 
%\textbf{why do we have to say coincidence test in definition of rate function?  We should impose this after the definition.}
a threshold $\tau$,
and a sequence of distributions $\smu=\{\mu^{(1)},\mu^{(2)}, \mu^{(3)}, \ldots\}$ with $\mu^{(n)} \in \bmPmub$:
\begin{equation}
J_\smu(\ktest, \tau)\!=\!-\!\limsup_{n \rightarrow \infty}\frac{m}{n^2} \log(\Prob_{\mu^{(n)}}\{\kstat\leq \Expect_{\pip}[\kstat]\!+\!\frac{n^2}{m}\tau\}).\nonumber
\end{equation}

We show that $J$ is a function of the following quantity:
\begin{equation}\label{e:defD}
\kinf(\smu)\eqdef \liminf_{n}\sum_{j}\frac{(\mu^{(n)}_j)^2}{\pip_j}.
\end{equation}
The proof of 
\Theorem{ratefunction} is given in \Appendix{kperformance}. 

\begin{theorem}
\label{t:ratefunction}
%Consider the test $\ktesta$ with thresholds satisfying \eqref{e:eqdeftau} and coefficients satisfying the condition of \Theorem{modifiedK}. 
%Suppose the likelihood ratio between $\mu \in \smu$ and $\pip$ is bounded: 
%\[\limsup_{n}\sup_{j}{\mu^{(n)}_j}/{\pip_j} <\infty,\]
%then 
The rate function for the coincidence-based test is the supremum,
\begin{equation}\label{e:formulamathbfJ}
{J}_\smu(\ktest,\tau) = \sup_{\theta \geq0} \{\theta(-1-\tau)-\half (e^{-2\theta}-1)\kinf(\smu)\}.
\end{equation}
\end{theorem}
%The formula in \eqref{e:formulamathbfJ} can be solved:
%\begin{equation}
%\begin{aligned}
%&{J}_\smu(\ktest,\tau) \\
%=& \left\{\begin{array}{c c}
%\!\half[\kinf(\smu)-(1\!+\!\tau)+(1\!+\!\tau)\log(\frac{1+\tau}{\kinf(\smu)})] & \!\textrm{if $\kinf(\smu) \geq 1\!+\!\tau$},\\
%0 & \textrm{otherwise}.
%\end{array}\right.
%\end{aligned}\nonumber
%\end{equation}

The rate function can be applied to identify the sequence of worst-case alternative distributions, for which the probability of missed detection is asymptotically the largest. Note that  $J_\smu(\ktest,\tau)$ is monotonically increasing in $\kinf(\smu)$. Therefore, the smaller the quantity $\kinf(\smu)$, the larger the probability of missed detection associated with $\smu$. The sequence of distributions achieving the minimum $\kinf(\smu)$ is given in the following lemma:
\begin{lemma}\label{T:WORSTMUJSQUARE}
When $\pip$ is the uniform distribution, we have
\begin{equation}
\inf_{\mu \in \mP} \bigl(\sum_{j=1}^m \frac{\mu_j^2}{\pip_j}\bigr)=(1+\taul(\epsy))(1+o(1))\label{e:worstmujsquare}.
\end{equation}
The infimum is achieved by the following bi-uniform distribution:
\begin{compactnumerate}
\item When $\epsy < 0.5$,
\begin{equation}\label{e:worstmu2a}
\mu_{j}^*=\left\{\begin{array}{c c}{1}/{m}+{\epsy}/{\lfloor m/2 \rfloor}, &  j \leq \lfloor m/2 \rfloor,\\ {1}/{m}-{\epsy}/{\lceil m/2 \rceil}, & j> \lfloor m/2 \rfloor.
\end{array}\right.\end{equation}
\item When $\epsy \geq 0.5$,
\begin{equation}\label{e:worstmu1}
\mu_{j}^{*}=\left\{\begin{array} {c c} {1}/{\lfloor m(1-\epsy)\rfloor}, & j \leq \lfloor m(1-\epsy)\rfloor,\\ 0, & j > \lfloor m(1-\epsy)\rfloor.\end{array}\right.
\end{equation}
\end{compactnumerate}
\end{lemma} 

Thus, the  worst case distributions are identified as \emph{bi-uniform} distributions whose p.m.f.s take only two possible values. 
\begin{proof}[Proof of Lemma~\ref{T:WORSTMUJSQUARE}]
The main task is to show that any optimizer $\mu^*$ is a bi-uniform distribution. The formulae \eqref{e:worstmu2a} and \eqref{e:worstmu1} follow from solving the optimization in \eqref{e:worstmujsquare} restricted to bi-uniform distributions. 

Let $\mathcal{J_+}=\{j: \mu^*_j \geq \pip_j\}$, $\mathcal{J_-}=\{j: \mu^*_j <\pip_j\}$.
The following quadratic programming problem has a unique optimal solution $x^*=\mu^*$:
\[
\begin{array}{c c}
\min & \sum_{j \in \mathcal{J_+}} x_j^2,\\
{\sf{s.t.}} &  \sum_{j \in \mathcal{J_+}} x_j=\sum_{j \in \mathcal{J_+}}\mu^*_j,\\
& x_j=\mu^*_j \textrm{ for $j \in \mathcal{J_-}$},\\
& x_j \geq \pip_j \textrm{ for $j \in \mathcal{J_+}$}.
\end{array}
\]
By Jensen's inequality, $x^*$ must satisfy $x^*_j=x^*_{j'}$ for all $j, j' \in \mathcal{J_+}$. Thus, $\mu^*$ also satisfies $\mu^*_j=\mu^*_{j'} \textrm{ for all $j, j' \in \mathcal{J_+}$}$. The same conclusion holds for $j \in \mathcal{J_-}$. Consequently, $\mu^*$ must be a bi-uniform distribution.
\end{proof}

\subsection{Sketch of the proofs for \Theorem{KPERFORMANCE} and \Theorem{converse}}

The large deviations characterization of the probability of false alarm $P_F$ for the coincidence-based test follows from the following asymptotic approximation of the logarithmic moment generating function of its test statistic:
\begin{equation}
\begin{aligned}
\log\bigl(\Expect_{\pip}[\exp\{\theta (n-\kstat)\}]\bigr)=&\frac{1}{2} \frac{n^2}{m}\bigl(m\sum_{j=1}^m\pip_j^2\bigr)(e^{-2\theta}-1)\\&+O(\frac{n^3}{m^2})+O(1).\nonumber
\end{aligned}
\end{equation}
A characterization of $P_M$ is obtained in similar way except we need to work with the set of alternative distributions. We show that the probability of missed detection is dominated by that associated with the worst-case distributions given in Lemma~\ref{T:WORSTMUJSQUARE}. The details are given in \Appendix{kperformance}.

The main idea to prove the converse result is the following: A sequence of events $\{B_{n, \tau,\delta}\}$ is constructed so that (i) the probability of these events can be lower-bounded based on the condition on $P_F$; (ii) the probability of  missed detection conditioned on these events is lower-bounded. The key to the proof is the following inequality: 
\begin{equation}%\label{e:PMboundPForigin}
\begin{aligned}
&P_M(\phi_n)\geq \sup_{\mu \in \mP}\Prob_{\mu}\bigl(\{\phi_n=0\} \cap B_{n, \tau,\delta}\bigr)\\ \geq& \sup_{\mu \in \mP}\frac{\mu^n}{\pip^n}(\{\phi_n=0\} \cap B_{n, \tau,\delta})\Prob_\pip(\{\phi_n=0\} \cap B_{n, \tau,\delta}).\nonumber\end{aligned}
\end{equation}
A lower-bound on the second term $\Prob_\pip(\{\phi_n=0\} \cap B_{n, \tau,\delta})$  follows from  the construction of the events and the assumption on the probability of false alarm.

To lower-bound the first term $ \sup_{\mu \in \mP}\frac{\mu^n}{\pip^n}(\{\phi_n=0\} \cap B_{n, \tau,\delta})$, we construct a collection of distributions over which the largest likelihood ratio is always lower-bounded on the event $B_{n,\tau,\delta}$. We use the mixing of indistinguishable distributions method previously used in proving hardness results for composite and hypothesis testing problems \citep{pan08p4750,bar89p107, kelwagtulvis13p782}. First, construct a collection of distributions so that for each distribution $\mu$, the likelihood ratio $\mu/\pip$ has a simple expression. Second, show that for any observations $\bfmz_1^n \eqdef\{z_1, \ldots, z_n\}$ in the event $B_n$, the \emph{average} of $\Prob_{\mu}\{\bfmZ_1^n=\bfmz_1^n\}/\Prob_{\pip}\{\bfmZ_1^n=\bfmz_1^n\}$ over the collection of distributions can be lower-bounded, which in turn lower-bounds the \emph{worst} case.  These distributions are obtained by taking the worst-case distribution $\mu^*$ given in \eqref{e:worstmu2a} and permuting the symbols in $[m]$. Let $\Kset$ denote the collection of all subsets of $[m]$ whose cardinality is $\lfloor m/2 \rfloor$. For each set $\mathcal{U} \in \Kset$, define the distribution $\mu_{\mathcal{U}}$ as
\begin{equation}
\mu_{\mathcal{U},j}=\left\{\begin{array}{l l}{1}/{m}+{\epsy}/{\lfloor m/2 \rfloor}, & j \in \mathcal{U};\\ {1}/{m}-{\epsy}/{\lceil m/2 \rceil}, & j \in [m]\setminus \mathcal{U}.
\end{array}\right.\label{e:muUalternative} 
\end{equation}
Then a lower-bound is obtained using \[\sup_{\mathcal{U} \in \Kset}\!\!\frac{\mu_{\mathcal{U}}^n}{\pip^n}(\{\!\phi_n\!=\!0\} \!\cap \!B_{n, \tau,\delta}) \!\geq \! \frac{1}{|\mathcal{U}|}\!\!\sum_{\mathcal{U} \in \Kset}\!\!\!\frac{\mu_{\mathcal{U}}^n}{\pip^n}(\{\!\phi_n\!=\!0\}\! \cap \!B_{n, \tau,\delta}).\]

The details are given in \Appendix{converse}.

This technique of using uniform lower-bounds on likelihood ratio (LR) to prove lower-bounds of probability of missed detection has been applied in \citep{pan08p4750,bar89p107}: In this prior work,  a uniform bound on LR is obtained \emph{over all possible} $\bfmz_1^n$. To prove the tight hardness result as in \Theorem{converse}, we expurgate the set of observations and only require the bound on LR to hold uniformly for the sequences in the event $B_n$ instead of all sequences. This gives us the freedom to optimize $B_n$ to obtain the tightest bound. 

%Thus the value of the new bound on LR is larger than the prevous one, leading to a tighter bound for the probability of misse detection. 

\section{Extensions of the Coincidence-Based Test}\label{sec:coincidencetest}

This section collects together extensions of \Section{sec:mainresults} in terms of tests and models. We first propose a collection of tests that extend the coincidence-based test, and provide the freedom for fine-tuning the performance for finite samples. We then propose an extension of the coincidence-based test for non-uniform $\pip$. 

\subsection{Extensions considering symbols appearing more than once} \label{sec:extensioncoincidence}
The coincidence-based test uses only the number of symbols that appear in the sequence exactly once. We now add terms to the test statistic that also depend on the number of symbols appearing more than once to create a broader collection of tests. Conditions will be established under which these tests have optimal generalized error exponents. 
Consider the class of test statistics of the following form: For some $\bar{l} \geq 2$ and $v \in \Re^{\bar{l}}$, 
\begin{equation}\label{e:modifiedK}
\kstata=\kstat+ \sum_{l=2}^{\bar{l}} v_l \ind\{n\Gamma_j^n=l\}.
\end{equation}
The test is given by 
\[\ktesta(\bfmZ_1)=\ind\{\kstata-\Expect_{\pip}[\kstata] \geq \tau_n\}.\]
%The tools developed in this paper can be used to analyze the generalized error exponents of these tests. 
\begin{theorem}\label{t:modifiedK}
If $\bar{l}<\infty$, $v_2=0$, and $v_l \geq 0$ for all $3 \leq l \leq \bar{l}$, then the test $\ktesta$ achieves the optimal generalized error exponents given in \eqref{e:EEregion}.
\end{theorem}
Its proof is given in \Appendix{extension}. 

The additional terms for $l\geq 3$ in the separable statistic give us ways to fine-tune the test for a better finite-sample performance. One interesting question is to obtain finer asymptotic approximations of $\log(P_F)$ and $\log(P_M)$ that provide guidance on how to select the weights $\{v_l\}$.

%For the case with $v_2 \neq 0$, if $v_2>-2$ and $\bar{l}<\infty$, then 
%\[\lim_{n\rightarrow \infty}
%
%
%we have the following conjecture:
%\begin{conjecture}
%If  $\kstata$ satisfies $\bar{l}<\infty$, $v_2>-2$, and $v_l \geq 0$ for all $3 \leq l \leq \bar{l}$, then the test is optimal in terms of the generalized error exponent.  
%\spm{You need some explanation for why you think $-2$ is special}
%\end{conjecture}

%In practice, an approximation to $\Expect_\pip[\sum_{j=1}^m \ind\{n\Gamma^n_j=1\}]$ is used, and the test is given by 
%$\bphi_n(\bfmZ_1)=\ind\{n-\frac{n^2}{m}-\sum_{j=1}^m \ind\{n\Gamma^n_j=1\} \geq \tau_n\}$.

\subsection{Extensions to non-uniform $\pip$}
The coincidence-based test can be extended to the case where $\pip$ is not necessarily uniform but the likelihood ratio between $\pip$ and the uniform distribution remains bounded. 
\begin{assumption}\label{as:nonuniform}
There exists a constant $\eta>0$ such that 
$\max_j m\pip_j \leq \eta$ holds for all $n$.
\end{assumption}
The following separable statistic is considered,
%\begin{equation}\label{e:ttest}
\[
\tstat=\sum_{j=1}^m f_j(n\Gamma^n_j)
\]
% \end{equation}
with
 \begin{equation}
\label{e:ttestf}f_j(n\Gamma^n_j)=\left\{\begin{array}{c c}
 \half n^2 \pip_j^2, & n\Gamma^n_j=0,\\
 -n\pip_j, & n\Gamma^n_j=1,\\
 1, & n\Gamma^n_j=2,\\
0, & \textrm{others}.
 \end{array}\right.
 \end{equation}
The \emph{weighted coincidence-based test} is $\ttest_n=\ind\{\tstat \geq \tau_n\}$. 

%Instead of assuming $\pip$ is uniform as in \Assumption{piuniform}, we make the following relaxed assumption on $pi$
%The expectation of $\tstat$ explains why the coefficient was chosen in the way of \eqref{e:ttestf}:
The choice of coefficients given in \eqref{e:ttestf} ensures $\Expect_\nu[\tstat]$ approximates the $\ell_2$-distance between $\nu$ and $\pip$:
\begin{lemma}\label{T:EXPECTT}
For $\nu \in \bmPmub$, the expectation of $\tstat$ is given by:
\[\Expect_{\nu}[\tstat]=\frac{1}{2}\frac{n^2}{m} [ m \sum_{j=1}^m (\nu_j-\pip_j)^2]+O(\frac{n^3}{m^2}).\]
\end{lemma}

The proposed test has nonzero generalized error exponents: 
\begin{theorem}\label{T:TBOUND}
Suppose \Assumption{largealphabet} and \Assumption{nonuniform} hold. For $\tau \in (0, 2\epsy^2)$ where $\tau$ is defined in \eqref{e:eqdeftau}, the test $\ttest$ has nonzero generalized error exponents:
\[J_F(\ttest) >0, \quad J_M(\ttest) >0.\]
\end{theorem}
Its proof is given in \Appendix{extension}. 
%The formulas for $J_F$ and $J_M$ have a complicated dependence on $\pip$, since the "worst-case" distributions no longer have the simple formulas in \eqref{e:worstmu1} and \eqref{e:worstmu2a}.% and \eqref{e:worstmu2b}.
 
\section{Pearson's Chi-Square Test}\label{sec:pearson}
In this section, we investigate the performance of Pearson's chi-square test given in \eqref{e:defpearson}. We find that this test has a zero generalized error exponent, and therefore its probability of error is asymptotically larger than that of the coincidence-based test.

Pearson's chi-square test is asymptotically consistent in the small sample case:
\begin{proposition}[Asymptotic consistency]\label{t:consistentpearson}
Under \Assumption{largealphabet}, there exists a sequence of thresholds $\{\tau_n\}$, with which the Pearson's chi-square test is asymptotically consistent:
%\[\lim_{n \rightarrow \infty}\Prob_\pip\{\ptest(\bfmZ_1^n)=1\}=0, \quad \lim_{n \rightarrow \infty}\sup_{\mu \in \mP}\Prob_\mu\{\ptest(\bfmZ_1^n)=0\}=0.\]
\[\lim_{n \rightarrow \infty}P_F(\ptest_n)=0, \quad \lim_{n \rightarrow \infty}P_M(\ptest_n)=0.\]
\end{proposition}
We give a proof that highlights the relationship between Pearson's chi-square test and  the coincidence-based test. % on the expectation and variance of $\pearson$:
\begin{proof}[Proof of \Proposition{consistentpearson}]
Let $\tau_n=n+\half\frac{n^2}{m}(\taul(\epsy)-1)$. Applying approximations of moments of separable statistic given in \Lemma{separableexpectation} and \Lemma{separablevariance}, we obtain 
\begin{equation} \begin{aligned}\label{e:pearsonexpectation}
\Expect_{\pip}[\pearson]\!&=\!n\!+\!O(\frac{n^3}{m^2}), \\
\Var_{\pip}[\pearson]\!&=\!2\frac{n^2}{m}(m\sum_{j=1}^m\pip_j^2)(1+o(1)).\!\end{aligned}
\end{equation}
Applying Chebyshev's inequality gives $\lim_{n \rightarrow \infty}P_F(\ptest_n)=0$.

We bound $P_M(\ptest_n)$ by coupling Pearson's chi-square statistic $\pearson$ with the coincidence-based test statistic $\kstat$:
\[\begin{aligned}
\pearson&=\sum_{j=1}^m (n\Gamma^n_j-n\pip_j)^2=\sum_{j=1}^m (n\Gamma^n_j)^2-\frac{n^2}{m}\\&\geq 2\sum_{j=1}^n \ind\{n\Gamma^n_j\geq 2\}n\Gamma^n_j+ \sum_{j=1}^m\ind\{n\Gamma^n_j=1\}-\frac{n^2}{m}\\&=2n+\kstat-\frac{n^2}{m},\end{aligned}\]
where the inequality follows from $(n\Gamma^n_j)^2\geq 2(n\Gamma^n_j)$ when $n\Gamma^n_j > 1$. Consequently,
\begin{equation}\label{e:couplingKpearson}
\{\pearson \leq \tau_n\} \subseteq \{\kstat \leq \tau_n-2n+\frac{n^2}{m}\}.
\end{equation}
The asymptotic approximation on the expectation of $\kstat$ obtained from \Lemma{separableexpectation} gives \[\tau_n-2n+\frac{n^2}{m}=\Expect_{\pip}[\kstat]+\frac{1}{2}\frac{n^2}{m}(\taul(\epsy)-1)+O(\frac{n^3}{m^2}).\] It follows from \Theorem{KPERFORMANCE} that the coincidence-based test is asymptotically consistent. Thus
\[\lim_{n \rightarrow \infty}\sup_{ \mu \in \mP}\Prob_\mu\{\kstat \leq \tau_n-2n+\frac{n^2}{m}\}=0.\] Applying \eqref{e:couplingKpearson}, we obtain 
\[\lim_{n \rightarrow \infty}\sup_{ \mu \in \mP}\Prob_\mu\{\pearson \leq \tau_n\}=0.\]
\end{proof}

However, the probability of false alarm of Pearson's chi-square test is asymptotically larger than that of the coincidence-based test: We show that its generalized error exponent of false alarm is zero: 
\begin{theorem}\label{t:pearsonerrorexponent}
Suppose \Assumption{largealphabet} hold. 
Assume in addition that $m=o(n^2/\log(n)^2)$. If the sequence of thresholds is chosen so that
\begin{equation}\label{e:PMpearson}
\lim_{n \rightarrow \infty} P_M(\ptest_n) =0,
\end{equation}
then the generalized error exponent of false alarm is zero, i.e., 
\begin{equation}\label{e:JFpearson}
J_F(\ptest) =0.
\end{equation}
\end{theorem} 
We conjecture that the conclusion holds without the assumption $m=o(n^2/\log(n)^2)$.

Now compare Pearson's chi-square test and the coincidence-based test. Pearson's chi-square test statistic can be written as 
\begin{equation}\label{e:pearsonexpansion}
\begin{aligned}
\pearson=&-\frac{n^2}{m}+\sum_{j=1}^m \ind\{n\Gamma_j^n=1\}+\sum_{j=1}^m 4\ind\{n\Gamma_j^n=2\}\\&+\sum_{l=3}^\infty \sum_{j=1}^m l^2\ind\{n\Gamma_j^n=l\}.
\end{aligned}
\end{equation}
The main difference between these two tests are how the coefficients of $\ind\{n\Gamma_j^n=l\}$ for $l\geq 2$ are chosen: Remove all the terms corresponding to $l \geq 3$ and consider the following separable statistic:
\begin{equation}\label{e:pearsonchisquaremodified}
\pstata=-\frac{n^2}{m}+\sum_{j=1}^m \ind\{n\Gamma_j^n=1\}+\sum_{j=1}^m 4\ind\{n\Gamma_j^n=2\}.
\end{equation}
Then we have the following relationship between these three test statistics:
\begin{equation}
{\Omega}^{\sf P} \!\eqdef\{\pearson \leq \ctau_n\}\subset {\Omega}^{*}\!\eqdef\{\kstat \leq \tau_n\} \subset {\Omega}^{\sf P0}\!\eqdef\{\pstata \leq \ctau_n\}\nonumber
\end{equation}
where  the thresholds $\tau_n$ and $\ctau_n$ satisfy $\ctau_n=\tau_n+2n-\frac{n^2}{m}$. This is depicted in \Figure{fig:TestGeo}. Note that the region which Pearson's chi-square test decides in favor of $\Ha$ is larger than the coincidence-based test, and the probability that the empirical distribution fall into this region is asymptotically larger than $\exp\{-\alpha n^2/m\}$ for any $\alpha>0$. This is made precise in the proof of \Theorem{pearsonerrorexponent}. On the other hand, we can show  that the test associated with $\ptesta$ has $J_M=0$ by considering a sequence of alternative distributions whose likelihood ratios with respect to $\pip$ increase to infinity. In sum, we have
\begin{enumerate}
\item $J_F(\ptest)=0, J_M(\ptest)>0$;
\item $J_F(\ktest)>0, J_M(\ktest)>0$;
\item $J_F(\ptesta)>0, J_M(\ptesta)=0$.
\end{enumerate}
\begin{figure}[h!]
\centering
\includegraphics[width=0.25\textwidth]{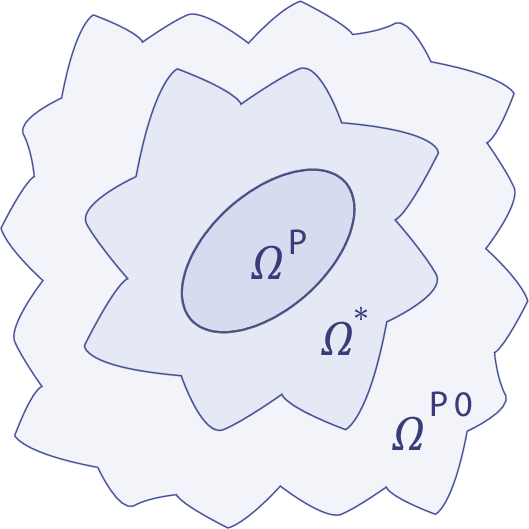}
\caption{ 
Decision regions in the space of p.m.f. for Pearson's chi-square test, the coincidence-based test and the test given in \eqref{e:pearsonchisquaremodified}.
} 
\vspace{0mm}
\label{fig:TestGeo}
\end{figure}

\begin{proof}[Proof of \Theorem{pearsonerrorexponent}]
The requirement $P_M(\ptest_n) \rightarrow 0$ imposes an upper-bound on the threshold $\tau_n$ for $\ptest$:
\begin{lemma}\label{T:PEARSONTAUBD}
In order for \eqref{e:PMpearson} to hold, for large enough $n$, we must have
%If the threshold $\tau_n$ is chosen to satisfy \eqref{e:PMpearson}, then for $n$ large enough:
\begin{equation}%\label{e:taubound}
\tau_n\leq \bar{\tau}_n \eqdef \Expect_{\pip}[\pearson]+\frac{n^2}{m}\taul(\epsy)+2\frac{n}{\sqrt{m}}.\nonumber
\end{equation}
%where 
%\[\bar{\tau}_n=\Expect_{\pip}[\pearson]+\frac{n^2}{m}\taul(\epsy)+2\frac{n}{\sqrt{m}},\]
%with $\taul(\epsy)$ defined in \eqref{e:equtaul}.
\end{lemma}
Consider the event that the first symbol appears many times:
\[
A_n\eqdef \{n\Gamma^n_1=\lfloor \frac{n\sqrt{2\taul(\epsy)}}{\sqrt{m}}\rfloor\}.
\]
In the event $A_n$, the first term $f_1(n\Gamma^n_1)$ in the summation in the definition of $\pearson$ given in \eqref{e:defpearson} is approximately $2\frac{n^2}{m}\taul(\epsy)$. This drives the value of $\pearson$ above the threshold $\tau_n$. Thus the probability of false alarm conditioned on this event converges to \emph{one}, as summarized in Lemma~\ref{T:PEARSONA}. On the other hand, the probability of $A_n$ does not decay exponentially fast with respect to $n^2/m$, as summarized in Lemma~\ref{T:PEARSONA2}. 

\begin{lemma}\label{T:PEARSONA}
\[
\Prob_\pip \{\pearson \geq \bar{\tau}_n|A_n\}=1-o(1).\]
\end{lemma}
\begin{lemma}\label{T:PEARSONA2}
\begin{equation}%\label{e:pearsonA2}
-\lim_{n \rightarrow \infty}\frac{m}{n^2}\log(\Prob_\pip\{A_n\})=0.\nonumber
\end{equation}
\end{lemma}
Combining Lemma~\ref{T:PEARSONTAUBD}, Lemma~\ref{T:PEARSONA} and  Lemma~\ref{T:PEARSONA2} together, we conclude  %obtain that it following holds for any test satisfying the assumptions of the theorem given in \eqref{e:PMpearson},
\[J_F(\ptest)\leq -\liminf_{n \rightarrow \infty}\frac{m}{n^2}\log\bigl(\Prob_\pip \{\pearson \geq \bar{\tau}_n|A_n\} \Prob_\pip\{A_n\}\bigr)=0.\]
The proofs of these three lemmas are given in \Appendix{lemmapearson}. 
\end{proof}

\section{Alternative Distributions Based on $f$-Divergence}\label{sec:extension}
The set of alternative distributions studied in previous sections is defined using the total variation distance. The generalized error exponent analysis with the same normalization $r(n,m)=n^2/m$ also applies to other distance functions, as we will show in \Proposition{GEEcondition1} and \Proposition{GEEcondition2}. The set of alternative distributions $\mP$ considered in this section is also defined in \eqref{e:H1set} using the general distance function $d$ rather than $d=d_{\sf TV}$. Examples include the Kullback-Leibler (KL) divergence
\[d_{\sf KL}(q,p)=\sum_{j} q_j\log(q_j/p_j),\]
and its generalization known as $f$-divergence, 
\begin{equation}
d_f(q,p)=\sum_j p_j f(q_j/p_j),
\end{equation}
where $f$ is a convex function with $f(1)=0$.

Conditions under which the generalized error exponent analysis applies are given in the following: 
\begin{proposition}\label{t:GEEcondition1}
Suppose the distance function $d$ in the definition of alternative distribution in \eqref{e:H1set} satisfies
\begin{enumerate}
\item $d(q,p) \geq \alpha \dtv(q,p)$ for some $\alpha>0$.
\item \[\liminf_{n \rightarrow \infty}\inf_q\{\sum_j\frac{\mu_j^2}{\pip_j}: d(\mu,\pip) \geq \epsy, \mu \in \bmP\}>0.\]
\end{enumerate}
Then $n^2/m$ is the appropriate normalization for the large deviations analysis for small $\epsy>0$: There exists a test $\bphi$ such that 
\[J_F(\bphi)>0, J_M(\bphi)>0.\]There is a constant $\bar{J}$ satisfying $0<\bar{J}<\infty$ such that for any test $\bphi$, we have
\[\min\{J_F(\bphi),J_M(\bphi)\} \leq \bar{J}.\]
\end{proposition}

When $f$-divergence $d=d_f$ is used in the definition of alternative distribution, the generalized error exponent can be applied subject to conditions on $f$:
\begin{proposition}\label{t:GEEcondition2}
Suppose $f$ satisfies the following conditions:
\begin{enumerate}
\item For some $0<x<1$, 
\[\half(f(1-x)+f(1+x))>f(1).\]
\item There is a constant $\alpha>0$ such that for all $x$, 
\[f(x) \leq \alpha (x-1)^2.\]
\end{enumerate}
Then $n^2/m$ is the appropriate normalization for the large deviations analysis for small $\epsy>0$: There exists a test $\bphi$ such that 
\[J_F(\bphi)>0, J_M(\bphi)>0.\]There is a constant $\bar{J}$ satisfying $0<\bar{J}<\infty$ such that for any test $\bphi$, we have
\[\min\{J_F(\bphi),J_M(\bphi)\} \leq \bar{J}.\]
\end{proposition}
Note that the KL divergence satisfies the conditions of this proposition.
\begin{proof}[Proof of \Proposition{GEEcondition1}]
The converse result in \Theorem{converse} is proved by showing that the worst-case probability of missed detection over the set of distributions given in \eqref{e:muUalternative} is lower-bounded regardless of the test used. The first condition in \Proposition{GEEcondition1} guarantees that these distributions are still in the set $\mP$ of alternative distributions. 

For the achievability result, the critical step is to show that the rate function is  positive for any alternative distribution whose likelihood ratio with respect to $\pip$ is bounded. The second condition in \Proposition{GEEcondition1}  guarantees that $\kappa$ defined in \eqref{e:defD} is positive, which by \Theorem{ratefunction} implies that the rate function of the coincidence-based test is positive.  
\end{proof}

\begin{proof}[Proof of \Proposition{GEEcondition2}]
The proof is similar to that of \Proposition{GEEcondition1}. The first condition of \Proposition{GEEcondition2} ensures that the collection of bi-uniform distributions given in \eqref{e:muUalternative}  used in the proof of the converse result is in the set $\mP$ of alternative distributions: For $\mu_{\mathcal{U}}$ defined in \eqref{e:muUalternative} with $\epsy$ replaced by $\epsy'$, for even $m$, for small enough $\epsy$, we have
\[d_f(\mu_{\mathcal{U}},\pip) = \half f(1+2\epsy')+\half f(1-2\epsy') \geq \epsy.\]
The second condition implies that
\[\alpha \sum_j\frac{\mu_j^2}{\pip_j} \geq d_f(\mu,\pip) \geq \epsy.\]
Thus,  the rate function is positive for any alternative distribution whose likelihood ratio with respect to $\pip$ is bounded. 
\end{proof}

\section{Discussions}\label{sec:discussion}

This paper invites more questions than it answers.  We collect here further connections with other concepts in statistics and information theory.  %Further directions for future research are collected in the conclusions.  

\subsection{Asymptotic relative efficiency}\label{sec:bahadur}

In the case of fixed alphabet, connections between error exponent and asymptotic relative efficiency such as the Chernoff efficiency are summarized in \cite[Chapter 22]{das08} and \cite[Chapter 10]{ser80}. This has been extended to the large sample case where $m \rightarrow \infty $ and $m=O(n)$ in \cite{quirob85p727} by treating $m$ as a function of $n$.  We will examine the connection for the small sample case.

We first examine the connection between Chernoff efficiency and generalized error exponent. Following \cite{quirob85p727}, we consider the setting of continuous-valued observations in \Section{sec:application2continous} in which the observations are grouped into cells $\prt$ that have equal probabilities under the null distribution $P$. Let $m(n,\bphi)$ be the number of cells used when the number of observations is $n$ for a  test $\bphi$. We are interested in the small sample case where $\lim_{n \rightarrow \infty} n/m(n,\bphi)=0$. Consider two tests $\bphi$ and $\bphi'$. Let $n(\alpha, \beta,\bphi)$ and $n'(\alpha, \beta,\bphi')$ be the number of observations required for the tests $\bphi$ and $\bphi'$, respectively, so that the probability of false alarm is $\alpha$ and the probability of missed detection under a \emph{particular} distribution $Q$ is $\beta$. When both lower probability of false alarm and missed detection are of interest, the Chernoff efficiency is used. It is defined as $e_{C}(\bphi, \bphi')=\lim_{\alpha \rightarrow 0} n(\alpha, \alpha,\bphi')/n(\alpha, \alpha,\bphi)$, where the probability of false alarm and missed detection is  set to be equal, i.e. $\alpha=\beta$.

The generalized error exponents $J_F(\bphi)$ and $J_{M, Q}(\bphi)$, where the subscript $Q$ indicates $J_{M,Q}$ is the generalized error exponent for a {particular} alternative distribution $Q$, are defined for a sequence of partitions $\prt$ whose number of cells is given by $m(n, \bphi)$. We choose the test threshold so that probability of false alarm and missed detection is equal. Define 
\[r(x)=\lim_{n \rightarrow \infty} m(nx,\bphi')/m(n,\bphi).\]
The function $r(x)$ characterizes how the number of cells increases with the number of samples for the two tests. For example, when $m(n, \bphi)=m(n, \bphi')= n^a$, we have $r(x)=x^a$. This function has also been used in  the study of relative efficiencies for the large sample case in \cite{quirob85p727}.

\begin{proposition}\label{t:ChernoffARE}
Suppose the following conditions hold:
\begin{enumerate}
\item $m(n, \bphi)$ and $m(n, \bphi')$ are both monotonically non-decreasing in $n$. 
\item $\lim_{n \rightarrow 0} n/m(n,\bphi)=0$, and $\lim_{n \rightarrow 0} n/m(n,\bphi')=0$.  
\item $r(x)$ is well-defined and continuous on $(0, \infty)$.
\item $n(\alpha,\alpha, \bphi)$ and $n(\alpha,\alpha, \bphi')$ are both monotonically non-increasing in $\alpha$. 
\item $0<\min\{J_C(\bphi), J_C(\bphi')\} \leq \max\{J_C(\bphi), J_C(\bphi')\}< \infty$. 
\item $0<e_{C}(\bphi',\bphi)<\infty$. 
\end{enumerate}
Then
$e_{C}(\bphi',\bphi)$ satisfies
\begin{equation}\label{e:PearsonAndGEE}
\frac{e_{C}(\bphi',\bphi)^2}{r(e_{C}(\bphi',\bphi))}=\frac{J_C(\bphi')}{J_C(\bphi)}.\nonumber
\end{equation}
\end{proposition}
\begin{proof}
It follows from the  monotonicity of $n(\alpha, \alpha, \phi)$ and the condition on $J_C(\bphi), J_C(\bphi')$ that $n \rightarrow \infty$ if  $\alpha \rightarrow 0$. Combining this with the monotonicity condition on $m$, we obtain
\[\lim_{\alpha \rightarrow 0}\frac{\log(\alpha)}{n(\alpha,\alpha, \bphi)^2/m(n(\alpha,\alpha, \phi), \bphi)}=-J_C(\bphi).\]
Therefore, 
\[\begin{aligned}
\frac{e_{C}(\bphi',\bphi)^2}{r(e_{C}(\bphi',\bphi))}&=\lim_{\alpha \rightarrow 0}\frac{n(\alpha,\alpha, \bphi)^2}{n(\alpha,\alpha, \bphi')^2}\lim_{\alpha \rightarrow 0}\frac{m(n(\alpha,\alpha,\bphi'),\bphi')}{m(n(\alpha,\alpha,\bphi),\bphi)}\\&=\frac{\lim_{\alpha \rightarrow 0}\frac{\log(\alpha)}{n(\alpha,\alpha, \bphi')^2/m(n(\alpha,\alpha, \phi'), \bphi')}}{\lim_{\alpha \rightarrow 0}\frac{\log(\alpha)}{n(\alpha,\alpha, \bphi)^2/m(n(\alpha,\alpha, \phi), \bphi)}}\!=\!\frac{J_C(\bphi')}{J_C(\bphi)}.\end{aligned}\]
\end{proof}

Bahadur efficiency is more relevant for the scenario where the probability of false alarm is small. We adopt the definition given in \cite[Chapter 22]{das08}: $e_{B}(\bphi', \bphi,\beta)=\lim_{\alpha \rightarrow 0} n(\alpha, \beta,\bphi)/n(\alpha, \beta,\bphi')$. Under mild conditions, this can be shown to be equivalent to Bahadur's original definition based on the concept of stochastic comparison. 
Consider two tests $\bphi$ and $\bphi'$ for which the generalized error exponents are positive and finite. We choose the test threshold of the two tests so that generalized error exponents $J_F(\bphi)$ and $J_F(\bphi')$ are maximized while satisfying the constraint that $P_M(\bphi)\leq \beta <1$ and $P_M(\bphi')\leq \beta <1$. We conjecture that the following holds under conditions similar to those in \Proposition{ChernoffARE}:
\begin{equation}
\frac{e_{B}(\bphi',\bphi)^2}{r(e_{B}(\bphi',\bphi))}=\frac{J_F(\bphi')}{J_F(\bphi)}.\nonumber
\end{equation}

A conjecture concerning Hodges-Lehmann efficiency is similar and not repeated here. 
Pitman efficiency, on the other hand, has been shown to be closely related to CLT analysis in the large sample case \cite{quirob85p727}. An analysis of Pitman efficiency in the small sample setting will be investigated in future work.

%Pitman efficiency is largely based on CLT analysis. An interesting future direction is to extend the results  Pitman efficiency for the small sample case and compare it with the results from Bahadur efficiency. 

\subsection{Unified analysis framework for large and small sample}

%\spm{Most of this paragraph should be removed, no?  It seems very repetitious}
%For the large sample case where $m=o(n)$, the number of \emph{types}, i.e., the frequency of occurrence of each symbol, increases sub-exponentially with respect to $n$. Thus it is negligible in large deviations analysis. This leads to many important large deviations results for the large sample case, such as the concentration of empirical distribution around the null distribution. This is no longer the case for
%small sample problems in which
% an increasingly larger percentage of symbols never appears in the observations. Therefore, we analyze  the profile, which is the number of symbols that appear a certain
%	 number of
%times. We focus on the number of symbols appearing $l$ times for $0 \leq l \leq \bar{l}$ where $\bar{l}$ is a fixed number. However, for large sample case, the number of times each symbol appears increases and is unbounded above. Therefore, our results in this paper do not directly apply to the large sample case.  

Our results in this paper do not directly apply to the large sample case, since it is based on the analysis of the number of symbols appearing once or twice, which vanishes to zero in the large sample case. On the other hand, some of the analysis method and insights can be applied towards finding a unified analysis framework.

First, the Poissonization technique can be applied in both the large and small sample case. Similar to the unified CLT results obtained in \cite{med77p1} using the Poissonization technique, the large deviations analysis for the achievability result in this paper might be extended to a general large deviations result for separable statistics. The key difference between the large and sample case in this analysis is which terms in the expansion of the log-moment generating function vanish. For example, in the small sample case, the term corresponding to symbols appearing more than twice becomes negligible for separable statistics with a bounded $f_j(x)$. 

Second, the results on the coincidence-based test and Pearson's chi-square test might be leveraged to obtain a test that achieves non-zero error exponent for both large sample and small sample problems with uniform null distributions: The coincidence-based test is not asymptotically consistent for the large sample problem since for a uniform distribution, the number of symbols appearing only once vanishes to zero as the number of sample increases. Pearson's chi-square test has been shown to be asymptotically consistent in both cases. However, it has a zero generalized error exponent in the small sample case. The key difference between these two tests is the weights: As two separable statistics, $f_j$ in the definition of separable statistics \eqref{e:sepstats} vanishes for $x>1$ in the coincidence-based test, and increases as $x^2$ for Pearson's chi-square test. This suggest that one should investigate tests whose $f_j(x)$ increases slower than $x^2$. Examples of these tests are $\ell_1$-norm based test and GLRT.

\subsection{Non-uniform null distribution}
The results in this paper are applicable when the null distribution is uniform or nearly uniform. %While it is possible in the case of real-valued observations to choose the partition such that the uniform assumption is satisfied, it is also of considerable both practical and theoretical interest to analyze the case where the null distribution is not uniform. 
To extend the results to general non-uniform null distributions, we need to find the correct conjecture on the proper normalization in the definition of error exponents, prove a converse result and an achievability result. 

The size of alphabet $m$ is found to be the proper normalization for the generalized error exponent for the uniform case. For the non-uniform case, a generalization of $m$, such as the Shannon or R\'enyi entropy of $\pip$, might be more appropriate. 

The worst-case distributions used in the analysis in this paper are likely to be different for the non-uniform case. In \cite{batforrubsmiwhi00p259, val08p383}, a hardness result is established for the two sample problem based on the analysis of two non-uniform distributions. These two distributions are constructed using a combination of symbols with large probability and small probability, where the likelihood ratio with respect to the uniform distribution increases unbounded on a large probability symbol, and remains bounded on a small probability symbol. This construction and analysis method might also be applicable for our problem.

We have proposed a weighted coincidence-based test for the near uniform case, which approximates the $\ell_2$ norm when the likelihood ratio between the null distribution and the uniform distribution is bounded. For arbitrary non-uniform null distribution, one possible approach is to choose a different weight. 
As the results in \cite{bar89p107, batforrubsmiwhi00p259, val08p383, kelwagtulvis13p782} implies, the key is to analyze large probability and small probability symbols. A unified result on the large deviations for separable statistics for both large and small probability symbols would serve as a basis for choosing the weight. Another possible approach is to use the bucketing method in \cite{batfisforkumrubwhi01p442}, in which the set of symbols is divided into several buckets so that the distribution over the symbols in the same bucket is nearly uniform. It remains to see whether these approaches give the best possible error exponents.

\section{Conclusions and Discussions}\label{sec:conclusion}

The classical error exponent criterion, which appears in the large deviation analysis for universal hypothesis testing problems with a large number of samples, can be extended to the small sample case, provided the normalization is modified to account for both the sample size $n$ and the alphabet size $m$. 

We offer a few discussions on the results and point out directions for future research:
\begin{compactnumerate}
\item The analysis in this paper is of asymptotic nature. The generalized error exponent gives the leading term in the asymptotic expansion of the probability of error. Finer approximations are valuable especially for characterizing the finite sample performance when $n/m$ is not very small. For example, finer approximations can reveal the difference among the class of tests described in \Section{sec:extensioncoincidence} that have the same generalized error exponents.

\item It is desirable to establish general large deviation characterizations of separable statistics for small sample problems, similar to those established for $n \sameorder m$ in \citep{ron84p800,kol05p255}. Such results could provide more insights on how the coefficients of a separable statistic affect the test's performance. For example, how the performance of a test with the test statistic $\sum_{j=1}^m |n\Gamma_j^n-n\pip_j|^\rho$ varies with $\rho$? 

\item We have focused on the simple goodness-of-fit problem in this paper, in which $\pip$ is fully specified. A natural extension is the composite goodness-of-fit problem in which $\pip$ is not fully specified but assumed to be in a known set. A similar generalized error exponent concept should exist for the composite case. 

\item There are many other problems for which the approach presented in this paper is relevant. Examples include the classification problem \citep{ziv88p278, gut89p401, kelwagtulvis13p782}, the problem of testing whether two distributions are close \citep{batforrubsmiwhi00p259, achjaforlpan11p47}, and probability estimation over a large or unknown alphabet \citep{wagviskul11p3207, sanorlvis07p638, orlsanzha03p179}.

\quad In the recent work \citep{huamey12p2586} it is shown how to adapt the methods presented in this paper to the classification problem. The generalized error exponent analysis is applied to characterize the different ways in which the number of training samples and the number of test samples affect the performance of classification algorithms.

%Our preliminary results in \citep{huamey12p2586} showed that the generalized error exponent concept and the set of tools also applies to the classification problem, and leads to a characterization of the different roles played by the number of test samples and training samples. 

\item Topological structure often contains critical information that is easily ignored in the approaches focused on in this work.  In particular, in this paper we have not considered any notion of distance between points in the alphabet.  Other approaches such as the support vector machine, or more recent work such as \citep{wriyangansasyi09p210} are based primarily on topology.   It will be desirable to  create a coherent bridge between the approach developed here and topological approaches to hypothesis testing.   It is likely that current information-theoretic tools can help to create these bridges, such as concepts from lossy source-coding.   We are also considering extensions of the work described here to the feature selection problem of \cite{unndaymeysurvee11p1587,huamey10p1618} in which $m$ is interpreted as the number of features rather than the alphabet size.  
\end{compactnumerate}
\appendices

\section*{Organization of the Appendix}\label{apx:organizationappendix}

Approximations to the moments of separable statistics are given in \Appendix{moments}. These results are used in the rest of the proofs. 

The proofs of \Theorem{KPERFORMANCE} and \Theorem{ratefunction} are given in \Appendix{kperformance}. The major portion of the proof is to obtain approximations of the log-moment generating function by applying asymptotic analysis methods . Similar arguments are used in the proofs of \Theorem{modifiedK} and Theorem~\ref{T:TBOUND} given in \Appendix{extension}.

The proof of the converse result \Theorem{converse} given in \Appendix{converse} can be read almost independently of \Appendix{kperformance} and \ref{apx:extension}. It is based on analyzing the worst-case distributions given in Lemma~\ref{T:WORSTMUJSQUARE}. 

The lemmas supporting the proof of  \Theorem{pearsonerrorexponent} which characterizes Pearson's chi-square test performance,  are given in \Appendix{lemmapearson}, and can be read independently of \Appendix{kperformance}, \ref{apx:extension} and \ref{apx:converse}.

\section{Moments of Separable Statistics}\label{apx:moments}

This section provides a survey of results on asymptotic approximations to moments of separable statistics. These results hold for the distributions in the set $\bmPmub$ defined in \eqref{e:defbmPmub}.

\begin{lemma}[Expectation of a separable statistic]\label{t:separableexpectation}
Consider a separable statistic given by $\sum_{j=1}^m f_j(n\Gamma^n_j)$. Suppose that  $\max_j |f_j(x)|\leq a_0e^{a_0x}$ for some $a_0>0$. The expectation of the separable statistic for $\nu \in \bmPmub$ is given by:
\[\begin{aligned}
&\Expect_\nu[\sum_{j=1}^m f_j(n\Gamma^n_j)]\\=&\sum_j f_j(0)+n\sum_{j=1}^m \nu_j(f_j(1)-f_j(0))\\&+\frac{1}{2} \frac{n^2}{m}(m\sum_{j=1}^m \nu_j^2)\bigl(f_j(0)-2f_j(1)+f_j(2)\bigr)+O(\frac{n^3}{m^2}).\end{aligned}\] 
\end{lemma}
\begin{proof}%[Proof of \Lemma{separableexpectation}]
For any $j$, $\nu_j^3{n \choose 3}|f_j(3)|=O(\frac{n^3}{m^3})$, and
\[
\begin{aligned}
\sum_{x=4}^\infty \nu_j^x{n \choose x}|f_j(x)|
 &\leq \!  a_0\sum_{x=4}^\infty (\frac{e^{a_0} \mub n}{m})^x\\& \leq\!  \frac{a_0}{|\log(e^{a_0}\mub n/m)|}(\frac{e^{a_0} \mub n}{m})^3\!=\!O(\frac{n^3}{m^3}).
\end{aligned}
\]
Consequently, 
\[\begin{aligned}
\Expect_\nu[\sum_{j=1}^m f_j(n\Gamma^n_j)]
=&\sum_{j=1}^m [f_j(0)(1\!-\!\nu_j)^n \!+\! f_j(1)n\nu_j(1\!-\!\nu_j)^{n-1} \!\\
&+\! f_j(2){\!n\! \choose \!2\!}\nu_j^2(1\!-\!\nu_j)^{n-2}\!+\!O(\frac{n^3\!}{m^3\!})] 
% \\
%=&\sum_j \!f_j(0)\!+\!n\!\sum_{j=1}^m \!\nu_j(f_j(1)\!-\!f_j(0))\!\\
%&+\!\frac{n^2}{2}\!\!\sum_{j=1}^m \!\nu_j^2(f_j(0)\!-\!2f_j(1)\!+\!f_j(2))\!+\!O(\!\frac{n^3\!}{m^2\!}).
\end{aligned}\nonumber
\]
Rearranging the right-hand side leads to the claim of this lemma.
\end{proof}
\Lemma{separableexpectation} implies  Lemma~\ref{T:EXPECTT}, as well as the following asymptotic approximation of the expectation of $\kstat$:
\begin{lemma}\label{t:expectK}
For any $\nu \in \bmPmub$:
%\[\Expect_{\nu}[\kstat]=n-\frac{n^2}{m}\bigl(m\sum_{j=1}^m\nu_j^2\bigr)+O(\frac{n^3}{m^2}), \Var_{\nu}[\kstat]=2\frac{n^2}{m}\bigl(m\sum_{j=1}^m\nu_j^2\bigr)(1+o(1)).\!\]
\begin{equation}\label{e:expectK}
\begin{aligned}
\Expect_{\nu}[\kstat]&=-n+\frac{n^2}{m}\bigl(m\sum_{j=1}^m\nu_j^2\bigr)+O(\frac{n^3}{m^2}).
\end{aligned}\nonumber
\end{equation}
\end{lemma}
This will be used in the proof of \Theorem{KPERFORMANCE}.

\begin{lemma}[Variance of a separable statistic]\label{t:separablevariance}
Consider a \emph{symmetric} separable statistic $\sum_{j=1}^m f(n\Gamma^n_j)$. Suppose that $|f(x)|\leq a_0e^{a_0x}$ for some $a_0>0$. If  $f(0)=0$ and $f(2) \neq 2 f(1)$, then its variance for $\nu \in \bmPmub$ is given by 
\[\Var_\nu[\sum_{j=1}^m f(n\Gamma^n_j)]\!=\!\half \frac{n^2}{m} (f(2)-2f(1))^2 (m\sum_{j=1}^m \nu_j^2)(1+o(1)).\]  
\end{lemma}
\Lemma{separablevariance} is the combination of Equation 2.11 and Equation 2.20 in \citep{med77p1}.

\section{Proofs of \Theorem{KPERFORMANCE} and \Theorem{ratefunction}}\label{apx:kperformance}
The proof of \Theorem{KPERFORMANCE} and \Theorem{ratefunction} is based on the Chernoff bound and the G\"{a}rtner-Ellis Theorem. The key step is to obtain an asymptotic approximation to the logarithmic moment generating function of the test statistic. To simplify the presentation we work with the following statistic instead of $\kstat$:
\[\kstattilde \eqdef -\kstat-n=\sum_{j=1}^m \ind\{n\Gamma^n_j=1\}-n.\]
Its logarithmic moment generating is given by 
\begin{equation}
\Lambda_{\nu, \kstattilde}(\theta)\eqdef \log\bigl(\Expect_{\nu}[\exp\{\theta \kstattilde\}]\bigr).
\end{equation}
Asymptotic approximations or bounds  to $\Lambda_{\nu, \kstattilde}(\theta)$ for $\nu \in \bmPmub$ and $\nu \notin \bmPmub$ are presented in Appendix~\ref{subsec:LambdanuA} and \ref{subsec:LambdanuB}.

\subsection{Approximation to the logarithmic moment generating function for distributions in $\bmPmub$}\label{subsec:LambdanuA}
Bounds and approximations for $\Lambda_{\nu, \kstattilde}$ are first obtained for the restricted set of distributions $\bmPmub$ defined in \eqref{e:defbmPmub}.
%To main job is then to to obtain an approximation $\Lambda_{\nu, \kstattilde}(\theta)$:
%to the logarithmic moment generating function associated with a distribution $\nu$, defined as
%\[\Lambda_{\nu, \kstattilde}(\theta)=\log(\Expect_{\nu}[\exp\{\theta \kstat\}]).\] 
%he first result is given below:
\begin{proposition}\label{T:ASYMLMGFK}
For any $\nu \in \bmPmub$, the logarithmic moment generating function for the statistic $\kstattilde$ has the following asymptotic expansion
\begin{equation}\label{e:lmgfKestimate1}
%\begin{aligned}
\Lambda_{\nu, \kstattilde}(\theta)=\half \frac{n^2}{m}\bigl(m\sum_{j=1}^m\nu_j^2\bigr)(e^{-2\theta}-1)+O(\frac{n^3}{m^2})+O(1).
%\end{aligned}
\end{equation}
\end{proposition}
The approximation errors $O(\frac{n^3}{m^2})$ and $O(1)$ are uniform over the set $\bmPmub$. 

We first demonstrate how to obtain a simple but not tight enough bound, given in  \eqref{e:simpleboundonLambda}. We then give the details of a proof to obtain a tigher bound. Both proofs use the Poissonization technique, and the procedure is applicable for many separable statistics including $\kstat$: 

Let $\{X_j\}$ be a sequence of independent Poisson random variables with parameter $\lambda \nu_j$ for some $\lambda >0$. Then for any integers $u_1, \ldots, u_m$ satisfying $\sum_{j=1}^m u_j=n$, we have
\begin{equation}%\label{e:conditionalpoisson}
\Prob\{n\Gamma^n_j=u_j, \textrm{for all $j$} \}\!=\!\Prob\{X_j=u_j, \textrm{for all $j$} | \sum_{j=1}^m X_j=n\}.\nonumber
\end{equation}
%Recall that $\kstat$ is a separable statistic, which is of the form $\sum_{j=1}^m f_j(n\Gamma^n_j)$.
% (See \citep{med77p1} and references therein), which are statistics of the form
%\[\sum_{j=1}^m f_j(n\Gamma^n_j)\]
%For $\kstat$, we have 
%\begin{equation}\label{e:fforkstat}
%f_j(n\Gamma^n_j) = \ind\{n\Gamma^n_j=1\}-n/m.
%\end{equation} 
%An important property of a separable statistic is that it can be written as the conditional expectation of a sum of functions of {independent} Poisson random variables: 
%\begin{equation}\label{e:conditionalpoisson}
%\sum_{j=1}^m f_j(n\Gamma^n_j)=\Expect[\sum_{j=1}^m f_j(X_j)| \sum_{j=1}^m X_j=n].
%\end{equation}
Therefore, the moment generating function of a separable statistic $\sum_{j=1}^m f_j(n\Gamma^n_j)$ admits the following representation: 
\begin{equation}\label{e:relationshipmgf}
\begin{aligned}
&\Expect_{\nu}[\exp\{ \!\theta\sum_{j=1}^m \!f_j(n\Gamma^n_j)\}]\\=&\Expect[\exp\{ \theta \sum_{j=1}^m \!f_j(X_j)\}| \sum_{j=1}^m \!X_j\!=\!n].\end{aligned}
\end{equation}
The moment generating function $A_\lambda(\theta)$ for $\sum_{j=1}^m f_j(X_j)$ is given by
\[
A_\lambda(\theta)\!\eqdef\! \Expect[\exp\{ \theta \sum_{j=1}^m f_j(X_j)\},
\]
and is easy to calculate. 

A simple bound on the moment generating function of $\sum_{j=1}^m f_j(n\Gamma^n_j)$ can then be obtained from \eqref{e:relationshipmgf} using the argument in in \cite{bar89p107}:
\begin{equation}\Expect_{\nu}[\exp\{ \theta\sum_{j=1}^m f_j(n\Gamma^n_j)\}]\leq  A_\lambda(\theta)/P\{\sum_{j=1}^m X_j=n\}.\label{e:simpleboundonLambda}\end{equation} However, this bound is not tight enough for the whole range $m=o(n^2)$ and $n=o(m)$ to prove \Theorem{KPERFORMANCE}.

A tighter approximationcan be obtained using the following relationship: \[
\begin{aligned}
&A_\lambda(\theta)\!=\! \Expect[\exp\{ \theta \sum_{j=1}^m f_j(X_j)\}]\\&=\sum_{n=0}^\infty \frac{\lambda^n}{n!}e^{-\lambda}\Expect[\exp\{ \theta \sum_{j=1}^m f_j(X_j)\}| \sum_{j=1}^m X_j=n]\\&=\sum_{n=0}^\infty \frac{\lambda^n}{n!}e^{-\lambda}\Expect_\nu[\exp\{\theta \sum_{j=1}^m f_j(n\Gamma^n_j)\}].\end{aligned}\]
It follows from the independence of  the variables $\{X_j\}$ that the moment generating function $A_\lambda(\theta)$ has the following formula:
\[A_\lambda(\theta)=\prod_{j=1}^m (\sum_{k=0}^\infty \frac{(\lambda \nu_j)^k}{k!} e^{-\lambda \nu_j} e^{\theta f_j(k)}).\]
Since $A_\lambda(\theta)$ is analytic in $\lambda$,   the moment generating function of $\sum_{j=1}^m f_j(n\Gamma^n_j)$ can be obtained  via Cauchy's theorem: 
\begin{equation}\label{e:formulamgf0}
\Expect_{\nu}[\exp\{\theta \sum_{j=1}^m f_j(n\Gamma^n_j)\}]=\frac{n!}{2\pori} \oint e^{\lambda} A_\lambda(\theta) \frac{d\lambda}{\lambda^{n+1}},
\end{equation}
where the integration is carried out along any closed contour around $\lambda=0$ in the complex plane. These arguments lead to the following lemma:
\begin{lemma}\label{t:poissonization}
The moment generating function of the separable statistic $\sum_{j=1}^m f_j(n\Gamma^n_j)$ is given by
\begin{equation}%\label{e:formulamgf}
\begin{aligned}
&\Expect_{\nu}[\exp\{\theta \sum_{j=1}^m f_j(n\Gamma^n_j)\}]\\
&=\frac{n!}{2\pori } \oint e^{\lambda} \prod_{j=1}^m \big(\sum_{k=0}^\infty \frac{(\lambda \nu_j)^k}{k!}e^{-\lambda \nu_j} e^{\theta f_j(k)}  \big)  \frac{d\lambda}{\lambda^{n+1}}.\nonumber
\end{aligned}
\end{equation}
\end{lemma}

\begin{proof}[Proof of Proposition~\ref{T:ASYMLMGFK}]
Applying \Lemma{poissonization} with $f_j(1)=1, f_j(k)=0 \textrm{ for $k\neq 1$}$, we obtain
\begin{equation}\label{e:formulamgf2}
\Expect_{\nu}[\exp\{\theta (\kstattilde)\}]=e^{-\theta n}\frac{n!}{2\pi i} \oint g(\lambda)d\lambda
\end{equation}
where 
\begin{equation}%\label{e:formulagK}
g(\lambda)=e^{\lambda} \prod_{j=1}^m (1-(\lambda \nu_j)e^{-\lambda \nu_j}+(\lambda \nu_j)e^{-\lambda \nu_j} e^{\theta })\frac{1}{\lambda^{n+1}}.\nonumber\end{equation} 
The rest of the proof is an application of the saddle point method \citep{bru81}. It consists of two steps: The first step is to pick a particular contour around $\lambda=0$ to carry out the integration. It is desirable to have a contour along which $g(\lambda)$ behaves violently: $g(\lambda)$ is large on a small interval on the contour and significantly smaller at the rest, so that the value of integral can be approximated by integrating over this small interval. Such a contour can be found,  by identifying a \emph{saddle point} of $g(\lambda)$ at which the derivative of $g(\lambda)$ vanishes, and then pick a contour that goes through the saddle point. %The direction that the contour goes through the saddle point is chosen so that $g(\lambda)$ is approximate real close to the saddle point, and achieves its maximum at the saddle point. 
The second step is to apply the Laplace method to estimate the integral along the contour. 

We now apply the first step of the saddle point method: identifying the saddle point and defining the contour for integration. Note that the derivative of $g$ is given by
\[\frac{d}{d \lambda} g(\lambda)=g(\lambda) [\sum_{j=1}^m \frac{\nu_j(e^{\theta }-1+e^{\lambda \nu_j})}{\lambda \nu_j(e^{\theta }-1)+e^{\lambda \nu_j} }-\frac{n+1}{\lambda}].\]
To simplify the derivation, we select a point that is close to a saddle point, defined as the solution to 
\begin{equation}\label{e:saddlepoint}
\sum_{j=1}^m \frac{\lambda \nu_j(e^{\theta }-1+e^{\lambda \nu_j})}{\lambda \nu_j(e^{\theta }-1)+e^{\lambda \nu_j} }=n.\end{equation}
If $\lambda$ on the left-hand side was taken to be a saddle point, then the right-hand side would be $n+1$ instead of $n$, and we will see this error is negligible for our purposes. % The  difference on the value of $\lambda$ between taking the right-hand side to be $n+1$ and $n$ is negligible. 

Equation \eqref{e:saddlepoint} has one unique real-valued nonnegative solution, which we denote by $\lambdab$. To see this, note that when restricting $\lambda$ to $[0, \infty)$, the left-hand-side is a continuous and strictly increasing  function of $\lambda$. Moreover, its value is $0$ when $\lambda=0$, increases to $\infty$ when $\lambda$ increases to $\infty$.
%Let $\lambdab$ be one solution to \eqref{e:saddlepoint}.

 We now obtain an asymptotic expansion of $\lambdab$. We first show that $\lambdab=O(n)$. When $\theta \geq 0$, using the fact that $0 \leq xe^{-x}\leq e^{-1}$ and $0 \leq e^{-x} \leq 1$ for $x \geq 0$, we obtain
\[\frac{1}{1+e^{-1}(e^{\theta}-1)}\leq \frac{e^{\theta }-1+e^{\lambda \nu_j}}{\lambda \nu_j(e^{\theta }-1)+e^{\lambda \nu_j}}\leq e^{\theta}.\]
Substituting this into \eqref{e:saddlepoint} leads to
\begin{equation}\label{e:boundlambda}
ne^{-\theta} \leq \lambdab\leq n(1+e^{-1}(e^{\theta}-1)).
\end{equation}
When $\theta <0$, we obtain
\[e^{\theta} \leq  \frac{e^{\theta }-1+e^{\lambda \nu_j}}{\lambda \nu_j(e^{\theta }-1)+e^{\lambda \nu_j}} \leq \frac{1}{1+e^{-1}(e^{\theta}-1)}.\]
Substituting this into \eqref{e:saddlepoint} leads to
\begin{equation}\label{e:boundlambda2}
n(1+e^{-1}(e^{\theta}-1))\leq \lambdab\leq ne^{-\theta} .
\end{equation}

%As a side note, we can conclude from the above bounds on $\lambdab$ that the real-valued solution of \eqref{e:saddlepoint} is \emph{unique} when $n$ is large enough: Any solution must be in the interval $[0,2n]$ and the left-hand side of \eqref{e:saddlepoint} is monotonically increasing on this interval.
It follows from the bounds \eqref{e:boundlambda}, \eqref{e:boundlambda2} and $\nu \in \bmPmub$ that $\lambdab \nu_j=o(1)$. Thus the demominator of \eqref{e:saddlepoint} satisfies 
\[\lambdab \nu_j(e^{\theta }-1)+e^{\lambdab \nu_j} =1+o(1).\] Substituting this into \eqref{e:saddlepoint} leads to
\[\sum_{j=1}^m \lambdab \nu_j(e^{\theta }-1+e^{\lambdab \nu_j})=n(1+o(1)).\]
Consequently, 
\begin{equation}\lambdab=n e^{-\theta}(1+o(1)).\nonumber%\label{e:lambdabestimatea}
\end{equation}
To obtain a refined approximation, let $w= \lambdab e^\theta /n-1$, which implies 
\begin{equation}\label{e:lambdawexpression}
\lambdab=n e^{-\theta} (1+w).
\end{equation} An approximation  for $w$ will be obtained: Since $\lambdab \nu_j=O(\frac{n}{m})$, we have that the numerator and denominator in the summand of \eqref{e:saddlepoint} satisfy
\[
\begin{aligned}
\lambdab \nu_j(e^{\theta }-1+e^{\lambdab \nu_j}) &=\lambdab \nu_j(e^\theta + \lambdab \nu_j +O(\frac{n^2}{m^2})),\\
\lambdab \nu_j(e^{\theta }-1)+e^{\lambdab \nu_j} &=1+ \lambdab \nu_j e^\theta +O(\frac{n^2}{m^2}).
\end{aligned}\]
Thus, 
\[\begin{aligned}
&\sum_{j=1}^m \frac{\lambdab \nu_j(e^{\theta }-1+e^{\lambdab \nu_j})}{\lambdab \nu_j(e^{\theta }-1)+e^{\lambdab \nu_j} }\\
=&\sum_j [\lambdab \nu_j e^\theta +\lambdab^2 \nu_j^2 (1-e^{2\theta})+O(\frac{n^3}{m^3})].\end{aligned}\]
Substituting this and \eqref{e:lambdawexpression}  into \eqref{e:saddlepoint} leads to
\[w+ n \sum_j \nu_j^2 (1+w)^2 (e^{-2\theta}-1)=O(\frac{n^2}{m^3}),\]
which gives
\begin{equation}\label{e:wvalue}
w=n\sum_j \nu_j^2 (1-e^{-2\theta}) (1+O(\frac{n}{m}))=O(\frac{n}{m}).
\end{equation}
%In conclusion, we have the following refined approximation
%\begin{equation}\label{e:lambdaexp}
%\lambdab=n e^{-\theta}\bigl(1+n\sum_j \nu_j^2 (1-e^{-2\theta})(1+O(\frac{n}{m}))\bigr).
%\end{equation}
%Therefore, the saddle point of $g$ can be approximated by $ne^{-\theta}(1+w)$ where 
%\[w\eqdef n\sum_j \nu_j^2 (1-e^{-2\theta})
%\] 
The integration in \eqref{e:formulamgf2} is now carried out along the closed contour given by $\lambda=\lambdab e^{i\psi}=ne^{-\theta}(1+w) e^{i \psi}$:% from $\psi=-\pip$ to $\psi=\pip$:
\begin{equation}
\begin{aligned}
\Expect_\nu[\exp\{\theta (\kstattilde)\}]
=&e^{-\theta n}\frac{n!}{2\pori} \int_{-\pi}^\pi g(\lambdab e^{i \psi}) \lambdab e^{i\psi}d\psi\\
%&=&\frac{n!}{n^n 2\pip}\int_{-\pi}^\pi e^{n\theta} (1+w)^{-n} e^{-i n \psi}\prod_{j=1}^m (n (1+w) \nu_j(1-e^{-\theta })e^{i\psi}+e^{n(1+w)\nu_j e^{-\theta}e^{i\psi} }) d\psi\nonumber\\
=&\frac{n!}{2\pori}\lambdab^{-n} e^{-\theta n}\Rec\bigl[\int_{-\pi}^\pi h(\psi) d\psi\bigr].\label{e:formulamgf3}
\end{aligned}
\end{equation}
where
\begin{equation}\label{e:defh}
h(\psi)\eqdef e^{-i n \psi}\prod_{j=1}^m \bigl(\lambdab \nu_j(e^{\theta }-1)e^{i\psi}+e^{\lambdab \nu_je^{i\psi} }\bigr).
\end{equation}

We now apply the second step of the saddle point method: estimating the integral by the Laplace method. We begin with a rough estimate of $h(\psi)$.
It follows from $\lambdab=n^{-\theta}(1+o(1))$ that
\begin{equation}
\begin{aligned}
&h(\psi)\\
=&e^{-i n \psi}\prod_{j=1}^m \bigl(\lambdab \nu_j(e^{\theta }-1)e^{i\psi}\!+\!1\!+\!\lambdab \nu_je^{i\psi}\!+\!O(\frac{n^2}{m^2})\bigr)\\
=&e^{-i n \psi}\prod_{j=1}^m \bigl(1+\lambdab \nu_je^{\theta}e^{i\psi}+\!O(\frac{n^2}{m^2})\bigr)\\
%&=&e^{-i n \psi}\exp\{\sum_{j=1}^m \log \bigl(1+\lambdab \nu_je^{\theta}e^{i\psi}+O(\frac{n^2}{m^2})\bigr)\}\nonumber\\
=&e^{-i n \psi}\exp\{\sum_{j=1}^m \bigl(\lambdab \nu_je^{\theta}e^{i\psi}+O(\frac{n^2}{m^2})\bigr)\}\\
=&e^{-i n \psi}e^n \exp\{-n (1-e^{i\psi})+O(\frac{n^2}{m})\}.\end{aligned}\label{e:hestimate}
\end{equation}
Therefore, for any $\psi \neq 0$, $|h(\psi)|$ is exponentially smaller than the value of $h(\psi)$ at $\psi=0$. This suggests that the integral in \eqref{e:formulamgf3} can be approximated by integrating over a small interval around $\psi=0$. Split the integral in \eqref{e:formulamgf3} into three parts:
\begin{equation}\label{e:integratesplit}
\begin{aligned}
I_1&=\Rec[\int_{-\pi/3}^{\pi/3} h(\psi) d\psi],\\ I_2&=\Rec[\int_{-\pi}^{-\pi/3}\! h(\psi) d\psi],\\ I_3&=\Rec[\int_{\pi/3}^{\pi}\! h(\psi) d\psi].
\end{aligned}
\end{equation}
%\]
%where $\pip/3$ will be specified later. 

We first estimate $I_1$. Denote $H(\psi)=\log(h(\psi))$.  
Simple calculus gives
\begin{equation}
\begin{aligned}
H(\psi)\!=\!&-in\psi+\sum_{j=1}^m \log(\lambdab \nu_j(e^\theta-1)e^{i\psi}\\&\quad\qquad\qquad\qquad+\exp\{\lambdab \nu_j e^{i\psi}\}),\\
H'(\psi)\!=\!&-\!in\\&+\!i\!\sum_{j=1}^m\! \!\frac{\lambdab \nu_j(\!e^\theta\!-\!1\!)e^{i\psi}\!\!+\!\!\lambdab \nu_je^{i\psi}\!\exp\{\!\lambdab \nu_j e^{i\psi}\!\}}{\lambdab \nu_j(e^\theta-1)e^{i\psi}+\exp\{\lambdab \nu_j e^{i\psi}\}},\\
H''(\psi)\!=\!&-\sum_{j=1}^m\exp\{\lambdab \nu_j e^{i\psi}\} \\&\times\frac{1}{(\lambdab \nu_j(e^\theta-1)e^{i\psi}+\exp\{\lambdab \nu_j e^{i\psi}\})^2}\\
&\times \bigl(\lambdab \nu_j(e^\theta-1)e^{i\psi}(1\!-\!\lambdab \nu_je^{i\psi}\!+\!\lambdab ^2\nu_j^2e^{2i\psi}) \\&\qquad +\lambdab \nu_je^{i\psi}\exp\{\lambdab \nu_j e^{i\psi}\}\bigr).
\end{aligned}\label{e:J12formula}
\end{equation}
It is clear that $\Im(H(0))=0$. It follows from \eqref{e:saddlepoint} that $H'(0)=0$. 
%When $\psi=0$, it follows from \eqref{e:saddlepoint} that $H'(0)=0$. It is also clear that $\Im(H(0))=0$. 
Estimates of $\Rec(H(0))$ and $H''(\psi)$ are obtained from substituting \eqref{e:lambdawexpression} and \eqref{e:wvalue} into the expression of $H(\psi)$ and $H''(\psi)$ and applying asymptotic analysis. In sum, %Consequently,
\begin{equation}\label{e:estimateJ}
\begin{aligned}
&\Im(H(0))\!=\!0, \\
&\Rec(H(0))\!=\!n(1+w)\!-\!\half n^2(\sum_{j=1}^m\nu_j^2) (1-e^{-2\theta })\!+\!O(\frac{n^3}{m^2}),\\
&H'(0)=0, %Rec(H'(0))\!=\!0, \Im(H'(0))\!=\!0,\\
H''(\psi)\!=\!-ne^{i\psi}+O(\frac{n^2}{m}).
\end{aligned}
\end{equation}

To obtain an upper-bound on $I_1$, note that for large enough $n$ and for any $\psi \in [-\pi/3, \pi/3]$, we have $\Rec(H''(\psi))\leq -0.4n$. It then follows from the mean value theorem that
\[\Rec(H(\psi)) \leq H(0)-0.2n\psi^2.\]
Consequently,  for large enough $n$ and $m$, 
\begin{equation}\label{e:I1upper-bound}
\begin{aligned}
I_1&\leq e^{H(0)}\int_{-\pi/3}^{-\pi/3} e^{-0.2 n\psi^2} d\psi \\&\leq e^{H(0)} \int_{-\infty}^{\infty} e^{-0.2 n\psi^2} d\psi = e^{H(0)}\frac{\sqrt{\pi}}{\sqrt{0.2n}}.
\end{aligned}
\end{equation}

To obtain a lower-bound on $I_1$, we begin with a bound on $\Im(H''(\psi))$: Since $\Im(H''(\psi))=-n\sin(\psi)+O(\frac{n^2}{m})$, applying $|\sin(\psi)| \!\leq \!|\psi|$, we have that for large enough $n$, for any $\psi\!\in\! [-\pi/3, \pi/3]$, $|\Im(H''(\psi))|\leq 1.1n|\psi|$. It also follows from \eqref{e:estimateJ} that $\Rec(H''(\psi))\geq -1.1n$. Applying the mean value theorem, we conclude that there exists some $c>0$ such that for $\psi \in [-\pi/3, \pi/3]$, 
\[\begin{aligned}
\Rec(H(\psi)) &\geq H(0)-1.1n\psi^2, \\
 |\Im(H(\psi))|&\leq 1.1n|\psi|^3+c\frac{n^2}{m}\psi^2.\end{aligned}\]
Use the short-hand notation $t_n=0.1 \min\{n^{\!-\!1/3},\! \sqrt{m}/(\sqrt{c}n)\}$. For $\psi \in [-t_n,t_n]$, we have $\cos(\Im(H(\psi)))\geq 0.5$, and thus $\Rec(e^{H(\psi)}) \geq 0.5e^{\Rec(H(\psi))}$. The integration for $I_1$ is further split into three parts:
\[
\begin{aligned}
I_1=&\Rec[\int_{-\pi/3}^{-t_n}e^{H(\psi)}d\psi]+\Rec[\int_{t_n}^{\pi/3}e^{H(\psi)}d\psi]\\&+\Rec[\int_{-t_n}^{t_n}e^{H(\psi)}d\psi].\end{aligned}\]
The absolute value of the first term is upper-bounded as follows:
\begin{equation}\label{e:firsttermI1}
\begin{aligned}
|\int_{-\pi/3}^{-t_n}e^{H(\psi)}d\psi |&\leq e^{H(0)} \int_{-\infty}^{-t_n}e^{-0.2n\psi^2}d\psi\\&= t_n e^{H(0)} \int_{-\infty}^{-1}e^{-0.2nt_n^2\bar{\psi}^2}d \bar{\psi}\\
&\leq t_n e^{H(0)} \int_{-\infty}^{-1}e^{-0.2nt_n^2|\bar{\psi}|}d \bar{\psi}\\&=e^{H(0)}O(\frac{1}{nt_n})=e^{H(0)}o(\frac{1}{\sqrt{n}}).%\\
%&=e^{H(0)}\frac{1}{\sqrt{n}}o(1).
\end{aligned}
\end{equation}
The second term is bounded in a similar way.
The third term is lower-bounded as follows:
\begin{equation}
\begin{aligned}
&\Rec[\int_{-t_n}^{t_n}e^{H(\psi)}d\psi ]\\\geq& \int_{-t_n}^{t_n}0.5 e^{\Rec(H(\psi))} d\psi\geq 0.5 e^{H(0)}\int_{-t_n}^{t_n} e^{-1.1n \psi^2}d\psi\\
\geq& 0.5 e^{H(0)}[\int_{-\infty}^{\infty} e^{-1.1n \psi^2}d\psi-2\int_{-\infty}^{-t_n} e^{-1.1n \psi^2}d\psi]\\
\geq& 0.5 e^{H(0)} (\frac{\sqrt{\pi}}{\sqrt{1.1n}}+O(\frac{1}{nt_n}))= 0.5 e^{H(0)} \frac{\sqrt{\pi}}{\sqrt{1.1n}}(1+o(1)).\!
\end{aligned}\nonumber
\end{equation}
where the last inequality follows from an argument similar to \eqref{e:firsttermI1}.
Combining these bounds together, we obtain
\begin{equation}%\label{e:I1lower-bound}
\begin{aligned}
I_1 \geq& \Rec[\int_{-t_n}^{t_n}e^{H(\psi)}d\psi] -|\Rec[\int_{-\pi/3}^{-t_n}e^{H(\psi)}d\psi]|\\&-|\Rec[\int_{t_n}^{\pi/3}e^{H(\psi)}d\psi]| \\
\geq& e^{H(0)} \frac{0.5\sqrt{\pi}}{\sqrt{1.1n}}(1+o(1)).
\end{aligned}\nonumber
\end{equation}
Combing this and \eqref{e:I1upper-bound} leads to,
\begin{equation}\label{e:I1estimate}
I_1 = e^{H(0)} \frac{1}{\sqrt{n}}e^{O(1)}=e^{n(1+o(1))} \frac{1}{\sqrt{n}}e^{O(1)}.
\end{equation}
where the last equality follows from the estimate of $H(0)$ given in \eqref{e:estimateJ} and  \eqref{e:wvalue}.

We now estimate $I_2$ and $I_3$. 
For $\psi \in [-\pi, -\pi/3] \cup [\pi/3, \pi]$, we obtain from \eqref{e:hestimate} that $|h(\psi)|\leq \exp\{0.5n+O(\frac{n^2}{m})\}$, which implies $ 
%\[
\Rec[I_2]+\Rec[I_3] =O(e^{0.6n})$. This shows that $I_2$ and $I_3$ are much smaller than $I_1$. Thus, the integral in \eqref{e:formulamgf3} can be approximated by the estimate of $I_1$: Substituting \eqref{e:I1estimate} and \eqref{e:estimateJ} into \eqref{e:formulamgf3}, we obtain
\begin{equation}%\label{e:formulamgf4}
\begin{aligned}
&\Expect_{\nu}[\exp\{\theta (\kstattilde)\}]\\
=&\frac{n!}{2\pori}\lambdab^{-n} e^{-\theta n}I_1(1+o(1)) \\
=&\frac{n!}{2\pori}\lambdab^{-n} e^{-\theta n}e^{H(0)} \frac{1}{\sqrt{n}}e^{O(1)}(1+o(1))\\
=&\frac{n!}{n^n\sqrt{2\pori n} }\bigl(1+n\sum_j \nu_j^2(1-e^{-2\theta})+O(\frac{n^2}{m^2})\bigr)^{-n}\\
&\times \exp\{\half n^2(\sum_{j=1}^m\nu_j^2) (1-e^{-2\theta })+O(\frac{n^3}{m^2})\}e^{O(1)})\\
=&\frac{n!e^{n}}{n^n \sqrt{2\pori n}}\exp\{-\half n^2(\sum_{j=1}^m\nu_j^2)(1-e^{-2\theta})+O(\frac{n^3}{m^2})\}e^{O(1)}.
\end{aligned}\nonumber
\end{equation}
Stirling formula gives $\frac{n!e^{n}}{n^n 2\pi\sqrt{n}}=1+O(\frac{1}{n})$. The claim of the proposition is obtained on taking logarithm on both sides.
\end{proof}

\subsection{Approximation to the logarithmic moment generating function for distributions not in $\bmPmub$}\label{subsec:LambdanuB}
We also need to consider distributions in $\mP \setminus \bmPmub$. 
%Combining Proposition~\ref{T:ASYMLMGFK} and Lemma~\ref{T:WORSTMUJSQUARE} leads to an upper-bound on the generalized error exponent of missed detection $J_M$, which turns out to be tight. It only gives a lower-bound for the smaller set of alternatives $\bmPmub \cap \mP$. We also need to consider distributions in $\mP \setminus \bmPmub$. 
For any $\mu \in \mP \setminus \bmPmub$, the set of indices $\mathcal{S}_0\eqdef \{j \in [m]: \mu_j \geq \mub m^{-1}\}$ is non-empty.  Now fix a small constant $\eta>0$, and consider each index $j$ in $\mathcal{S}_0$ in two separate cases, according to whether $n \mu_j \geq \eta$. Denote \begin{equation}
\supseta(\mu)=\{j: n\mu_j \geq \eta\}, \quad \beta(\mu) =\sum_{j \in \supseta(\mu)} \mu_j.\nonumber
\end{equation} 

Proposition~\ref{T:UPPER-BOUNDAPPROXIMATELMGF} below addresses the case where $\beta(\mu)$ is large. It implies that the probability of missed detection associated with such a distribution is much smaller than that associated with the worst-case distributions: The probability decays exponentially fast with respect to $n$, which is larger than $n^2/m$. Proposition~\ref{T:UPPER-BOUNDAPPROXIMATELMGF2} considers the alternate case, and shows that if $\beta(\mu)$ is not large, then a bound similar to that in Proposition~\ref{T:ASYMLMGFK} holds. 
\begin{proposition}\label{T:UPPER-BOUNDAPPROXIMATELMGF}
For all sufficiently small $\eta>0$, any $\theta \in (0,0.5]$, and any $\underline{\beta}>0$, there exists $n_0$ such that for any $n>n_0$, and any $\nu$ satisfying  $\beta(\nu) \geq \underline{\beta}$, the following holds,
\[\Lambda_{\nu, \kstattilde}(\theta)\leq - \beta(\nu) \alpha(\theta) n,\]
where $\alpha(\theta) >0 $.
\end{proposition}
\begin{proposition}\label{T:UPPER-BOUNDAPPROXIMATELMGF2}
For any $\delta>0$, $\theta \in (0,0.5]$, $\overline{\eta}>0$, there exist $\eta \in (0,\overline{\eta})$, $\overline{\beta}>0$, and $n_0$ such that for any $n>n_0$, and any $\nu$ satisfying $\beta(\nu) \leq \overline{\beta}$, the following holds,
\[\Lambda_{\nu, \kstattilde}(\theta)\leq \half \frac{n^2}{m}(m\sum_{j \notin \supseta(\nu)}\nu_j^2)(e^{-2\theta}-1)(1-\delta).\]
%\end{compactnumerate}
\end{proposition}
The proofs of Proposition~\ref{T:UPPER-BOUNDAPPROXIMATELMGF} and Proposition~\ref{T:UPPER-BOUNDAPPROXIMATELMGF2} use steps similar to those leading to the upper-bound in Proposition~\ref{T:ASYMLMGFK}. However, the approximation given by  \eqref{e:lambdawexpression} and \eqref{e:wvalue} is no longer valid, so a different approximation is required. The conclusions on the existence and uniqueness of the solution $\lambdab$ and the bounds in \eqref{e:boundlambda} are still valid, and our proof starts from there. % and we start from there. 

To simplify the presentation, we use the following notation similar to the small ``$o$" notation: We write $x=o^\eta(1)$ whenever there exists a function $s(\eta)$  that does not depend on $\theta$, $n$, and $\nu$, such that $|x|\leq s(\eta)$ and $\lim_{\eta \rightarrow 0}s(\eta)=0$. 

%Fixed a sequence $\bdm$. We also assume that all $\theta$ considered in the proof satisfies $\theta \in [0,0.5]$. 
%Fixed a small $\eta>0$. 
Consider any $\eta$ and $\nu$. Write $\supseta=\supseta(\nu)$. For any $j \notin \supseta$, we obtain the expansion of the summand in \eqref{e:saddlepoint} via the mean value theorem:
%\begin{equation}\label{e:approxlambdaleqeta0}
%\frac{\lambdab \nu_j(e^{\theta }-1+e^{\lambda \nu_j})}{\lambdab \nu_j(e^{\theta }-1)+e^{\lambdab \nu_j} }=\lambdab \nu_j e^{\theta}+\lambdab^2 \nu_j^2(1-e^{2\theta})(1+r_1(\nu_j,\theta))
%\end{equation}
%where $r_1(\nu_j,\theta)$ satisfies the following: There exists a function $s_1(\epsy)$ such that $|r_1(\nu_j,\theta)|\leq s_1(\eta)$   for $\theta \in [0,1]$ and $\lim_{\eta \rightarrow 0}s_1(\eta)=0$. 
%Using this notation, the equation \eqref{e:approxlambdaleqeta0} can be written as: 
%\begin{equation}\label{e:approxlambdaleqeta}
\[
\frac{\lambdab \nu_j(e^{\theta }-1+e^{\lambdab \nu_j})}{\lambdab \nu_j(e^{\theta }-1)+e^{\lambdab \nu_j} }=\lambdab \nu_j e^{\theta}+\lambdab^2 \nu_j^2(1-e^{2\theta})(1+o^\eta(1)).
\]
%\end{equation}
For any $j \in \supseta$, the following equality holds:
%\begin{equation}\label{e:approxlambdageqeta}
\[
\frac{\lambdab \nu_j(e^{\theta }-1+e^{\lambdab \nu_j})}{\lambdab \nu_j(e^{\theta }-1)+e^{\lambdab \nu_j} }=D_j  \lambdab \nu_j e^{\theta},
\]
%\end{equation}
where  
\begin{equation}\label{e:equationD}
D_j \eqdef \frac{e^{-\theta }+e^{-\lambdab \nu_j}(1-e^{-\theta })}{1+\lambdab \nu_je^{-\lambdab \nu_j} (e^{\theta }-1)} \geq e^{-2\theta}.
\end{equation}
%Note that it is easy to see that $D_j \geq e^{-2\theta}$ for $\theta \geq 0$. 

Substituting these estimates into \eqref{e:saddlepoint} leads to
\begin{equation}%\label{e:lambdaequation2}
\lambdab(1+\!\sum_{j\in \supseta}\!\!\nu_j (D_j-1))e^{\theta}+\lambdab^2\! \sum_{j  \notin \supseta} \!\!\nu_j^2 (1-e^{2\theta})(1+o^\eta(1))=n.\nonumber
\end{equation}
Combining this with $\lambdab \sum_{j  \notin \supseta} \nu_j^2 \leq \eta\sum_{j  \notin \supseta} \nu_j \leq\eta$ gives,
%\begin{equation}\label{e:lambdavalue3}
\[
\lambdab=\frac{ne^{-\theta}}{1+\sum_{j\in \supseta}\nu_j (D_j-1)}(1+o^\eta(1)).
\]
%\end{equation}
We now substitute this estimate into the previous equation, and introduce a variable $w$ as before, 
\begin{equation}\label{e:lambdabw}
\lambdab=\frac{ne^{-\theta}}{1+\sum_{j\in \supseta}\nu_j (D_j-1)}(1+w).
\end{equation} 
We obtain
\begin{equation}\label{e:approximatew2}
\begin{aligned}
w=\frac{n\bigl(\sum_{j  \notin \supseta} \nu_j^2 (1-e^{-2\theta})\bigr)}{(1+\sum_{j\in \supseta}\nu_j (D_j-1))^2}(1+o^\eta(1))
=o^\eta(1).\end{aligned}
\end{equation}

In the proofs of both propositions, we integrate  \eqref{e:formulamgf2} along the closed contour corresponding to $\lambda=\lambdab e^{i\psi}$ from $\psi=-\pi$ to $\psi=\pi$, and use the same definition of $h(\psi)$ given in \eqref{e:defh} and $H(\psi)=\log(h(\psi))$. The integral is given in  \eqref{e:formulamgf3}  and our task is to estimate it. We now give the details. 

\begin{proof}[Proof of Proposition~\ref{T:UPPER-BOUNDAPPROXIMATELMGF}]
We first show that any $\psi$, 
\begin{equation}
\begin{aligned}
\Rec(H(\psi)) &\leq\!  H(0)\\
&\!=\!\sum_j [\lambdab \nu_j\! +\!\log\bigl(1\!+\!\lambdab \nu_j e^{-\lambdab \nu_j }( e^\theta-1)\bigr)]. \end{aligned}\label{e:Jpsibound}
\end{equation}
so that we only need to bound $H(0)$ to bound the integral in \eqref{e:formulamgf3}.
%bounding the integration beco
%Recall the expression of $H(\psi)$ given in \eqref{e:J12formula}.
%We now prove \eqref{e:Jpsibound}. 
For $\psi \!\in\! [-\half \pi,\! \half \pi]$, the summand in the expression of $\Rec(H(\psi))$ given in \eqref{e:J12formula} is bounded as follows:
\begin{eqnarray}
%\begin{aligned}
&&\Rec[\log\bigl(\lambdab \nu_j(e^\theta-1)e^{i\psi}+e^{\lambdab \nu_j e^{i\psi}}\bigr)]\nonumber\\
&=&\Rec[\log(e^{\lambdab \nu_j e^{i\psi}}\!)\!+\!\log\bigl(1\!+\!\lambdab \nu_j(e^\theta\!-\!1)e^{i\psi}e^{-\lambdab \nu_j e^{i\psi}}\bigr)]\nonumber \\
&\leq& \lambdab \nu_j \cos \psi+\log\bigl(1+\lambdab \nu_j e^{-\lambdab \nu_j \cos \psi}( e^\theta-1)\bigr).\label{e:Jpsiineq}
%\end{aligned}
\end{eqnarray}
The right-hand side is a convex function of $\cos \psi$ for $ \psi \in [-\half \pi, \half \pi]$. Thus, it achieves its maximum value at $\cos \psi=1$ or $\cos \psi=0$. Note that its value at $\cos \psi=1$ is exactly equal to the summand in $H(0)$. Moreover, we can show that its value at $\cos \psi=1$ is no smaller than its value at $\cos \psi=0$:
%\begin{equation}
\[
\begin{aligned}
&\lambdab \nu_j\!+\!\log\bigl(1\!+\!\lambdab \nu_j(e^\theta\!-\!1)e^{-\lambdab \nu_j}\bigr)\!-\!\log\bigl(1\!+\!\lambdab \nu_j(e^\theta-1)\bigr)
\\
&=\lambdab \nu_j+\log\bigl(\frac{1+\lambdab \nu_j(e^\theta-1)e^{-\lambdab \nu_j}}{1+\lambdab \nu_j(e^\theta-1)}\bigr)\\& \leq \lambdab \nu_j+\log(e^{-\lambdab \nu_j})=0,
\end{aligned}
\nonumber
\]
%\end{eqnarray}
where the inequality follows from $\theta \geq 0$. This leads to \eqref{e:Jpsibound} for $\psi \in [-\half \pi, \half \pi]$. 

For $\psi \in [-\pi, -\half \pi] \cup [\half \pi,\pi]$, we have $|e^{\lambdab \nu_j e^{i\psi}}| \leq 1$. Consequently, 
\[|\lambdab \nu_j(e^\theta-1)e^{i\psi}+e^{\lambdab \nu_j e^{i\psi}}|\leq 1+\lambdab \nu_j (e^\theta-1),\]
which leads to
\begin{equation}%\label{e:Jpsiminuspi}
\Rec[\log\bigl(\lambdab \nu_j(e^\theta\!-\!1)e^{i\psi}\!+\!e^{\lambdab \nu_j e^{i\psi}}\bigr)] \!\leq\! \log\bigl(1\!+\!\lambdab \nu_j (e^\theta\!-\!1)\bigr).\nonumber
\end{equation}
The right-hand side of the above equation is equal to the value of the right-hand side of \eqref{e:Jpsiineq} at $\cos \psi=0$, which has been shown in the previous paragraph to be smaller than $H(0)$. This leads to \eqref{e:Jpsibound} for $\psi \in [-\pi, -\half \pi] \cup [\half \pi,\pi]$.
% Combining \eqref{e:Jpsiminuspi} and \eqref{e:Jpsiineq} together, we have
%\begin{equation}\label{e:Jpsiineqoverall}
%\Rec[\log\bigl(\lambdab \nu_j(e^\theta-1)e^{i\psi}+e^{\lambdab \nu_j e^{i\psi}}\bigr)] \leq \lambdab \nu_j +\log\bigl(1+\lambdab \nu_j e^{-\lambdab \nu_j }( e^\theta-1)\bigr).
%\end{equation}
%This leads directly to \eqref{e:Jpsibound} since the left-hand side is the summand in $\Rec(H(\psi))$ and the right-hand side is the summand in $H(0)$.

We now approximate the right-hand side of \eqref{e:Jpsibound}: For $ j \notin \supseta$, we have %the following approximation holds:
%\begin{equation}\label{e:Jpsisummanda}
\[
\begin{aligned}
&\lambdab \nu_j +\log\bigl(1+\lambdab \nu_j e^{-\lambdab \nu_j }( e^\theta-1)\bigr)\\&=\lambdab \nu_j^{\theta}+\half \lambdab^2 \nu_j^2 (1-e^{2 \theta})(1+o^\eta(1)).
\end{aligned}
\]
%\end{equation}
For $j \in \supseta$, we have the inequality
%\begin{equation}\label{e:Jpsisummandb}
\[
\begin{aligned}
&\lambdab \nu_j +\log\bigl(1+\lambdab \nu_j e^{-\lambdab \nu_j }( e^\theta-1)\bigr) \\&\leq \lambdab \nu_je^\theta +\lambdab \nu_j (1-e^{-\lambdab \nu_j })(1- e^\theta).
\end{aligned}
\]
%\end{equation}
Substituting these two estimates,  \eqref{e:lambdabw}, and \eqref{e:Jpsibound}  into \eqref{e:formulamgf3} leads to
\begin{eqnarray}%\label{e:formulamgf4}
&&\quad\Expect_{\nu}[\exp\{\theta (\kstattilde)\}] \leq  \frac{n!}{2\pori}\lambdab^{-n} e^{-\theta n}  \exp\{H(0)\}\label{e:inequalH}\\
%&&= n!\lambdab^{-n} e^{-\theta n} \exp\{\sum_j [\lambdab \nu_j +\log\bigl(1+\lambdab \nu_j e^{-\lambdab \nu_j }( e^\theta-1)\bigr)]\}\nonumber\\
&&\leq n!\lambdab^{-n} e^{-\theta n}\nonumber \\&&\quad\times\exp\{\sum_{j \notin \supseta} [\lambdab \nu_j e^{\theta}+\half \lambdab^2 \nu_j^2 (1-e^{2 \theta})(1+o^\eta(1))]\}\nonumber\\
&&\quad \times \exp\{\sum_{j \in \supseta}[\lambdab \nu_j e^\theta +\lambdab \nu_j (1-e^{-\lambdab \nu_j })(1- e^\theta)]\}\nonumber\\
&&= \frac{n! e^n}{n^n } \bigl(1+\sum_{j\in \supseta}\nu_j (D_j-1)\bigr)^{n} (1+w)^{-n} \nonumber\\
&&\quad \times  \exp\{-\frac{ \half n^2 \sum_{j \notin \supseta}\nu_j^2 (1-e^{-2\theta})(1+o^\eta(1))}{1+\sum_{j\in \supseta}\nu_j (D_j-1)}\}\nonumber\\
&&\quad \times \exp\{n [\frac{(1\!+\!w)\!+\!\sum_{j\in \supseta}\!\nu_j(1\!-\!e^{-\lambdab \nu_j })(e^{-\theta}\!-\!1)}{1\!+\!\sum_{j\in \supseta}\nu_j (D_j\!-\!1)}]\}\nonumber\\
%&=& \exp\{-n+n\sum_j \nu_j \frac{ e^{-\theta}+ e^{-\lambdab \nu_j }( e^\theta-1)}{1+\sum_{j\in \supseta}\nu_j (D_j-1)}\}\nonumber\\
&&\leq   \frac{n! e^n}{n^n }\exp\{-n\log(1+w)\!+\!\frac{nw}{1\!+\!\sum_{j\in \supseta}\!\nu_j (D_j-1)}\}\nonumber\\
&&\quad \times  \exp\{-\frac{ \half n^2 \sum_{j \notin \supseta}\nu_j^2 (1-e^{-2\theta})(1+o^\eta(1))}{1+\sum_{j\in \supseta}\nu_j (D_j-1)}\}\nonumber\\
%\exp\{-\half n w (1+\sum_{j\in \supseta}\nu_j (D_j-1)) (1+o^\eta(1))\}\nonumber\\
&&\quad \times \exp\{n [\sum_{j\in \supseta}\nu_j (D_j-1) -1\nonumber\\
&&\qquad+\frac{1+\sum_{j\in \supseta}\nu_j(1-e^{-\lambdab \nu_j })(e^{-\theta}-1)}{1+\sum_{j\in \supseta}\nu_j (D_j-1)}]\}.\label{e:mgfbounda}
\end{eqnarray}
We now bound each exponential term on the right-hand side of \eqref{e:mgfbounda}. 
Applying \eqref{e:approximatew2} and the lower-bound on $D_j$ in \eqref{e:equationD} gives the following bound on the second term:
%A bound on the second exponential term can obtained from the using the fact $\sum_{j \notin \supseta}\nu_j^2 =o^\eta(1)/n$:
\begin{equation}\label{e:mgfboundasecondterm}
-\frac{ \half n^2 \sum_{j \notin \supseta}\nu_j^2 (1-e^{-2\theta})}{1+\sum_{j\in \supseta}\nu_j (D_j-1)}\leq -\half e^{-2\theta}nw(1+o^\eta(1)).
\end{equation}
%Substitute this into \eqref{e:mgfbounda}. 
The first exponential term satisfies 
\begin{equation}\label{e:mgfboundafirstterm}
-n\log(1+w)+\frac{nw}{1+\sum_{j\in \supseta}\nu_j (D_j-1)}= -n w o^\eta(1), 
\end{equation}
which follows from \eqref{e:equationD} and $w=o^\eta(1)$. Combining \eqref{e:mgfboundasecondterm} and \eqref{e:mgfboundafirstterm} implies that for small enough $\eta$, the sum of the first and second term is \emph{negative}. 

%(given in \eqref{e:approximatew2}). % and $\log(1+w) =o^\eta(1)$. 
%Combining these bounds together, and note that for small enough $\eta$, the first exponential term and second exponential term becomes negligible comparing to the third exponential term, we conclude 
%Also apply \eqref{e:mgfboundafirstterm} and \eqref{e:approximatew2}. For small enough $\eta$, the first exponential term and second exponential term in \eqref{e:mgfbounda} is dominated by the third:
%\[-n\log(1+w)+\frac{nw}{1+\sum_{j\in \supseta}\nu_j (D_j-1)}-\frac{ \half n^2 \sum_{j \notin \supseta}\nu_j^2 (1-e^{-2\theta})(1+o^\eta(1))}{1+\sum_{j\in \supseta}\nu_j (D_j-1)}]\leq \half n\underline{\beta} e^{-3\theta} t(\eta {e^{-\theta}})\]
%Therefore, we conclude
The exponent in the last term on the right-hand side of \eqref{e:mgfbounda} is bounded as follows:
\begin{eqnarray}
%\begin{aligned}
&&\sum_{j\in \supseta}\!\!\nu_j (D_j\!-\!1) \!-\!1\!+\!\frac{\!1\!+\!\sum_{j\in \supseta}\!\nu_j(\!1\!-\!e^{-\lambdab \nu_j })(e^{-\theta}\!-\!1)}{1\!+\!\sum_{j\in \supseta}\nu_j (D_j-1)}\nonumber\\
&\!=\!\!\!\!&\frac{\bigl(\sum_{j\in \supseta}\!\nu_j (D_j\!-\!1)\bigr)^2\!+\!\sum_{j\in \supseta}\!\nu_j(1\!-\!e^{-\lambdab \nu_j })(e^{-\theta}\!-\!1)}{1+\sum_{j\in \supseta}\nu_j (D_j\!-\!1)}\nonumber\\
&\!\leq\!\!\!\! &\frac{(\sum_{j \in \supseta}\nu_j)\sum_{j\in \supseta}\nu_j (D_j\!-\!1)^2}{1+\sum_{j\in \supseta}\nu_j (D_j-1)}\nonumber \\
&\!&\!\!\!\!+\frac{\sum_{j\in \supseta}\nu_j(1-e^{-\lambdab \nu_j })(e^{-\theta}-1)}{1+\sum_{j\in \supseta}\nu_j (D_j-1)}\nonumber \\
&\!\leq\! \!\!\!&\frac{\sum_{j\in \supseta}\nu_j [(D_j-1)^2+(1-e^{-\lambdab \nu_j })(e^{-\theta}-1)]}{1+\sum_{j\in \supseta}\nu_j (D_j-1)}\label{e:ncoefficienta}
%\end{aligned}
\end{eqnarray}
where the first inequality follows from Jensen's inequality and the second follows from $\sum_{j \in \supseta}\nu_j \leq 1$. 

We first bound the summand in the numerator on the right-hand side of \eqref{e:ncoefficienta}. Consider any $j \in \supseta$. Let  $x\eqdef \lambdab \nu_j$. Applying the formula of $D_j$ in  \eqref{e:equationD} gives% the following equality for the summand  \eqref{e:ncoefficienta}:
\begin{equation}\label{e:equationt}
\begin{aligned}
&(D_j-1)^2+(1-e^{-x})(e^{-\theta}-1)\\=&\frac{e^{-x}+e^{-\theta}-e^{-x}e^{-\theta}}{\bigl(1+xe^{-x}(e^\theta-1)\bigr)^2}\\&\times[(1-e^{-x})(e^{-\theta}-1)\!+\!\bigl(xe^{-x}(e^\theta-1)\bigr)^2].
\end{aligned}
\end{equation}
Let $t(x)=(1-e^{-x})(e^{-\theta}-1)+\bigl(xe^{-x}(e^\theta-1)\bigr)^2$. Note that $j \in \supseta$ implies $n \nu_j \geq \eta$, which combined with \eqref{e:boundlambda} implies $x = \lambdab \nu_j \geq \eta e^{-\theta}$. 
Since for $\theta \in (0,0.5]$, $t(x)$ is strictly decreasing on $[0,\infty)$, we obtain $t(x) \leq  t(\eta {e^{-\theta}})<0$.
Substituting this into \eqref{e:equationt} and using the elementary fact that 
\[\frac{e^{-x}+e^{-\theta}-e^{-x}e^{-\theta}}{\bigl(1+xe^{-x}(e^\theta-1)\bigr)^2}\leq e^{-3\theta},\] we obtain \[(D_j-1)^{2}+(1-e^{-x})(e^{-\theta}-1) \leq -e^{-3\theta}t(\eta e^{-\eta}).\]
The denominator of on the right-hand side of \eqref {e:ncoefficienta} is positive and upper-bounded by $1$ because  $D_j \leq 1$. Combining the bounds on the numerator and denominator gives a bound on the exponent in the last term on the right-hand side of \eqref{e:mgfbounda}
\begin{equation}%\label{e:mgfboundalastterm}
\begin{aligned}
&\sum_{j\in \supseta}\!\nu_j (D_j\!-\!1) \!-\!1\!+\!\frac{(1\!+\!\sum_{j\in \supseta}\nu_j(1-e^{-\lambdab \nu_j })(e^{-\theta}-1)}{1+\sum_{j\in \supseta}\nu_j (D_j-1)}\\&\leq - \beta(\nu) \alpha(\theta) \leq 0,\nonumber\end{aligned}
\end{equation}
where \[\alpha(\theta)\!=\!\frac{1}{3} e^{-3\theta}[(1-e^{-\eta e^{-\theta}})(e^{-\theta}-1)+ \bigl(\eta e^{-\theta}e^{-\eta e^{-\theta}}(e^\theta-1)\bigr)^2].\]
Combining this with \eqref{e:mgfboundasecondterm} and \eqref{e:mgfboundafirstterm}, and using the fact that the right-hand sides of \eqref{e:mgfboundasecondterm} \eqref{e:mgfboundafirstterm}  are negative, we obtain:
\[\Expect_{\nu}[\exp\{\theta (\kstattilde)\}] \leq \frac{n! e^n}{\sqrt{2\pori n}n^n } \sqrt{2\pori n} \exp\{ - n {\beta(\nu)} \alpha(\theta) \}.\]
Taking the logarithm on both side and applying Stirling's formula leads to
\[\Lambda_{\nu,\kstat}(\theta) \leq - n {\beta(\nu)} \alpha(\theta) + \half \log(2\pi n)+O(\frac{1}{n}).\]
Since $\beta(\nu) \geq \underline{\beta}$, the second term $\half \log(2\pi n)$ becomes negligible comparing to the first term for large $n$. This leads to the claim of the proposition.
\end{proof}

\begin{proof}[Proof of Proposition~\ref{T:UPPER-BOUNDAPPROXIMATELMGF2}]
We pick $\overline{\beta}$ so that $\overline{\beta}=o^\eta(1)$. It then follows that
\begin{equation}\label{e:boundDjbeta}
\sum_{j \in \supseta} \nu_j (D_j-1)=o^\eta(1).
\end{equation} Substituting this into \eqref{e:lambdabw} and \eqref{e:approximatew2} gives \begin{equation}\label{e:lambdavalue4}
\begin{aligned}
\lambdab&=ne^{-\theta}(1+o^\eta(1)), \\ %\label{e:approximatew3}
w&=n(\sum_{j  \notin \supseta} \nu_j^2) (1-e^{-2\theta})(1+o^\eta(1)).
\end{aligned}
\end{equation}

%Since we have an upper-bound on $\beta(\nu)$, we obtain $|\sum_{j \in \supseta} \nu_j (D_j-1)| \leq \obeta$. Consequently, 
The rest of the proof is similar to the proof of Proposition~\ref{T:ASYMLMGFK}. 
Applying \eqref{e:Jpsibound} to $j \in \supseta$, we obtain
\begin{equation}
\begin{aligned}
&|h(\psi)|\\\leq&   |e^{-i n \psi}\prod_{j\notin \supseta} \bigl(\lambdab \nu_j(e^{\theta }-1)e^{i\psi}+e^{\lambdab \nu_je^{i\psi} }\bigr)|\\&\times \prod_{j \in \supseta}\exp\{\lambdab \nu_j +\log\bigl(1+\lambdab \nu_j e^{-\lambdab \nu_j }( e^\theta-1)\bigr)\}\\
\leq &|e^{-i n \psi}|\\
&\times\exp\{\!\sum_{j \notin \supseta}\! \!\lambdab \nu_je^{\theta}\cos \psi (1\!+\!o^\eta(1))\!+\! \sum_{j \in \supseta} \!\!\lambdab \nu_j e^\theta\}\\
=&e^n \exp\{-n (1-\cos \psi+ o^\eta(1))\}.\label{e:hestimate2}
\end{aligned}
\end{equation} 
It is clear from \eqref{e:hestimate2} that the integrand is large at the interval around $0$. Thus, we again split the integral in \eqref{e:formulamgf3} into three parts $I_1$, $I_2$ and $I_3$ as in \eqref{e:integratesplit}. We will show later that $I_2$ and $I_3$ are much smaller than $I_1$. 

We first upper-bound $I_1$. Similar to \eqref{e:estimateJ}, we have 
\[\Im(H(0))=0, \Rec(H'(0))=0, \Im(H'(0))=0.\]
%We also have 
%\begin{equation}\label{e:lower-boundReJ02}
%\Re(H(0)) \geq n(1+o^\eta(1)).
%\end{equation}
We now estimate $H''(\psi)$, whose exactly formula is given in \eqref{e:J12formula}. Consider $j \in \supseta$. For $\psi \in [-\pi/3, \pi/3]$, we have the following inequality:
\[|1+\lambdab \nu_j(e^\theta-1)e^{i\psi}\exp\{-\lambdab \nu_j e^{i\psi}\}|\geq 1,\]
\[\begin{aligned}
|&\lambdab \nu_j(e^\theta-1)e^{i\psi}(1-\lambdab \nu_je^{i\psi}+\lambdab ^2\nu_j^2e^{2i\psi})\exp\{-\lambdab \nu_j e^{i\psi}\}\\&+\lambdab \nu_je^{i\psi}|\leq \!100 \lambdab \nu_j e^\theta.\end{aligned}\]
Substituting these into \eqref{e:J12formula}, we obtain the following inequality $|H''(\psi)| \leq 100\obeta n(1+o^\eta(1))=no^\eta(1)$.
%\[
%\begin{aligned}
%&\sum_{j\in\supseta}\!\!\frac{\lambdab \nu_j(e^\theta\! \!-\!\!1)e^{i\psi\!}(1\!\!-\!\!\lambdab \nu_je^{i\psi}\!\!+\!\!\lambdab ^2\nu_j^2e^{2i\psi}\!)\!\exp\{\lambdab \nu_j e^{i\psi}\!\}\!\!+\!\!\lambdab \nu_je^{i\psi}\!(\exp\{\lambdab \nu_j e^{i\psi\!}\}\!)^2}{(\lambdab \nu_j(e^\theta-1)e^{i\psi}+\exp\{\lambdab \nu_j e^{i\psi}\})^2}\\
%&\leq 100\obeta n(1+o^\eta(1))=no^\eta(1).\end{aligned}\]
Substituting this and the estimate \eqref{e:lambdavalue4} into the expression of $H''(\psi)$ leads to
\[H''(\psi)=-n (e^{i\psi}+o^\eta(1)).\] 
Note that the assumption of the proposition allows us to take very small $\eta$. We choose it small enough so that the term $o^\eta(1)$ in the above equation is smaller than $0.05$.  For large enough $n$ and any $\psi \in [-\pi/3, \pi/3]$, we have $\Rec(H''(\psi))\leq -0.4n$. It follows from the mean value theorem that
\[\Rec(H(\psi)) \leq H(0)-0.2n\psi^2.\]
Consequently,  for large enough $n$ and $m$, we have
\begin{equation}\label{e:I1upper-bound2}
\begin{aligned}
I_1&\leq e^{H(0)}\int_{-\pi/3}^{-\pi/3} e^{-0.4 \psi^2} d\psi\\&\leq e^{H(0)} \int_{-\infty}^{\infty} e^{-0.4 \psi^2} d\psi = e^{H(0)}\frac{\sqrt{\pi}}{\sqrt{0.4n}}.\end{aligned}\end{equation}

We now bound the tails $I_2$ and $I_3$. 
For $\psi \in [-\pi, -\pi/3] \cup [\pi/3, \pi]$, we obtain from \eqref{e:hestimate2} that $|h(\psi)|\leq \exp\{0.5n(1+o^\eta(1))\}$. Thus, for small enough $\eta$, we have
\[\Rec[I_2]+\Rec[I_3] =O(e^{0.6n}).\]

%which is exponentially smaller than $I_1$ in view of \eqref{e:lower-boundReJ02}. 

Substituting the estimate for $I_1$, $I_2$ and $I_3$ into \eqref{e:formulamgf3} gives
\begin{equation}%\label{e:formulamgf4}
\Expect_{\nu}[\exp\{\theta (\kstattilde)\}]
\leq \frac{n!}{\sqrt{1.6n\pi}} \lambdab^{-n} e^{-\theta n}e^{H(0)}(1+o(1)).\nonumber
\end{equation}
Note that the right-hand side is almost the same as \eqref{e:inequalH} except for the multiplication term $\frac{1}{\sqrt{1.6n\pi}}(1+o(1))$. Thus, we can bound it  using the right-hand side of \eqref{e:mgfbounda} after taking into account this additional multiplication term. We obtain 
%
%The bounds for the terms in \eqref{e:mgfbounda} are now different: The first exponential term is now   
%e:mgfbounda
%Thus, we can apply two intermediate results in the proof of Proposition~\ref{T:UPPER-BOUNDAPPROXIMATELMGF} to obtain
%\[\Expect_{\nu}[\exp\{\theta (\kstat\!-\!n)\}] 
%\leq\frac{n! e^n}{\sqrt{1.6n\pip}n^n } \exp\{-\frac{ \half n^2 \!\sum_{j \notin \supseta}\!\nu_j^2 (1\!-\!e^{-2\theta})(1\!+\!o^\eta(1))}{1+\sum_{j\in \supseta}\nu_j (D_j-1)}\}(1\!+\!o^\eta(1))\]
%\begin{equation}\label{e:propupperlgmf2ineq1}%\label{e:formulamgf4}
\[
\begin{aligned}
&\Expect_{\nu}[\exp\{\theta (\kstattilde)\}] \\\leq& \frac{n! e^n}{n^n \sqrt{1.6n\pi}} \exp\{-\frac{ \half n^2 \!\sum_{j \notin \supseta}\!\nu_j^2 (1\!-\!e^{-2\theta})(1\!+\!o^\eta(1))}{1+\sum_{j\in \supseta}\nu_j (D_j-1)}\}\\&\times (1\!+\!o^\eta(1)).\end{aligned}
\]
%\nonumber
%\end{equation}
%and 
%\begin{equation}\label{e:propupperlgmf2ineq2}\sum_{j\in \supseta}\nu_j (D_j-1) -1+\frac{1+\sum_{j\in \supseta}\nu_j(1-e^{-\lambdab \nu_j })(e^{-\theta}-1)}{1+\sum_{j\in \supseta}\nu_j (D_j-1)}\leq 0.
%\end{equation}
%The inequality \eqref{e:propupperlgmf2ineq2} is proved using an argument similar to that used in the proof of \Lemma{boundlastexponential}.
%Note that although we do not have a lower-bound on $\beta(\mu)$, the inequality in \eqref{e:tnegative} is still valid, and hence the previous inequality. 
Substituting 
%\eqref{e:propupperlgmf2ineq2}, \eqref{e:boundDjbeta} and \eqref{e:approximatew3} 
\eqref{e:boundDjbeta} and Stirling's formula into the right-hand side of the above inequality leads to
\[
\begin{aligned}
&\Expect_{\nu}[\exp\{\theta (\kstat\!-\!n)\}]\\\leq& \frac{1}{\sqrt{0.8}} \exp\{-\half n^2(\sum_{j \notin \supseta}\!\!\nu_j^2)(1\!-\!e^{-2\theta})(1\!+\!o^\eta(1))\}(1\!+\!o(1)).\!
\end{aligned}
\]
Taking logarithm on both sides gives the claim of this proposition.
\end{proof}

\subsection{Proof of \Theorem{KPERFORMANCE} and \Theorem{ratefunction}}
\begin{proof}[Proof of \Theorem{ratefunction}]
Let $\Lambda_{\smu}(\theta)$ be the limit of the logarithmic moment generating function of $\Lambda_{q^{(n)}, \kstattilde}$:
\begin{equation}
\Lambda_{\smu}(\theta)\eqdef \lim_{n \rightarrow \infty} \frac{m}{n^2}\Lambda_{\mu^{(n)}, \kstattilde}(\theta).\nonumber
\end{equation}
It follows from Proposition~\ref{T:ASYMLMGFK} that the limit exists and is given by the following $C^{1}$ function:
\[\Lambda_{\smu}(\theta)= \half(e^{-2\theta}-1)\kinf(\smu).\]
%Clearly, the function $\Lambda_{\mu}(\theta)$ is \emph{essentially smooth}, continuous,  convex, and real-valued over the entire real line. 
Denote its Fenchel-Legendre transformation \[\Lambda_{\smu}^*(t) \eqdef \sup_{\theta} [\theta t-\Lambda_{\smu}(\theta)].\] It follows from the G\"{a}rtner-Ellis Theorem \citep[Theorem 2.3.6]{demzei98} that
%\[
%\begin{aligned}
%-\liminf_{n \rightarrow \infty} \frac{m}{n^2} \log( \Prob_{\mu^*}(\ktest=0))&=\inf_{t \geq -\tau-1}\Lambda_1^*(t)\\ &=\Lambda^*(-\tau-1)=J_M^*(\tau).\end{aligned}\]
\[
\begin{aligned}
&-\limsup_{n \rightarrow \infty}\frac{m}{n^2} \log(\Prob_{\mu^{(n)}}\{\kstat\leq \Expect_{\pip}[\kstat]\!+\!\frac{n^2}{m}\tau\})\\
=&-\limsup_{n \rightarrow \infty}\frac{m}{n^2} \log(\Prob_{\mu^{(n)}}\{\kstattilde \geq -\Expect_{\pip}[\kstat]-n-\frac{n^2}{m}\tau\})\\
=&\inf_{t \geq -\tau-1}\Lambda_1^*(t)=\Lambda_1^*(-\tau-1)\\
=&\sup_{\theta \geq0} \{\theta(-1-\tau)-\half(e^{-2\theta}-1)\kinf(\smu)\}.
\end{aligned}
\]
where $-\tau-1$ is the normalized limit of $-\Expect_{\pip}[\kstat]-n-\frac{n^2}{m}\tau$ by \Lemma{expectK}.
\end{proof}

\begin{proof}[Proof of \Theorem{KPERFORMANCE}]
The proof for the result on  the generalized error exponent of false alarm $J_F(\ktest)$ is very similar to that of \Theorem{ratefunction}. Let $\Lambda_{0}(\theta)$ be the limit of the logarithmic moment generating function of $\Lambda_{\pip, \kstattilde}$:
\begin{equation}
\Lambda_{0}(\theta)\eqdef \lim_{n \rightarrow \infty} \frac{m}{n^2}\Lambda_{\pip, \kstattilde}(\theta).\nonumber
\end{equation}
It follows from Proposition~\ref{T:ASYMLMGFK} that the limit exists and is given by the following $C^{1}$ function:
\[\Lambda_{0}(\theta)= \half (e^{-2\theta}-1).\]
Let $\Lambda_0^*(t)=\sup_{\theta}[\theta t-\Lambda_0(\theta)]$. It follows from the G\"{a}rtner-Ellis Theorem that 
\[
\begin{aligned}
&-\limsup_{n \rightarrow \infty} \frac{m}{n^2} \log( \Prob_{\pip}(\ktest_n=1))\\
=&-\limsup_{n \rightarrow \infty}\frac{m}{n^2} \log(\Prob_{\pip}\{\kstattilde \leq -\Expect_{\pip}[\kstat]-n-\frac{n^2}{m}\tau\})\\
=&\inf_{t \leq -\tau-1}\Lambda_0^*(t)=\Lambda_0^*(-\tau-1)\\
=&\sup_\theta\{\theta(-\tau-1)-\half (e^{-2\theta}-1)\}=J_F^*(\tau).\end{aligned}\]

For the result on the generalized error exponent of missed detection $J_M(\ktest)$, we prove an upper-bound and a lower-bound. For the upper-bound, consider the sequence of distributions given in \eqref{e:worstmu2a} and \eqref{e:worstmu1} and let $\smu^*$ denote this sequence. The rate function associated with $\smu^*$ satisfies 
\[ J_{\smu^*}(\ktest, \tau)=J^*_M(\tau).\]
On the other hand, since each element of $\smu^*$ is in the set of alternative distributions, it follows from the definition of $J_M(\ktest)$ and $J_{\smu^*}(\ktest, \tau)$  that
\[J_M(\ktest) \leq J_{\smu^*}(\ktest, \tau)\]

To obtain the lower-bound on $J_M(\ktest)$, we apply Proposition~\ref{T:UPPER-BOUNDAPPROXIMATELMGF} and Proposition~\ref{T:UPPER-BOUNDAPPROXIMATELMGF2} . We only need to prove it for the case $\tau \in [0,\taul(\epsy))$. The case $\tau=\taul(\epsy)$ then follows from a continuity argument. 

Take $\theta_0$ to be the maximizer in the optimization problem defining $J_M^*(\tau)$ in \eqref{e:Kperformance}. It is not difficult to see that $\theta_0 > 0$. It follows from Lemma~\ref{T:WORSTMUJSQUARE} that %, we obtain for $\mu \in \mP$ satisfying $\beta(\mu)\leq \obeta=o^\eta(1)$, 
\[m\sum_{j \notin \supseta}\mu_j^2 \geq(1+\taul(\frac{\epsy-\beta(\mu)}{1-\beta(\mu)}))(1-\beta(\mu))(1+o(1)).\]
Thus, for any $\delta>0$, we can choose $\eta, \beta_0$ small enough so that for any $\mu \in \mP$ satisfying $\beta(\mu)\leq \beta_0$, it holds that $m\sum_{j \notin \supseta}\mu_j^2 \geq (1+\taul(\epsy))(1-\delta)$. 
It then follows from Proposition~\ref{T:UPPER-BOUNDAPPROXIMATELMGF2} that for large enough $n$, 
\begin{equation}
%\begin{aligned}
\Lambda_{\mu
,\kstattilde}(\theta_0)\leq \half \frac{n^2}{m}(1+\taul(\epsy))(e^{-2\theta_0}-1)(1-\delta)^2+O(1).
%\end{aligned}
\label{e:lgmf_muleqbetabound}
\end{equation}
For $\mu$ satisfying $\beta(\mu) \geq \beta_0$, it follows from Proposition~\ref{T:UPPER-BOUNDAPPROXIMATELMGF} that for large enough $n$,
\begin{equation}
\Lambda_{\mu
,\kstattilde}(\theta_0)\leq -\beta_0 \alpha(\theta_0) n. \label{e:lgmf_mugeqbetabound}
\end{equation}
We can pick $n$ large enough so that the right-hand side of \eqref{e:lgmf_mugeqbetabound} is smaller than the right-hand side of \eqref{e:lgmf_muleqbetabound}. 
%The first claim in 
%Consequently, for any $\delta>0$, we can choose $\eta, \beta_0$ small enough, so that for any $\mu \in \mP$ satisfying $\beta(\mu)\leq \beta_0$, we have from Proposition~\ref{T:UPPER-BOUNDAPPROXIMATELMGF2} that
%\[\Lambda_{\mu
%,\kstat}(\theta_0)\leq \frac{n^2}{m}(1+\taul(\epsy))\{-\theta_0 +\half[e^{-2\theta_0}-(1-2\theta_0)]\}(1-\delta)+O(1).\]
%Proposition~\ref{T:UPPER-BOUNDAPPROXIMATELMGF} implies that any $\mu$ satisfying $\beta(\mu) \geq \beta_0$ has a probability of missed detection that is much smaller, and thus does not affect the generalized error exponent.
Applying the Chernoff bound leads to
 \begin{equation}
 \begin{aligned}
&\log (\sup_{\mu \in \mP}\Prob_\mu(\ktest_n=0))\\
 \leq& -\theta_0 (\Expect_{\pip}[\kstattilde]-\tau_n)+\sup_{\mu \in \mP}\Lambda_{\mu
,\kstattilde}(\theta_0)\\
\leq& \theta_0(\tau_n\! -\!\Expect_{\pip}[\kstattilde]) \!+\!\half\frac{n^2}{m}(1\!+\!\taul(\epsy))(e^{-2\theta_0}-1)(1\!-\!\delta)^2\!+\!O(1).
\end{aligned}\nonumber
\end{equation}
Thus, 
\[
%\begin{aligned}
J_M(\ktest) \geq \theta_0(-1-\tau)-\half( e^{-2\theta_0}-1)(1+\taul(\epsy))(1-\delta)^2.
%\end{aligned}
\]
This holds for any $\delta>0$. Consequently,  $J_M(\ktest) \geq J_M^*(\tau)$. 
\end{proof}

\section{Proofs of \Theorem{modifiedK} and Theorem~\ref{T:TBOUND}}\label{apx:extension}
\subsection{Proof of \Theorem{modifiedK}}
The performance of $\ktesta$ is analyzed by connecting it to the performance of $\ktest$. We first show that its probability of missed detection is no larger than that of $\ktest$. We then apply a result similar to Proposition~\ref{T:ASYMLMGFK} to analyze its probability of false alarm. 
Consider the statistic 
\[\kstatatilde=-\kstata-n.\] 
Define
\begin{equation}
\Lambda_{\nu, \kstatatilde}(\theta)\eqdef \log\bigl(\Expect_{\nu}[\exp\{\theta (\kstatatilde)\}]\bigr).
\end{equation}
\begin{proposition}\label{T:ASYMLMGFKstata}
For any $\nu \in \bmPmub$, the logarithmic moment generating function for the statistic $\kstatatilde$ has the following asymptotic expansion
\begin{equation}%\label{e:lmgfKestimate1stata}
\begin{aligned}
\Lambda_{\nu, \kstatatilde}(\theta)=&\frac{n^2}{m}\bigl(m\sum_{j=1}^m\nu_j^2\bigr)\{-\theta +\half[e^{-2\theta}-(1-2\theta)]\}\\&+O(\frac{n^3}{m^2})+O(1).\nonumber
\end{aligned}
\end{equation}
%\qed
\end{proposition}
\begin{proof}[Proof of Proposition~\ref{T:ASYMLMGFKstata}]
The proof follows exactly the same step as that of Proposition~\ref{T:ASYMLMGFK} except some of the approximations are different. We now only describe the key steps and highlight the difference: First, the estimate of the saddle point is the same as \eqref{e:lambdawexpression} and \eqref{e:wvalue}. Second, different from \eqref{e:formulamgf3}, we have the following expression of the moment generating function: 
\begin{equation}
\begin{aligned}\label{e:formulamgf3kstata}
\Expect_\nu[\exp\{\theta (\kstatatilde)\}]=\frac{n!}{2\pori}\lambdab^{-n} e^{-\theta n}\Rec\bigl[\int_{-\pi}^\pi h(\psi) d\psi\bigr]\nonumber.
\end{aligned}
\end{equation}
where instead of \eqref{e:defh},  
\[
\begin{aligned}
h(\psi)\eqdef e^{-i n \psi}\prod_{j=1}^m \bigl(&\lambdab \nu_j(e^{\theta }-1)e^{i\psi}\\&+e^{\lambdab \nu_je^{i\psi} }+ \sum_{l=2}^{\bar{l}} \frac{(\lambda_0\nu_j)^l}{l!}(e^{\theta v_l}-1)\bigr).
\end{aligned}
\]
It follows from $\lambdab=n^{-\theta}(1+o(1))$ that the last term is negligible when $v_2=0$ and $\bar{l}< \infty$. 
\[\sum_{l=2}^{\bar{l}} \frac{(\lambda_0\nu_j)^l}{l!}(e^{\theta v_l}-1)=O(\frac{n^3}{m^3})\]
The asymptotic approximation of $h(\psi)$ is the same as that in \eqref{e:hestimate}:
\begin{equation}
%\begin{aligned}
h(\psi)
\!=\!e^{-i n \psi}\prod_{j=1}^m \bigl(\lambdab \nu_j(e^{\theta }\!-\!1)e^{i\psi}\!+\!1
\!+\!\lambdab \nu_je^{i\psi}\!+\!O(\frac{n^2}{m^2})\bigr).\nonumber
%\end{aligned}
\end{equation}
Finally, the approximations of $H(0), H'(0), H''(\psi)$ are the same as in \eqref{e:estimateJ}. 
Therefore,  $\Lambda_{\nu, \kstatatilde}$ has the same asymptotic approximation as that of  $\Lambda_{\nu, \kstattilde}$ up to an approximation error of $O(\frac{n^3}{m^2})$.
\end{proof}

\begin{proof}[Proof of \Theorem{modifiedK}]
Since $v_l \geq 0$ for $l \geq 2$, we have 
\[\kstata \geq \kstat.\]
Thus, for the same sequence of thresholds $\tilde{\tau}_n$, we have
\[\Prob_{\mu}\{\kstata \leq \tilde{\tau}_n\} \leq  \Prob_{\mu}\{\kstat \leq \tilde{\tau}_n\}\]
On the other hand, since $\Lambda_{\nu, \kstatatilde}$ has the same asymptotic approximation as that of  $\Lambda_{\nu, \kstattilde}$ up to an approximation error of $O(\frac{n^3}{m^2})$, we have
\begin{equation}%\label{e:Chernoffpibound}
\begin{aligned}
& \log \Prob_\pip\{\kstata \geq -n+\tilde{\tau}_n\}\\
=&\log \Prob_\pip\{\kstatatilde \leq -\tilde{\tau}_n\}\\
\leq& \theta (-\tilde{\tau}_n)+\Lambda_{\pip,\kstatatilde}(-\theta)\\
%&=-\theta \tau_n\!+\!\theta \Expect_{\pip}[\kstat\!-\!n]\!-\!n^2(\sum_{j=1}^m\pip_j^2)(-\theta) \!+\!n^2(\sum_{j=1}^m\pip_j^2)\half[e^{2\theta}\!-\!(1\!+\!2\theta)]\!+\!O(\frac{n^3}{m^2})\!+\!O(1)\\
=&-\theta \tilde{\tau}_n+\frac{n^2}{m}\bigl(\theta+\half[e^{2\theta}-(1+2\theta)]\bigr)+O(\frac{n^3}{m^2})+O(1).
\end{aligned}\nonumber
\end{equation}
which is the same bound as that for $\log \Prob_\pip\{\kstat \geq -n+\tilde{\tau}_n\}$.
\end{proof}

\subsection{Proof of Theorem~\ref{T:TBOUND}}\label{apx:nonuniformK}
The proof of Theorem~\ref{T:TBOUND} follows exactly the same steps as those in the proof of \Theorem{KPERFORMANCE}. We use \Proposition{asymlmgfT}, \Proposition{upper-boundapproximatelmgfT} and \Proposition{upper-boundapproximatelmgfT2} in place of 
Proposition~\ref{T:ASYMLMGFK}, Proposition~\ref{T:UPPER-BOUNDAPPROXIMATELMGF} and Proposition~\ref{T:UPPER-BOUNDAPPROXIMATELMGF2}. 

Denote
\[\Lambda_{\nu, \tstat}(\theta)\eqdef \log\bigl(\Expect_{\nu}[\exp\{\theta \tstat\}]\bigr).\] 
\begin{proposition}\label{t:asymlmgfT}
For any $\nu \in \bmPmub$, the logarithmic moment generating function for the statistic $\tstat$ has the following asymptotic expansion
\begin{equation}%\label{e:lmgfTupper}
\begin{aligned}
\Lambda_{\nu,\tstat}(\theta)=&\half n^2(\sum_{j=1}^m(\pip_j-\nu_j)^2) \theta\!+\!\half n^2(\sum_{j=1}^m\nu_j^2) [e^{\theta}\!-\!(1\!+\!\theta)]\\ &+O(\frac{n^3}{m^2})+O(1).\nonumber
\end{aligned}
\end{equation}
%\qed
\end{proposition}
\begin{proposition}\label{t:upper-boundapproximatelmgfT}
For all sufficiently small $\eta>0$, any $\theta \in [-1,0)$ and any $\underline{\beta}>0$. There exists $n_0$ such that for any $n>n_0$, and any $\nu$ satisfying  $\beta(\nu) \leq \underline{\beta}$, the following holds,
\[\Lambda_{\nu,\tstat}(\theta)\leq - \beta(\mu) \alpha'(\theta) n\]%  + \half \log(2\pip n)\] 
where $\alpha'(\theta)>0$ for $\theta \in [-1,0)$.%
%\item 
\end{proposition}
\begin{proposition}\label{t:upper-boundapproximatelmgfT2}
%There exists a function $h(\eta)$ satisfying  $\lim_{\eta \rightarrow 0}h(\eta)=0$ such that the following holds:
For any $\delta>0$, $\theta \in [-1,0)$, $\overline{\eta}>0$, there exists $\eta \in (0, \overline{\eta})$, $\overline{\beta}>0$, and $n_0$ such that for any $n>n_0$, and any $\nu$ satisfying $\beta(\mu) \leq \overline{\beta}$, the following holds, 
\[
\begin{aligned}
\Lambda_{\nu,\tstat}(\theta)\leq \frac{n^2}{m}[&(m\!\!\sum_{j \notin \supseta(\nu)}\!(\pip_j-\nu_j)^2)\theta\\&+\half(m\!\!\sum_{j\notin \supseta(\nu)}\!\!\nu_j^2)(e^{\theta}-(1+\theta))](1-\delta).
\end{aligned}\]
%\end{compactnumerate}
\end{proposition}
We only outline the proof for \Proposition{asymlmgfT}. 
\begin{proof}[Proof of \Proposition{asymlmgfT}]
The steps are the same as those in the proof of  Proposition~\ref{T:ASYMLMGFK}. Again, we describe the main steps and highlight the difference. First, the estimate of the saddle point is different than that in \eqref{e:lambdawexpression} and \eqref{e:wvalue}. We have
\begin{equation}%\label{e:lambdaexpT}
\begin{aligned}
\lambdab&=n(1+w), \\w&=n(\sum_j \nu_j \pip_j \theta- \sum_j \nu_j^2 (e^\theta -1)) (1+O(\frac{n}{m})).
\end{aligned}\nonumber
\end{equation}

Second, different from \eqref{e:formulamgf3}, we have the following expression of the moment generating function: 
\[\Expect_\nu^n[\exp\{\theta \tstat\}]=\frac{n!}{2\pori}\lambdab^{-n}\Rec[\int_{-\pi}^\pi h(\psi)d\psi]\]
where
\begin{equation}
\begin{aligned}
h(\psi)&= e^{\!-i n \psi}\!\prod_{j=1}^m [\exp\{\lambdab \nu_je^{i\psi}\!\}\!+\!(e^{\half n^2\pip_j^2 \theta}\!-\!1)\\&\quad\qquad\qquad+\!\lambdab e^{i\psi} \nu_j(e^{\!-n\pip_j \theta\!}\!-\!\!1)\!+\!\half \lambdab^2 e^{2i\psi\!} \nu_j^2 (e^{\theta}\!\!-\!\!1) ]\\%\bigr)\\
%&=&e^{-i n \psi}\prod_{j=1}^m \bigl(1+\lambdab \nu_je^{i\psi}+O(\frac{n^2}{m^2})\bigr)\nonumber\\
&=e^{-i n \psi}\exp\{n e^{i\psi}+O(\frac{n^2}{m})\},\end{aligned}\nonumber%\label{e:hestimateT}
\end{equation}

Finally, the approximation of $\Rec(H(0))$ is different from that in \eqref{e:estimateJ}
\[\begin{aligned}
\Rec(H(0))\!=&n (1\!+\!w)\!+\!\half n^2(\sum_{j=1}^m(\pip_j\!-\!\nu_j)^2) \theta\\&+ \!\half n^2(\sum_{j=1}^m\nu_j^2) (e^\theta\!-\!1\!-\!\theta)+\!O(\frac{n^3}{m^2}).\!\end{aligned}\]
The rest of the steps are the same as those in Proposition~\ref{T:ASYMLMGFK}. 
\end{proof}
\begin{proof}[Proof of Theorem~\ref{T:TBOUND}]
First, we prove the lower-bound on $J_F$. Substituting the asymptotic approximation of $\Lambda_{\pip,\tstat}(\theta)$ given in \Proposition{asymlmgfT} into the Chernoff bound, we obtain for $\theta \geq0$, 
\begin{equation}%\label{e:ChernoffpiboundT0}
\begin{aligned}
&\log \Prob_\pip(\ttest_n=1)\\
 \leq& -\theta \tau_n+\Lambda_{\pip,\tstat}(\theta)\nonumber\\
=&-\theta \tau_n+n^2(\sum_{j=1}^m\pip_j^2)\half[e^{\theta}-(1+\theta)]+O(\frac{n^3}{m^2})+O(1).\end{aligned}\nonumber\end{equation}
Since $m\sum_{j=1}^m \pip_j^2\leq \mub^2$, which is a consequence of \Assumption{nonuniform}, we have
\[J_F(\ttest) \geq  \sup_{\theta \geq 0} \{\half \tau \theta -  \half \mub^2 [e^{\theta}-(1+\theta)]\}>0.\]

Lower-bounding $J_M(\ttest)$ requires us to obtain a uniform bound on the probability $\Prob_\mu(\phi_n=0)$ over $\mu \in \mP$. We apply \Proposition{upper-boundapproximatelmgfT} and \Proposition{upper-boundapproximatelmgfT2}. Using an argument similar to the proof in \Theorem{KPERFORMANCE}, we conclude that for any $\delta>0$, and $\theta \in (0,1]$, for large enough $n$, 
\begin{equation}
\begin{aligned}
&\log \Prob_\mu(\ttest_n\!=\!0)\\
\!\leq& \theta \tau_n+\Lambda_{\mu,\tstat}(-\theta)\\
\!=&\theta \tau_n\!\!-\!\!\frac{n^2\!}{m} [\half\theta m\!\sum_{j=1}^m\! (\mu_j\!\!-\!\pip_j)^2\!-\!\half(m\!\sum_{j=1}^m\!\mu_j^2)\bigl(e^{\!-\theta\!}\!\!-\!(1\!\!-\!\!\theta)\bigr)](1\!-\!\delta).%\label%{e:Tchernoffmu}
%&\leq&\theta \tau_n-n^2 (\sum_{j=1}^m(\mu_j-\pip_j)^2) \half\theta+n^2[(\sum_{j=1}^m \pip_j^2)+(\sum_{j=1}^m(\mu_j-\pip_j)^2)][e^{-\theta}-(1-\theta)]+O(\frac{n^3}{m^2})+O(1)\nonumber
\end{aligned}\nonumber
\end{equation}
We need to upper-bound the right-hand side uniformly over all $\mu \in \mP$. Using the inequalities 
$\mu_j^2 \leq 2\pip_j^2+2(\pip_j-\mu_j)^2$
and $e^{-\theta}-(1-\theta) \leq \half \theta^2$ for $\theta >0$, we obtain
\begin{equation}
\begin{aligned}
&\frac{m}{n^2}\log \Prob_\mu(\ttest_n=0) \nonumber\\
\leq& \theta \frac{m\tau_n\!}{n^2}\\
&\!-\![\half \theta m\!\sum_{j=1}^m(\!\mu_j\!-\!\pip_j\!)^2\!\!-\!\! \half\theta^2 \!\bigl(m\!\sum_{j=1}^m \pip_j^2\!+\!\!m\!\sum_{j=1}^m\!(\mu_j\!-\!\pip_j)^2\bigr)](1\!-\!\delta)\\&\!+\!O(1)\nonumber\\
=&\half \theta [-(m\sum_{j=1}^m(\mu_j-\pip_j)^2)(1-\theta)\\&\qquad+\theta (m\sum_{j=1}^m \pip_j^2)](1-\delta) +\theta\frac{m\tau_n}{n^2}+O(1).%\label{e:Tchernoffmu2}
\end{aligned}\nonumber
\end{equation}
Applying $m\sum_{j=1}^m (\mu_j-\pip_j)^2 \geq 4\epsy^2$ and $m\sum_{j=1}^m \pip_j^2\leq \mub^2$ leads to, 
%For $\theta \leq1$, the first term on the right-hand side of \eqref{e:Tchernoffmu2} is non-increasing with $m\sum_{j=1}^m(\mu_j-\pip_j)^2$. The latter is no smaller than $4\epsy^2$. the second term on the right-hand side of \eqref{e:Tchernoffmu2} is no larger than $\mub^2$. The  Combining this with \eqref{e:lower-boundTdiff}, 
\[%\begin{equation}
\frac{m}{n^2}\log [P_M(\ttest_n)] \!\leq\! \half \theta [-4\epsy^2(1-\theta)+\theta \mub^2](1-\delta)+\frac{m\tau_n}{n^2}+O(1).%\label{e:Tchernoffmu3}
\]
%\end{equation}
Taking $\theta=(4\epsy^2(1-\delta)-2\tau)/[(8\epsy^2+2\mub^2)(1-\delta)]$, and taking the limit on both sides gives
\[
J_M(\ttest) \geq \frac{1}{4} 4\epsy^2\frac{4\epsy^2(1-\delta)-2\tau}{(8\epsy^2+2\mub^2)(1-\delta)}.%\label{e:Tchernoffmu4}
\]%\end{equation}
Since this holds for all $\delta>0$, and $2\tau < 4\epsy^2$, we conclude that 
\[
J_M(\ttest) \geq \frac{1}{4} 4\epsy^2\frac{2\epsy^2-\tau}{(8\epsy^2+2\mub^2)(\half+\tau/(4\epsy^2))}>0.%\label{e:Tchernoffmu4}
\]
\end{proof}

\section{Proof of \Theorem{converse}}\label{apx:converse}

We first give an outline of the proof: Consider any $\tau \in [0, \taul(\epsy)]$. %Suppose there is a test $\phi$ such that $J_F(\phi)\geq J_F^*(\tau)$. The goal is to prove that $J_M(\phi) \leq J_M^*(\tau)$.  
Given $\delta>0$, a sequence of events $\{B_{n, \tau,\delta}\}$ is constructed so that the following is satisfied:
\begin{list}{{\upshape (\roman{rmnum})}}{\usecounter{rmnum}}
\item The probability of the event is close to the probability of false alarm: 
\begin{equation}\label{e:boundPA}
\limsup_{n \rightarrow \infty}-\frac{m}{n^2}\log(\Prob_\pip(B_{n, \tau,\delta}))\leq J_F^*(\tau)-\delta.
\end{equation}
\item For any $\bfmz_1^n$ satisfying $\{\bfmZ_1^n=\bfmz_1^n\}\subseteq B_{n, \tau,\delta}$, the following uniform bound on the likelihood ratio holds:
\begin{equation}\label{e:boundLLRonA}
\sup_{\mu \in \mP}\frac{\mu^n}{\pip^n}(\bfmz_1^n) \geq  \exp\{-\frac{n^2}{m}(J_M^*(\tau)-J_F^*(\tau)+\delta)\}.
\end{equation}
\end{list}
The lower-bound on $P_M$ is then obtained from the following inequality:
\begin{equation}\label{e:PMboundPForigin}
\begin{aligned}
&P_M(\phi_n)\\
\geq& \sup_{\mu \in \mP}\!\Prob_{\mu}\bigl(\{\phi_n=0\} \cap B_{n, \tau,\delta}\bigr)\\
\geq& \sup_{\mu \in \mP}\!\frac{\mu^n}{\pip^n}(\{\phi_n\!=\!0\} \!\cap\! B_{n, \tau,\delta})\Prob_\pip(\{\phi_n\!=\!0\} \!\cap\! B_{n, \tau,\delta}). %\\
%\geq& \sup_{\mu \in \mP}\frac{\mu^n}{\pip^n}(\{\phi_n=0\} \cap B_{n, \tau,\delta})\bigl(\Prob_\pip(B_{n, \tau,\delta})-\Prob_\pip(\{\phi_n=1\})\bigr) .
\end{aligned}
\end{equation}
The first term on the right-hand side is lower-bounded in \eqref{e:boundLLRonA}. The second term can be shown to have the same large deviations limit  as that of $\Prob_\pip(B_{n ,\tau,\delta})$:
\begin{equation}\label{e:boundBPF}
\Prob_\pip(\{\phi_n=0\} \cap B_{n, \tau,\delta}) \geq \Prob_\pip(B_{n, \tau,\delta})-\Prob_\pip(\{\phi_n=1\}).
\end{equation}
The inequality  in \eqref{e:boundPA} ensures that $\Prob_\pip(\{\phi_n=1\})$ is negligible comparing to $\Prob_\pip(B_{n, \tau,\delta})$.  

%The technique of using uniform lower-bounds on likelihood ratio (LR) to prove lower-bounds of probability of missed detection has been applied in \citep{pan08p4750,bar89p107}: In this prior work,  a uniform bound on LR is obtained \emph{over all possible} $\bfmz_1^n$. To prove the tight hardness result as in \Theorem{converse}, we require the bound on LR to hold uniformly for the sequences in the event $B_n$, instead of all sequences. This gives us the freedom to optimize $B_n$ to obtain the tightest bound. 
%
%%Thus the value of the new bound on LR is larger than the prevous one, leading to a tighter bound for the probability of misse detection. 
%The technique to prove \eqref{e:boundLLRonA} has been previously used in providing hardness results for composite and hypothesis testing problems \citep{pan08p4750,bar89p107, kelwagtulvis13p782}. First, construct a collection of distributions so that for each distribution $\mu$, the likelihood ratio $\mu/\pip$ has a simple expression. Second, show that for all observations $\bfmz_1^n \eqdef\{z_1, \ldots, z_n\}$ in the event $B_n$, the \emph{average} of $\Prob_{\mu}\{\bfmZ_1^n=\bfmz_1^n\}/\Prob_{\pip}\{\bfmZ_1^n=\bfmz_1^n\}$ over the collection of distributions can be lower-bounded, which in turn lower-bounds the left-hand side of \eqref{e:boundLLRonA}. 

We now give the details of constructing the event $B_{n, \tau, \delta}$ and lower-bounding the likelihood ratio. The proof for $\epsy < 0.5$ and $\epsy \geq 0.5$  uses different constructions of distributions. 

\subsection{Construction of $B_{n, \tau, \delta}$}
Define the event 
\begin{equation}\label{e:defBgamma}
\begin{aligned}
B_{n, \tau,\delta}=\bigl\{&\sum_{j=1}^m \ind\{n\Gamma^n_j \!\!=\!\!1\} \!\geq  \!n-(1\!+\tau\!+\delta)\frac{n^{2\!}}{m}, \\ &\quad \sum_{j=1}^m \ind\{n\Gamma^n_j\!\! =\! \!2\} \!\geq\! \half (1\!+\tau\!-\delta) \frac{n^{2\!}}{m}\bigr\}.
\end{aligned}
\end{equation}
The probability of the event $B_{n,\tau,\delta}$ has the following asymptotic approximation:
%The following lemma shows that the event $B_{n, \tau,\delta}$ is typical under the null hypothesis when $\tau<0$, and characterizes the asymptotic behavior of the event $B_{n,\tau,\delta}$ when $\tau>0$.
%\begin{equation}\label{e:defBgamma}
%B_{n, \gamma1,\gamma_2}=\bigl\{ n-(1+\gamma_1+|\gamma_2|)\frac{n^2}{m} \leq     \sum_{j=1}^m \ind\{n\Gamma^n_j =1\} \leq n-(1+\gamma_1)\frac{n^2}{m},     \sum_{j=1}^m \ind\{n\Gamma^n_j = 2\} \geq \half (1+\gamma_2) \frac{n^2}{m}\bigr\}.
%\end{equation}
\begin{lemma}\label{T:BNGAMMA0}
For $\tau=0$ and any $\delta>0$, 
\begin{equation}\label{e:limitBn}
\lim_{n \rightarrow \infty} \Prob_{\pip}(B_{n, \tau,\delta})=1.
\end{equation}
For any $\tau, \delta$ satisfying $\tau>\delta>0$,
\begin{equation}\label{e:Bnprob}
\lim_{n \rightarrow \infty}-\frac{m}{n^2}\log\Prob_\pip(B_{n, \tau,\delta}) = J_F^*(\tau-\delta).
\end{equation}
\end{lemma}
\begin{proof}[Proof of Lemma~\ref{T:BNGAMMA0}]
First consider the case where $\tau=0$. 
Applying \Theorem{KPERFORMANCE} with $\tau$ replaced by $\delta$ gives
\begin{equation}\label{e:defBconst1}
\Prob_\pip\big\{\sum_{j=1}^m \ind\{n\Gamma^n_j =1\} \leq n-(1+\delta)\frac{n^2}{m}\big \}=1-o(1).
\end{equation}
The following asymptotic approximations on the expectation and variance of the statistic $\sum_{j=1}^m \ind\{n\Gamma^n_j = 2\}$ follows from \Lemma{separableexpectation} and \Lemma{separablevariance}: 
\begin{equation}
\begin{aligned}
\Expect_{\pip}[\sum_{j=1}^m\ind\{n\Gamma^n_j = 2\}]\! &=\!\half \frac{n^2}{m}(1+o(1)), \\
\Var_{\pip}[\sum_{j=1}^m\ind\{n\Gamma^n_j = 2\}]\! &=\!\half \frac{n^2}{m}(1+o(1)).\!
\end{aligned}\nonumber
\end{equation}
Applying Chebyshev's inequality leads to
\[\Prob_{\pip}\bigl\{\sum_{j=1}^m\ind\{n\Gamma^n_j = 2\} \leq \half \frac{n^2}{m}(1-\delta)\bigr\} =O(\frac{m}{n^2}).\]
The claim of this lemma for $\tau=0$ follows from combining this inequality with \eqref{e:defBconst1}.

Next consider the case where $\tau>0$. We first obtain a large deviations characterization of 
\[S^{(2)}\eqdef \sum_{j=1}^m \ind\{n\Gamma^n_j = 2\}\] by deriving an approximation to the logarithmic moment generating function.  The steps are the same as those in the proof of  Proposition~\ref{T:ASYMLMGFK}. Again, we describe the main steps and highlight the difference. First, the estimate of the saddle point is different than that in \eqref{e:lambdawexpression} and \eqref{e:wvalue}. We have
\begin{equation}%\label{e:lambdaexpT}
\begin{aligned}
\lambdab&=n(1+w), \\w&=-n\sum_j \nu_j^2 (e^\theta -1) (1+O(\frac{n}{m})).
\end{aligned}\nonumber
\end{equation}
Second, different from \eqref{e:formulamgf3}, we have the following expression of the moment generating function: 
\[\Expect_\nu^n[\exp\{\theta S^{(2)}\}]=\frac{n!}{2\pori}\lambdab^{-n}\Rec[\int_{-\pi}^\pi h(\psi)d\psi]\]
where
\begin{equation}
\begin{aligned}
h(\psi)
&=\! e^{\!-i n \psi}\!\prod_{j=1}^m \![\exp\{\lambdab \nu_je^{i\psi}\}\! +\!\half \lambdab^2 e^{2i\psi\!} \nu_j^2 (e^{\theta}\!\!-\!\!1) ]\\&=e^{-i n \psi}\exp\{n e^{i\psi}+O(\frac{n^2}{m})\}.\end{aligned}\nonumber%\label{e:hestimateT}
\end{equation}
Finally, the approximation of $\Rec(H(0))$ is different from that in \eqref{e:estimateJ}
\[\begin{aligned}
\Rec(H(0))=&n (1+w)+ \half n^2(\sum_{j=1}^m\nu_j^2) (e^\theta-1)+O(\frac{n^3}{m^2}).\end{aligned}\]
The rest of the steps are the same as those in Proposition~\ref{T:ASYMLMGFK}. We obtain
\begin{equation}
\Lambda_{\nu, S^{(2)}}(\theta)\!=\!\half \frac{n^2}{m}\bigl(m\sum_{j=1}^m\!\nu_j^2\bigr)(e^{-2\theta}\!-\!1)\!+\!O(\frac{n^3}{m^2})\!+\!O(1).\\
\end{equation}
Applying the same steps as those for the characterization of $J_F(\ktest)$ in \Theorem{KPERFORMANCE}, we have
\begin{equation}
\begin{aligned}
&\lim_{n \rightarrow \infty}-\frac{m}{n^2}\log\Prob_\pip\big\{\sum_{j=1}^m \ind\{n\Gamma^n_j = 2\} \geq \half (1+\tau-\delta) \frac{n^2}{m}\big\}\\=& J_F^*(\tau-\delta).\nonumber\end{aligned}
\end{equation}

Applying \Theorem{KPERFORMANCE} with $\tau$ replaced by $\tau+\delta$, we obtain  
\[
\begin{aligned}
&\lim_{n \rightarrow \infty}-\frac{m}{n^2}\log\Prob_\pip\big\{\sum_{j=1}^m \ind\{n\Gamma^n_j =1\} \!\leq \!n-(1+\tau+\delta)\frac{n^2}{m}\big \}\\=& J_F^*(\tau+\delta).
\end{aligned}
\]

Note that $J_F^*(\tau+\delta) > J_F^*(\tau-\delta)$. Thus the probability that the first constraint in the definition of $B_{n,\tau,\delta}$ is violated is negligible comparing to the probability that the second constraint is satisfied. This shows that the probability of $B_{n,\tau,\delta}$ can be approximated by the probability that the second constraint in the definition of $B_{n,\tau,\delta}$ is satisfied. This leads to the claim of the lemma. 

\end{proof}
\subsection{A lower-bound on the likelihood ratio for $\epsy\geq 0.5$}
When $\epsy \geq 0.5$, we use the following construction of distributions:
Let $\Kset$ denote the collection of all subsets of $[m]$ whose cardinality is $\lfloor m(1-\epsy) \rfloor$. For each  $\mathcal{U} \in \Kset$, define the distribution
\begin{equation}\label{e:worstb}
\mu_{\mathcal{U},j}=\left\{\begin{array} {l l} \frac{1}{\lfloor m(1-\epsy)\rfloor}, & j \in \mathcal{U};\\ 0, & j \in [m]\setminus \mathcal{U}.\end{array}\right.\nonumber
\end{equation}
Consider the mixture $\bmu^n\!=\!\frac{1}{|\Kset|}\sum_{\mathcal{U} \in \Kset}\mu_{\mathcal{U}}^n$.
The following bound on $\bmu^n/\pip^n$ holds:
\begin{lemma}\label{T:LEMMABMUBOUND2}
Suppose $\epsy\geq 0.5$. For any sequence $\bfmz_1^n= \{z_1, \ldots, z_n\}$ satisfying $\{\bfmZ_1^n=\bfmz_1^n\} \subseteq B_{n, \tau,\delta}$, the following holds:
\[\log\bigl(\frac{\bmu^n}{\pip^n}(\bfmz_1^n)\bigr)\!\geq\!\! -\half \!\frac{n^2}{m}[\taul(\epsy)-\log(1+\taul(\epsy))(1+\tau-\delta)]+O(\frac{n^3}{m^2}).\]
\end{lemma}
\begin{proof}[Proof of Lemma~\ref{T:LEMMABMUBOUND2}]
Let $\mathcal{S}\eqdef\{j: \textrm{ $j$ appears in $\bfmz_1^n$}\}$. Let $s=|\mathcal{S}|$. It follows from $\{\bfmZ_1^n=\bfmz_1^n\} \subseteq B_{n, 
\tau,\delta}$ that
\begin{equation}\label{e:sbound}
 n-\half \frac{n^2}{m} (1+\tau+3\delta)  \leq s \leq n-\half \frac{n^2}{m} (1+\tau-\delta).
\end{equation} The likelihood ratio $\frac{\mu_{\mathcal{U}}^n}{\pip^n}$ has the expression:  \[\frac{\mu_{\mathcal{U}}^n}{\pip^n}(\bfmz_1^n)=(\frac{m}{\lfloor m(1-\epsy)\rfloor})^n \ind_{\mathcal{S} \subseteq \mathcal{U}}.\]
%\begin{equation}
%\frac{\mu_{\mathcal{U}}^n}{\pip^n}(\bfmz_1^n)=(\frac{m}{\lfloor m(1-\epsy)\rfloor})^n \ind_{\mathcal{S} \subseteq \mathcal{U}}.\nonumber
%\end{equation}
Thus, 
\begin{equation}\label{e:muU}
\frac{\bmu^n}{\pip^n}(\bfmz_1^n)=(\frac{m}{\lfloor m(1-\epsy)\rfloor})^n(\frac{1}{|\Kset|}\sum_{\mathcal{U} \in \Kset}\ind_{\mathcal{S} \subseteq \mathcal{U}} ),
\end{equation}
where
\begin{equation}%\label{e:bmuz}
\frac{1}{|\Kset|}\sum_{\mathcal{U} \in \Kset}\ind_{\mathcal{S} \subseteq \mathcal{U}} =\frac{{m-s \choose \lfloor m(1-\epsy) \rfloor-s}}{{m \choose \lfloor m(1-\epsy) \rfloor}}.\nonumber
\end{equation}
Stirling's formula gives 
\begin{equation}
\begin{aligned}
&\frac{{m-s \choose \lfloor m(1-\epsy) \rfloor-s}}{{m \choose \lfloor m(1-\epsy) \rfloor}}\\=&(\frac{\lfloor m(1-\epsy)\rfloor}{m})^s\exp\{-\half \frac{s^2}{m}\frac{\epsy}{1-\epsy}+O(\frac{k^3}{m^2})\}(1+O(\frac{1}{m})).\nonumber
\end{aligned}
\end{equation}
Substituting this into \eqref{e:muU} leads to 
\[\frac{\bmu^n}{\pip^n}(\bfmz_1^n)= (1-\epsy)^s\exp\{-\half \frac{s^2}{m}\frac{\epsy}{1-\epsy}+O(\frac{n^3}{m^2})\}(1+O(\frac{n}{m})).\]
The claim of this lemma follows from applying the inequality \eqref{e:sbound} and the fact that $\taul(\epsy)=\frac{\epsy}{1-\epsy}$ when $\epsy \geq 0.5$. 
%\[\frac{\bmu^n}{\pip}(\bfmz_1^n)\geq \exp\{-\half \frac{n^2}{m}\frac{\epsy}{1-\epsy}-\half \frac{n^2}{m} \log(1-\epsy)(1+\tau-\delta)+O(\frac{n^3}{m^2})\}.\]
%Substituting this into \eqref{e:muU}, and then using the fact that $n \geq s \geq n-\half (1+\gamma) n^2/m$, we conclude
%\[\frac{\bmu^n}{\pip}(\bfmz_1^n)\geq \exp\{-\half \frac{n^2}{m}\frac{\epsy}{1-\epsy}-\half \frac{n^2}{m} \log(1-\epsy)(1+3\gamma)+O(\frac{n^3}{m^2})\}(1+O(\frac{1}{m})).\]
\end{proof}

\subsection{A lower-bound on the likelihood ratio for $\epsy< 0.5$}
When $\epsy<0.5$, we use the following construction of distributions:
Let $\Kset$ denote the collection of all subsets of $[m]$ whose cardinality is $\lfloor m/2 \rfloor$. For each set $\mathcal{U} \in \Kset$, define the distribution $\mu_{\mathcal{U}}$ as
\begin{equation}
\mu_{\mathcal{U},j}=\left\{\begin{array}{l l}\frac{1}{m}+\frac{\epsy}{\lfloor m/2 \rfloor}, & j \in \mathcal{U};\\ \frac{1}{m}-\frac{\epsy}{\lceil m/2 \rceil}, & j \in [m]\setminus \mathcal{U}.
\end{array}\right.\nonumber \end{equation}
This collection of distributions can be obtained by taking the worst-case distribution $\mu^*$ given in \eqref{e:worstmu2a}, and permuting the symbols in the alphabet $[m]$. 

Let $\mu^n_{\mathcal{U}}$ be the $n$-order product of $\mu_{\mathcal{U}}$. Define the following mixture distribution,
%\begin{equation}
\[
%\label{e:bmuformula}
\bmu^n=\frac{1}{|\Kset|}\sum_{\mathcal{U} \in \Kset}\mu^n_\mathcal{U}.
\]
%\end{equation}
The LR $\bmu^n/\pip^n$ can be lower-bounded on $B_{n, \tau,\delta}$ :
%, and an asymptotic approximation of the exact bound is given below:
\begin{lemma}\label{T:LEMMABMUBOUND1}
Suppose $\epsy<0.5$. The following holds for any sequence $\bfmz_1^n$ satisfying $\{\bfmZ_1^n=\bfmz_1^n\} \subseteq B_{n, \tau,\delta}$:
\[
\begin{aligned}
\log\bigl(\frac{\bmu^n}{\pip^n}(\!\bfmz_1^n\!)\bigr)\!\geq&\!-\!\frac{n^2}{2m}[\taul(\epsy)\!-\!\log(1\!+\!\taul(\epsy))(1\!+\!\tau\!-\!\delta)](1\!+\!o(1))\\&-\! \frac{n^2}{m}2\delta\log(1-2\epsy).
\end{aligned}
\]
\end{lemma}
\begin{proof}[Proof of Lemma~\ref{T:LEMMABMUBOUND1}]
For simplicity of exposition we restrict to the case where $m$ is even. 
Define
\[\begin{aligned}
\mathcal{S}_1\!&\eqdef\!\{j\!: \!\textrm{ $j$ appears in $\bfmz_1^n$ exactly \emph{once}}\},\\  \mathcal{S}_2\!&\eqdef\!\{j\!: \!\textrm{ $j$ appears in $\bfmz_1^n$ exactly \emph{twice}}\}.
\end{aligned}
\]
Denote their cardinality by $s_1=|\mathcal{S}_1|$, $s_2=|\mathcal{S}_2|$. It follows from $\{\bfmZ_1^n=\bfmz_1^n\} \subseteq {B}_{n, \tau, \delta}$ that
\begin{equation}\label{e:s1s2bound}
n \geq s_1 \geq n-\frac{n^2}{m}(1+\tau+\delta), \quad s_2 \geq \half\frac{n^2}{m}(1+\tau-\delta).
\end{equation}
%\begin{equation}\label{e:s1s2bound}
%n-\frac{n^2}{m}(1+\tau+\delta) \leq s_1 \leq n, \quad \frac{n^2}{m}\half(1+\tau-\delta) \leq s_2 \leq \half (n-s_1) \leq \frac{n^2}{m}\half (1+\tau+\delta).
%\end{equation} 
Consider any set $\mathcal{U} \in \Kset$. Let $k_{\mathcal{U},1}=|\mathcal{U} \cap \mathcal{S}_1|$, and $k_{\mathcal{U},2}=|\mathcal{U} \cap \mathcal{S}_2|$. Then
\begin{equation}
\frac{\mu^n_{\mathcal{U}}}{\pip^n}(\bfmz_1^n) \geq (1-2\epsy)^n (\frac{1+2\epsy}{1-2\epsy})^{k_{\mathcal{U},1}+2k_{\mathcal{U},2}}\nonumber.% (\frac{1+2\epsy}{1-2\epsy})^{2k_{\mathcal{U},2}} \nonumber\\
\end{equation}
Consequently, 
\begin{equation}\label{e:averagebmu1}
\frac{\bmu^n}{\pip^n}(\bfmz_1^n)\geq G(s_1,s_2)
\end{equation}
where 
\begin{equation}
\begin{aligned}
&G(s_1,s_2)\\\eqdef& \frac{1}{|\Kset|}(1-2\epsy)^n \\&\sum_{k_1=1}^{s_1}\sum_{k_2=1}^{s_2} \bigl( (\frac{1+2\epsy}{1-2\epsy})^{k}(\frac{1+2\epsy}{1-2\epsy})^{2k_{2}}\\&\qquad\quad\quad|\{\mathcal{U} \in \Kset:k_{\mathcal{U},1}=k_1,k_{\mathcal{U},2}=k_2\}|\bigr)\\
=& \frac{1}{{m \choose  m/2}}(1-2\epsy)^n \\& \sum_{k_1=1}^{s_1}\sum_{k_2=1}^{s_2}\bigl((\frac{1+2\epsy}{1-2\epsy})^{k}(\frac{1+2\epsy}{1-2\epsy})^{2k_{2}} \\&\qquad{s_1 \choose k_1}{s_2 \choose k_2}{m-(s_1+s_2) \choose m/2-(k_1+k_2)}\bigr).\end{aligned}\label{e:approxg}
\end{equation}

The summand on the right-hand side of \eqref{e:approxg} takes its maximum value approximately when \begin{equation}k_1=\bar{k}_1 \eqdef \lceil \frac{1+2\epsy}{2}s_1\rceil,  k_2=\bar{k}_2 \eqdef \lceil  \half (1+\frac{4\epsy}{1+4\epsy^2})\rceil.\label{e:defk1k2}
\end{equation}
We apply the Laplace method to approximate the summation: Denote
\[
\begin{aligned}
y(\Delta_1, \Delta_2)=&(\frac{1+2\epsy}{1-2\epsy})^{\bar{k}_1+\Delta_1+2(\bar{k}_2\!+\!\Delta_{2})}{s_1 \choose \bar{k}_1+\Delta_1}{s_2 \choose \bar{k}_2\!+\!\Delta_2}\\&\times{m-(s_1+s_2) \choose m/2-(\bar{k}_1+\Delta_1+\bar{k}_2+\Delta_2)}/{m \choose m/2}.\end{aligned}\]
Stirling's formula gives
\begin{equation}
\begin{aligned}
&{m-(s_1+s_2) \choose \frac{m}{2}-(\bar{k}_1+\Delta_1+\bar{k}_{2}+\Delta_{2})}/{m-(s_1+s_2) \choose \frac{m}{2}-(\bar{k}_1+\bar{k}_2)}\\=&\exp\{1+O(\frac{(\Delta_1+\Delta_2)(\bar{k}_1+\bar{k}_2)}{m})+o(1)\}.
\end{aligned}\label{e:stirlinglaplaceapprox}
\end{equation}
Let \[
\begin{aligned}
y_1(\Delta_1)&=(\frac{1+2\epsy}{1-2\epsy})^{\Delta_1}{s_1 \choose \bar{k}_1+\Delta_1}/{s_1 \choose \bar{k}_1},\\ y_2(\Delta_2)&=(\frac{1+2\epsy}{1-2\epsy})^{2\Delta_{2}}{s_2 \choose \bar{k}_2+\Delta_2}/{s_2 \choose \bar{k}_2}.
\end{aligned}
\] 
Note that $y(\bar{k}_1,\bar{k}_2)$ is the largest summand. Keeping only the $\lceil\sqrt{s_1}\rceil \lceil\sqrt{s_2}\rceil$ number of terms in the summation in \eqref{e:averagebmu1} whose index $(k_1,k_2)$ is close to  $(\bar{k}_1, \bar{k}_2)$, and applying \eqref{e:stirlinglaplaceapprox}, we obtain
\begin{equation}
\begin{aligned}
\frac{\bmu^n}{\pip^n}(\bfmz_1^n)
&\!\geq\! \sum_{\Delta_1=-\lceil \sqrt{s_1} \rceil}^{\lceil \sqrt{s_1} \rceil}\sum_{\Delta_2=-\lceil \sqrt{s_2} \rceil}^{\lceil \sqrt{s_2} \rceil} y(\Delta_1, \Delta_2)\\
&\!=\!\bigl(\!\!\!\sum_{\Delta_1=-\lceil \sqrt{s_1} \rceil}^{\lceil \sqrt{s_1} \rceil}\!\!\!y_1(\Delta_1)\bigr)\bigl(\!\!\sum_{\Delta_2=-\lceil \sqrt{s_2} \rceil}^{\lceil \sqrt{s_2} \rceil}\!\!\!y_2(\Delta_2)\bigr) y(0,0)\\&\quad \times\exp\{1\!+\!O(\frac{n^{\frac{3}{2}}}{m})\}.\end{aligned}\label{e:Deltaapproximation}
\end{equation}
We first approximate $\sum_{\Delta_1=-\lceil \sqrt{s_1} \rceil}^{\lceil \sqrt{s_1} \rceil}y_1(\Delta_1)$. Note that for $\Delta_1>0$, 
\[\log(y_1(\Delta_1))=\Delta_1 \log(\frac{1+2\epsy}{1-2\epsy})+\sum_{t=1}^{\Delta_1} \log(\frac{s-\bar{k}_1-t}{\bar{k}_1+t}).\]
Approximating the above summation by integrals leads to
\[\log(y_1(\Delta_1))=- \half (\frac{1}{s_1-\bar{k}_1}+\frac{1}{\bar{k}_1})\Delta_1^2(1+o(1))+O(1).\]
Approximating the summation over $\Delta_1$ using integrals, and applying the above approximation of $y_1(\Delta_1)$ leads to 
\[\begin{aligned}
\sum_{\Delta_1=-\lceil \!\sqrt{s_1} \rceil}^{\lceil \sqrt{s_1} \rceil}\!\!\!\!\!\!y_1(\Delta_1)&=\!e^{O(1)}\!\!\!\int_{-\infty}^{\infty}\!\!e^{\!- \!\half (\frac{1}{s_1-\bar{k}_1}+\frac{1}{\bar{k}_1})\Delta_1^2}d \Delta_1\\&=\!e^{O(1)}\!\sqrt{\!\frac{(s_1\!-\!\bar{k}_1)\bar{k}_1}{s_1}}\!=\!e^{O(1)\!}\!\sqrt{s_1},\end{aligned}\]
where the last equality follows from \eqref{e:defk1k2}.
A similar approximation for the summation over $y_2$ holds:  \[\sum_{\Delta_2=-\lceil \sqrt{s_2} \rceil}^{\lceil \sqrt{s_2} \rceil}y_2(\Delta_2)=e^{O(1)}\sqrt{s_2}.\] 
%A similar approximation to the summation involving $y_2$ in \eqref{e:Deltaapproximation} holds: \[\bigl(\sum_{\Delta_2=-\lceil \sqrt{s_2} \rceil}^{\lceil \sqrt{s_2} \rceil}y_2(\Delta_2)\bigr)=e^{O(1)}\sqrt{s_2}.\]
Substituting these into \eqref{e:approxg} gives
\begin{equation}
\begin{aligned}
&G(s_1,s_2)\\=&e^{O(1)\!+\!O(\frac{n^{3/2}}{m})}\sqrt{s_1s_2} (1\!-\!2\epsy)^n  (\frac{1+2\epsy}{1-2\epsy})^{\bar{k}_1}(\frac{1+2\epsy}{1-2\epsy})^{2\bar{k}_{2}} \\&\times{s_1 \choose \bar{k}_1}{s_2 \choose \bar{k}_2}{m-(s_1+s_2) \choose m/2-(\bar{k}_1+\bar{k}_2)}/{m \choose m/2}.\label{e:approxg2}\end{aligned}
\end{equation}
Stirling's formula gives the following asymptotic approximations the  combinatorial terms in  \eqref{e:approxg2}:
\begin{equation}
\begin{aligned}
{{s_1} \choose \bar{k}_1}&\!=\!\frac{(1\!+\!2\epsy)^{-\bar{k}_1} (1\!-\!2\epsy)^{\bar{k}_1-{s_1}}2^{s_1}}{\sqrt{2\pi \bar{k}_1({s_1}-\bar{k}_1)/{s_1}}}(1\!+\!o(1)),\\
{{s_2} \choose \bar{k}_2}&\!=\!\frac{(1+2\epsy)^{-2\bar{k}_2} (1-2\epsy)^{2(\bar{k}_2-{s_2})}}{\sqrt{2\pi \bar{k}_2({s_2}-\bar{k}_2)/{s_2}}}\\&\quad\times (1\!+\!4\epsy^2)^{s_2}2^{s_2}(1+o(1)),\\
{m\!-\!(s_1\!+\!s_2) \choose m\!/\!2\!-\!(\bar{k}_1\!+\!\bar{k}_2)}
&\!=\!2^{m-s_1-s_2}\exp\{-\frac{s_1^2(2\epsy)^2}{2m}(1+o(1))\}\\&\quad\times \frac{\sqrt{2}}{\sqrt{\pi m}}(1+o(1)),\\
{m \choose m/2}&\!=\!\frac{2^m}{\sqrt{2\pi m}}(1+o(1)).\nonumber
\end{aligned}
\end{equation}
Substituting these approximations and the value of $\bar{k}_1$ and $\bar{k}_2$ into \eqref{e:approxg2} leads to 
\begin{equation}%\label{e:averagebmu3}
\begin{aligned}
G(s_1,s_2)=& (1\!-\!2\epsy)^{n\!-\!s_1\!-\!2s_2} \\&\times \exp\{\!-\frac{s_1^2(\!2\epsy\!)^2\!}{2m}(1\!+\!o(1))\!+\!s_2\!\log(\!1+4\epsy^2\!)\}\\&\times\!\exp\{O(1)\!+\!O(\!\frac{n^{3/2\!}}{m}\!)\}.\nonumber
\end{aligned}
\end{equation}
Combining this with \eqref{e:s1s2bound}, \eqref{e:averagebmu1} gives the claim of the lemma.
%It then follows from \eqref{e:s1s2bound} that
%\begin{equation}%\label{e:averagebmu4}
%\frac{\bmu^n}{\pip^n}(\bfmz_1^n)\geq \exp\{-\half \frac{n^2}{m}\bigl(4\epsy^2-\log(1+4\epsy^2)(1+\tau-\delta)\big)(1+o(1))+\frac{n^2}{m}2\delta\log(1-2\epsy)\}e^{O(1)}.\nonumber
%\end{equation}
\end{proof}
\subsection{Proof of \Theorem{converse}}
\begin{proof}
Consider first the case $\tau>0$. Consider any $\delta \in (0, \tau)$, and any test $\phi$ such that $J_F(\phi) \geq J_F^*(\tau)$. Applying \eqref{e:boundBPF} and Lemma~\ref{T:BNGAMMA0}, we obtain
\begin{equation}\label{e:BnprobPF}
\lim_{n \rightarrow \infty}-\frac{m}{n^2}\log\Prob_\pip(\{\phi_n=0\} \cap B_{n, \tau,\delta}) = J_F^*(\tau-\delta).
\end{equation}
When $\epsy\geq 0.5$, we apply \eqref{e:PMboundPForigin}, \eqref{e:BnprobPF}, and Lemma~\ref{T:LEMMABMUBOUND2} to obtain 
\begin{equation}\label{e:JMbounddelta2}
\begin{aligned}
&J_M(\phi) \\\leq& \half [\taul(\epsy)-\log(1+\taul(\epsy))(1+\tau-\delta)]+J_F^*(\tau-\delta)\\
=&J_M^*(\tau-\delta)+r_2(\delta).
\end{aligned}
\end{equation}
where $r_2$ again vanishes as $\delta \rightarrow 0$, 
\[\begin{aligned}
r_2(\delta)=&\half[-\delta \log(1+\kappa(\epsy))\\
&\quad\!+\!(1+\tau)\log(1-\frac{\delta}{1+\tau})-\delta \log(1+\tau-\delta)+\delta].
\end{aligned}\]
We have used the following explicit expressions of $J_F^*$ and $J_M^*$:
\[\begin{aligned}
J_F^*(\tau)&=\half[-\tau+(1+\tau)\log(1+\tau)],\!\!\! \\
J_M^*(\tau)&=\half[\taul(\epsy)-\tau+(1+\tau)\log(\frac{1+\tau}{1+\taul(\epsy)})].
\end{aligned}
\]
Since  \eqref{e:JMbounddelta2} holds for any $\delta\!>\!0$ and $J_M^*(\tau)$ is continuous, we conclude $J_M(\phi) \!\leq\! J_M^*(\tau)$.

When $\epsy< 0.5$, we apply \eqref{e:PMboundPForigin}, \eqref{e:BnprobPF}, and Lemma~\ref{T:LEMMABMUBOUND1} to obtain 
\begin{equation}\label{e:JMbounddelta1}
\begin{aligned}
&J_M(\phi) \\
\leq&  \half[\taul(\epsy)\!-\!\log(1\!+\!\taul(\epsy))(1\!+\!\tau\!-\!\delta)+4\delta\log(1-2\epsy)\\&+J_F^*(\tau-\delta)\\
%&=\half [\taul(\epsy)-\log(1+\taul(\epsy))(1+\tau-\delta)+4\delta\log(1-2\epsy)]\\
%& \quad +\half[-\tau+\delta+(1+\tau-\delta)\log(1+\tau-\delta)]\\
=&  J_M^*(\tau-\delta)+r_1(\delta).
\end{aligned}\vspace{-0.1cm}
\end{equation}
where 
\[\begin{aligned}
r_1(\delta)\!=&\!\half[-\delta \!\log(1+\kappa(\epsy)\!)+(\!1+\tau\!)\!\log(1-\frac{\delta}{\!1\!+\!\tau\!})\\
&\quad-\delta\! \log(1+\tau-\delta)+\delta+4\delta\!\log(1-2\epsy)].\vspace{-0.1cm}\end{aligned}\]
Since the inequality \eqref{e:JMbounddelta1} holds for any $\delta>0$, $J_M^*(\tau)$ is continuous in $\tau$, and $r_1(\delta) \rightarrow 0$ as $\delta \rightarrow 0$, we conclude that $J_M(\phi) \leq J_M^*(\tau)$.

The proof for the case where $\tau=0$ is exactly the same as that for the case $\tau>0$, except \eqref{e:limitBn} is used in place of \eqref{e:Bnprob}. We omit the details.

\end{proof}

\section{Proof of Lemma~\ref{T:PEARSONTAUBD}, Lemma~\ref{T:PEARSONA} and Lemma~\ref{T:PEARSONA2} Used in the Proof of  \Theorem{pearsonerrorexponent}}\label{apx:lemmapearson}
\begin{proof}[Proof of Lemma~\ref{T:PEARSONTAUBD}]
%Assume without loss of generality that $m$ is even. 
Applying \Lemma{separableexpectation}  to the distribution $\mu^* \in \mP$ given in \eqref{e:worstmu1} and \eqref{e:worstmu2a} gives $
\Expect_{\mu^*}[\pearson]=\Expect_{\pip}[\pearson]+\frac{n^2}{m}\kappa(\epsy)(1+o(1))$.
It follows from Chebyshev's inequality that for $\tau_n > \Expect_{\pip}[\pearson]+\frac{n^2}{m}\taul(\epsy)$,
\[
\Prob_{\mu^*}\{\ptest_n(\bfmZ_1^n)=1\}\leq \frac{\Var_{\mu^*}[\pearson]}{(\tau_n-\Expect_{\pip}[\pearson]-\frac{n^2}{m}\taul(\epsy))^2}.
\]
Thus, in order for $\lim_{n \rightarrow \infty}\Prob_{\mu^*}\{\ptest_n(\bfmZ_1^n)=1\}=1$ to hold, we must have 
\[\begin{aligned}
(\tau_n-\Expect_{\pip}[\pearson]-\frac{n^2}{m}\taul(\epsy))^2 &\leq \Var_{\mu^*}[\pearson](1+o(1))\\&=2\frac{n^2}{m}(1+\taul(\epsy))(1+o(1)).
\end{aligned}\]
where the last equality follows from \Lemma{separablevariance}.
This leads to the claim of Lemma~\ref{T:PEARSONTAUBD}.
%\begin{equation}\label{e:taubound}
%\tau_n\leq \Expect_{\pip}[\pearson]+\frac{n^2}{m}\taul(\epsy)+\sqrt{2}\frac{n}{\sqrt{m}}(1+o(1))
%\end{equation}
%This follows from Chebyshev's inequality: For $\tau_n \geq \Expect_{\pip}+\frac{n^2}{m}\taul(\epsy)$, we have
%\[(\tau_n-\Expect_{\pip}[\pearson]-\frac{n^2}{m}\taul(\epsy))^2\leq \frac{\Var_{\mu}[\pearson]}{\Prob_\mu\{\ptest(\bfmZ_1^n)=1\}}=2\frac{n^2}{m}(1+\taul(\epsy))(1+o(1)).\]
\end{proof}

\begin{proof}[Proof of Lemma~\ref{T:PEARSONA}]
Consider the statistic
\[\bar{\pearson}=\pearson-\frac{n}{m}\frac{(n\Gamma^n_1-n\pip_1)^2}{n\pip_1}=\pearson-2\frac{n^2}{m}\taul(\epsy)+O(\frac{n}{\sqrt{m}}).\]
The conditional distribution of $\bar{\pearson}$ in the event $A$ under $\pip$ is the same as the distribution of $S^{\sf P}_{n'}$ under $\pip'$, where the number of samples is $n'=n-\lfloor \frac{n\sqrt{2\taul(\epsy)}}{\sqrt{m}}\rfloor$ and $\pip'$ is the uniform distribution over $[m-1]$. 
It then follows from Lemma~\ref{t:separableexpectation} and \Lemma{separablevariance}  that
\[
\begin{aligned}
\Expect_{\pip}[\bar{\pearson}|A]&=\Expect_{\pip'}[S^{\sf P}_{n'}]=n-\lfloor  \frac{n\sqrt{2\taul(\epsy)}}{\sqrt{m}}\rfloor+O(\frac{n^2}{m}),\\
\Var_{\pip}[\bar{\pearson}|A]&=\Var_{\pip'}[S^{\sf P}_{n'}]=2\frac{n^2}{m}(1+o(1)).\end{aligned}\]
It then follows from Chebyshev's inequality, \Lemma{separableexpectation} and \Lemma{separablevariance} that for large enough $n$,
\begin{equation}
\begin{aligned}
%&\Prob_\pip \{\pearson \leq \Expect_\pip[\pearson]+\bar{\tau}_n\}\\
&\Prob_\pip \{\pearson \leq \Expect_{\pip}[\pearson]+\frac{n^2}{m}\taul(\epsy)+2\frac{n}{\sqrt{m}}|A_n\}\\
&=\Prob_\pip \{\bar{\pearson} \!+\! 2\frac{n^2}{m}\taul(\epsy) \leq n\!+\!\frac{n^2}{m}\taul(\epsy)\!+\!2\frac{n}{\sqrt{m}}\!+\!O(\frac{n}{\sqrt{m}})|A_n\}\\
&=\Prob_\pip \{\bar{\pearson} \leq \Expect_{\pip}[\bar{\pearson}|A]-\frac{n^2}{m}\taul(\epsy)+O(\frac{n}{\sqrt{m}})|A_n\}\\
&\leq \frac{2\frac{n^2}{m}(1+O(\frac{n}{\sqrt{m}}))}{\bigl(\frac{n^2}{m}\taul(\epsy)+O(\frac{n}{\sqrt{m}})\bigr)^2}=O(\frac{m}{n^2}).
\end{aligned}\nonumber
\end{equation}\vspace{-0.1cm}
\end{proof}
\begin{proof}[Proof of Lemma~\ref{T:PEARSONA2}]
A simple combinatorial argument gives 
\begin{equation}
\Prob_\pip\{A_n\}={n \choose \lfloor {n\frac{\sqrt{2\taul(\epsy)}}{\sqrt{m}}}\rfloor} \pip_1^{\lfloor {n\frac{\sqrt{2\taul(\epsy)}}{\sqrt{m}}}\rfloor}(1-\pip_1)^{n-\lfloor {n\frac{\sqrt{2\taul(\epsy)}}{\sqrt{m}}}\rfloor}.\nonumber
\end{equation}
Applying Stirling's formula and substituting $\pip_1=\frac{1}{m}$ leads to
\begin{equation}%\label{e:PAapproximation}
\Prob_\pip\{A_n\}=\exp\{-\half \frac{n\sqrt{2\taul(\epsy)}}{\sqrt{m}} \log(m)(1+o(1))\}(1+o(1)).\nonumber\end{equation}
The following approximation to the exponent the above equation follows from 
 $m=o(\frac{n^2}{\log(n)^2})$ and $m=o(n^2)$:
\[\frac{n\sqrt{2\taul(\epsy)}}{\sqrt{m}} \log(m)=\frac{n\sqrt{2\taul(\epsy)}}{\sqrt{m}}o(2\log(n))=o(\frac{n^2}{m}).\] This leads to the claim of this lemma.
\end{proof}

\end{document}